\documentclass[12pt]{amsart}

\usepackage{amsmath, amssymb, amsthm}
\usepackage{thmtools}
\usepackage{geometry}
\usepackage{graphicx}
\usepackage{enumitem}
\usepackage{color}
\usepackage{mathtools}
\usepackage{tikz}
\usepackage{tikz-cd}
\usepackage{quiver}
\usepackage{mathrsfs}
\usepackage{proof}
\usepackage{todonotes}
\usepackage{array}
\usepackage{booktabs}
\usepackage{siunitx}
\usepackage[utf8]{inputenc}
\usepackage[T1]{fontenc}
\usepackage{lmodern}
\usepackage{fancyvrb}
\fvset{commandchars=\\\{\}}
\usepackage{pmboxdraw}
\usepackage{comment}
\usetikzlibrary{positioning,arrows.meta,fit}

\tikzset{
  human/.style   = {->, very thick, color=black, shorten >=1pt},
  auto/.style    = {->, thin, dashed, color=black!60, shorten >=1pt},
  semi/.style    = {->, thick, dash pattern=on 4pt off 2pt on 1pt off 2pt, color=black!80, shorten >=1pt}
}

\usepackage{listings}
\lstnewenvironment{lean}{\lstset{language=lean}}{}
\definecolor{kwcolor}{rgb}{0.8, 0.1, 0.1}    
\definecolor{taccolor}{rgb}{0.1, 0.1, 0.8}    
\definecolor{opcolor}{rgb}{0.5, 0.1, 0.5}    
\lstdefinelanguage{lean} {
  basicstyle = {\small \ttfamily},
  columns = fullflexible,
  keepspaces = true,
  morecomment = [s][\color{green!50!black}]{@[}{]},
  morekeywords = [1]{theorem, by, fun, let, have, obtain},
  keywordstyle = [1]{\ttfamily\color{kwcolor}},
  morekeywords = [2]{use, constructor, intro, only, simp, decide, decideFin, by_contra, subsumption},
  keywordstyle = [2]{\ttfamily\color{taccolor}},
  literate =
  {:=}{{{\color{kwcolor}\raise1pt\hbox{:}\hspace{-1pt}=}}}{2}
  {=>}{{{\color{kwcolor}=>}}}{2}
  {∃}{{{\color{opcolor}\ensuremath{\exists}}}}{1}
  {∧}{{{\color{opcolor}\ensuremath{\wedge}}}}{1}
  {¬}{{{\color{opcolor}\ensuremath{\lnot}}}}{1}
  {×}{{{\color{opcolor}\ensuremath{\times}}}}{1}
  {◇}{{{\color{opcolor}\ensuremath{\diamond}}}}{1}
  {·}{{{\color{taccolor}\ensuremath{\bullet}}}}{1}
  {≠}{{{\color{opcolor}\ensuremath{\neq}}}}{1}
  {ℕ}{{\ensuremath{\mathbb{N}}}}{1}
  {⟨}{{\ensuremath{\langle}}}{1}
  {⟩}{{\ensuremath{\rangle}}}{1}
  {«}{\color{gray}{\guillemetleft}}{1}
  {»}{{\color{gray}\guillemetright}\color{black}}{1}
  ,
}
\usepackage{iftex}
\ifpdftex\else
\usepackage{newunicodechar}
\newunicodechar{∃}{\ensuremath{\exists}}
\newunicodechar{∧}{\ensuremath{\wedge}}
\newunicodechar{¬}{\ensuremath{\lnot}}
\newunicodechar{×}{\ensuremath{\times}}
\newunicodechar{◇}{\ensuremath{\diamond}}
\newunicodechar{·}{\ensuremath{\bullet}}
\newunicodechar{≠}{\ensuremath{\neq}}
\newunicodechar{ℕ}{\ensuremath{\mathbb{N}}}
\newunicodechar{⟨}{\ensuremath{\langle}}
\newunicodechar{⟩}{\ensuremath{\rangle}}
\newunicodechar{«}{\guillemetleft}
\newunicodechar{»}{\guillemetright}
\newunicodechar{–}{\textendash}
\newunicodechar{’}{\textquoteright}
\newunicodechar{“}{\textquotedblleft}
\newunicodechar{”}{\textquotedblright}
\newunicodechar{š}{\v s}
\fi

\usepackage[hidelinks]{hyperref}
\usepackage{cleveref}

\newtheorem{theorem}{Theorem}[section]
\renewcommand{\theHtheorem}{\theHsection.\the\value{theorem}}
\newtheorem{lemma}[theorem]{Lemma}
\newtheorem{proposition}[theorem]{Proposition}
\newtheorem{corollary}[theorem]{Corollary}
\newtheorem{problem}[theorem]{Problem}
\theoremstyle{definition}
\newtheorem{definition}[theorem]{Definition}
\newtheorem{example}[theorem]{Example}
\newtheorem{remark}[theorem]{Remark}

\newcommand{\N}{\mathbb{N}}
\newcommand{\Z}{\mathbb{Z}}

\newcommand{\F}{\mathbb{F}}
\newcommand{\x}{\mathrm{x}}
\newcommand{\y}{\mathrm{y}}
\newcommand{\z}{\mathrm{z}}
\newcommand{\w}{\mathrm{w}}
\newcommand{\uu}{\mathrm{u}}
\newcommand{\vv}{\mathrm{v}}
\newcommand{\op}{\diamond}
\newcommand{\formaleq}{\simeq}
\newcommand{\eps}{\varepsilon}
\newcommand{\nmodels}{\not\models}
\newcommand{\modelsfin}{\models_{\mathrm{fin}}}
\newcommand{\nmodelsfin}{\nmodels_{\mathrm{fin}}}
\newcommand{\Magma}{{\mathcal{M}}}
\newcommand{\MagmaN}{{\mathcal{N}}}
\newcommand{\Eq}[1]{\mathrm{E}#1}
\newcommand{\E}{\mathrm{E}}

\title[Equational Theories Project]{The Equational Theories Project: Advancing Collaborative Mathematical Research at Scale}
\author[Equational Theories Project contributors]{Matthew Bolan, Joachim Breitner, Jose Brox, Nicholas Carlini, Mario Carneiro, Floris van Doorn,
  Martin Dvorak, Andr\'es Goens, Aaron Hill, Harald Husum, Hern\'an Ibarra Mejia, Zoltan A. Kocsis, Bruno Le Floch, Amir Livne Bar-on, Lorenzo Luccioli, Douglas McNeil,
  Alex Meiburg, Pietro Monticone, Pace P. Nielsen, Emmanuel Osalotioman Osazuwa, Giovanni Paolini, Marco Petracci, Bernhard Reinke, David Renshaw, Marcus Rossel, Cody Roux,
  J\'er\'emy Scanvic, Shreyas Srinivas, Anand Rao Tadipatri, Terence Tao, Vlad Tsyrklevich, Fernando Vaquerizo-Villar,
  Daniel Weber, Fan Zheng}
\date{December 2025}

\begin{document}

\begin{abstract}
  We report on the \emph{Equational Theories Project} (ETP), an online collaborative pilot project
  to explore new ways to collaborate in mathematics with machine assistance. The project successfully determined all $\num{22028942}$ edges of the implication graph between the $4694$ simplest equational laws on magmas, by a combination of
  human-generated and automated proofs, all validated by the formal proof assistant language
  \emph{Lean}. As a result of this project, several new constructions of magmas satisfying specific laws were discovered, and several auxiliary questions were also addressed, such as the effect of restricting attention to finite magmas.
\end{abstract}

\begingroup
\def\uppercasenonmath#1{} 
\let\MakeUppercase\relax 
\maketitle
\endgroup

{
\setlength{\parskip}{0em}
\setcounter{tocdepth}{1}
\tableofcontents
}

\section{Introduction}

The purpose of this paper is to report on the \emph{Equational Theories Project} (ETP)\footnote{\url{https://teorth.github.io/equational_theories/}}, a pilot project launched~\cite{Tao_blog_Sep_2024} in September 2024 to explore new ways to collaboratively work on mathematical research projects using machine assistance. The project goal, in the area of universal algebra, was selected\footnote{The specific mathematical goal was inspired by \href{https://mathoverflow.net/questions/450930}{the MathOverflow question ``Is there an identity between the associative identity and the constant identity?''}, posed on 17 July 2023.} to be particularly amenable to crowdsourced and computer-assisted techniques, while still being of mathematical research interest.

The project achieved its primary goal on 14 April 2025, when the $\num{4694} \times (\num{4694}-1) = \num{22028942}$ implications between the test set of $\num{4694}$ equational laws were completely determined, with proofs or refutations formalized in \emph{Lean}.  This required coordinating the efforts of a large number of participants contributing both human-written formalizations and automatically generated proofs from various computer tools.  In this paper, we report on both the scientific outcomes of the project, as well as the organizational issues that came up with organizing a mathematical project of this scale.

\subsection{Magmas and equational laws}

In order to describe the mathematical goals of the ETP, we need some notation. A \emph{magma} $\Magma = (M,\op)$ is a set $M$ (known as the \emph{carrier}) together with a binary operation $\op \colon M \times M \to M$. An \emph{equational law} for a magma, or \emph{law} for short, is an identity involving $\op$ and some formal indeterminates, which we will typically denote using the Roman letters $\x,\y,\z,\w,\uu,\vv$, as well as the formal equality symbol $\formaleq$ in place of the equality symbol $=$ to emphasize the formal nature of the law.  If $M$ is finite, we refer to its cardinality as the \emph{size} of the magma $\Magma$.

An \emph{equational theory} is a collection of equational laws; in this paper we will primarily be concerned with theories generated by a single such law, although it is certainly of interest to explore larger theories as well.  Equational theories are one of the simplest non-trivial examples of a theory in the model-theoretic sense; they also arise in various areas of computer science, such as term rewriting systems \cite{term-rewriting}, automated theorem proving \cite{mccune-survey}, and in the Dolev--Yao model \cite{dolev} of interactive cryptographic protocols.

In the ETP, a unique number was assigned to each equational law, via a numbering system that we describe in \Cref{numbering-app}.  For instance, the \emph{commutative law} $\x \op \y \formaleq \y \op \x$ is assigned to the equation number $\Eq{43}$, while the \emph{associative law} $\x \op (\y \op \z) \formaleq (\x \op \y) \op \z$ is assigned to the equation $\Eq{4512}$.  A list of all equations referred to by number in this paper is also provided in \Cref{numbering-app}.

A magma $\Magma = (M,\op)$ satisfies a law $\E$ if the law $\E$ holds for all possible assignments of the indeterminates to elements of $M$, in which case we write $\Magma \models \E$. Thus, for instance $\Magma \models \Eq{43}$ if one has $x \op y = y \op x$ for all $x,y \in M$.  Note that the formal indeterminate symbols $\x, \y$ in $\Eq{43}$ are now replaced by concrete elements $x,y$ of the carrier $M$.

We say that a law $\E$ \emph{entails} or \emph{implies} another law ${\E'}$ if every magma that satisfies $\E$, also satisfies ${\E'}$: $(\Magma \models \E) \implies (\Magma \models {\E'})$.  We write this relation as ${\E} \models {\E'}$. We say that two laws are \emph{equivalent} if they entail each other. For instance, the constant law $\x \op \y \formaleq \z \op \w$ \eqref{eq46} can easily be seen to be equivalent to the law $\x \op \x \formaleq \y \op \z$ \eqref{eq41}.  It is clear that $\models$ is a pre-order, that is to say a partial order after one quotients by equivalence.

In this entailment pre-ordering, the maximal element is given by the trivial law $\x\formaleq\x$ \eqref{eq1}, and the minimal element is given by the singleton law $\x\formaleq \y$ \eqref{eq2}, thus $\Eq{2} \models {\E} \models \Eq{1}$ for all laws $\E$.

We also define a variant: we say that $\E$ \emph{entails} ${\E'}$ \emph{for finite magmas}, and write ${\E} \modelsfin {\E'}$, if every \emph{finite} magma that satisfies $\E$, also satisfies ${\E'}$.  Clearly, the relation ${\E} \models {\E'}$ implies ${\E} \modelsfin {\E'}$; but, as observed by Austin \cite{austin_finite}, the converse is not true in general.

The \emph{order} of an equational law is the number of occurrences of the magma operation, and can be viewed as a crude measure of complexity of the law. For instance, the commutative law $\Eq{43}$ has order $2$, while the associative law $\Eq{4512}$ has order $4$. We note some selected laws of small order that have previously appeared in the literature:
\begin{itemize}
\item The \emph{central groupoid law} $\x \formaleq (\y \op \x) \op (\x \op \z)$ \eqref{eq168} is an order-$3$ law introduced by Evans \cite{evans} and studied further by Knuth \cite{knuth} and many further authors, being closely related to central digraphs (also known as unique path property digraphs), and leading in particular to the discovery of the Knuth-Bendix algorithm \cite{knuth-bendix}; see \cite{klt} for a more recent survey.
\item \emph{Tarski's axiom} $\x \formaleq \y \op (\z \op (\x \op (\y \op \z)))$ \eqref{eq543} is an order-$4$ law that was shown by Tarski \cite{Tarski1938} to characterize the operation of subtraction in an abelian group; that is to say, a magma $\Magma = (M,\op)$ satisfies $\Eq{543}$ if and only if there is an abelian group structure on $\Magma$ for which $x \op y = x-y$ for all $x,y \in M$.
\item In a similar vein, it was shown in \cite{mendelsohn-padmanabhan} (see also \cite{meredith-prior}) that the order-$4$ law
$\x \formaleq (\y \op \z) \op (\y \op (\x \op \z))$ \eqref{eq1571} characterizes addition (or subtraction) in an abelian group of exponent $2$; it was shown in \cite{mccune_et_al} that the order-$6$ law $\x \formaleq (\y \op ((\x \op \y) \op \y)) \op (\x \op (\z \op \y))$ \eqref{eq345169} characterizes the Sheffer stroke in a boolean algebra, and it was shown in \cite{higman-neumann} that the order-$8$ law
$\x \formaleq \y \op ((((\y \op \y) \op \x) \op \z) \op (((\y \op \y) \op \y) \op \z))$ \eqref{eq42323216} characterizes division in a (not necessarily abelian) group.
\end{itemize}
Some further examples of laws characterizing well-known algebraic structures are listed in~\cite{mccune-survey}.

The Birkhoff completeness theorem \cite[Th.~3.5.14]{term-rewriting} implies that an implication ${\E} \models {\E'}$ of equational laws holds if and only if the left-hand side of ${\E'}$ can be transformed into the right-hand side by a finite number of substitution rewrites using the law $\E$. However, the problem of determining whether such an implication holds is undecidable in general \cite{mckenzie}. Even when the order is small, some implications\footnote{Another contemporaneous example of this phenomenon was the solution of the Robbins problem \cite{robbins}.} can require lengthy computer-assisted proofs; for instance, it was noted in \cite{Kisielewicz2} that the order-$4$ law $\x \formaleq (\y \op \x) \op ((\x \op \z) \op \z)$ \eqref{eq1689} was equivalent to the singleton law $\x \formaleq \y$ \eqref{eq2}, but all known proofs were found with computer assistance.\footnote{We improved such a proof to make it human-readable, see \href{https://teorth.github.io/equational_theories/blueprint/implications-chapter.html}{the blueprint of the ETP}.}  Furthermore, for the finite magma implication relation ${\E} \modelsfin {\E'}$, no analogue of the Birkhoff completeness theorem is available.

\subsection{The Equational Theories Project}

As noted in \Cref{numbering-app}, there are $\num{4694}$ equational laws of order at most $4$. The primary mathematical goal of the ETP was to completely determine the \emph{implication graph} for these laws, in which there is a directed edge from $\E$ to ${\E'}$ precisely when ${\E} \models {\E'}$. As the project progressed, an additional goal was added to determine the slightly larger \emph{finite implication graph}, in which there is a directed edge from $\E$ to ${\E'}$ precisely when ${\E} \modelsfin {\E'}$.

Such systematic determinations of implication graphs have been seen previously in the literature; for instance, in \cite{phillips-vojtechovsky}, the relations between $60$ identities of Bol--Moufang type were established, and in the blog post \cite[\S 17]{Wolfram_2022}, some initial steps towards generating this graph for the first hundred or so laws on our list were performed. However, to our knowledge, the ETP is the first project to study such implications at the scale of thousands of laws.

The ETP requires the determination of the truth or falsity of $\num{4694}^2 = \num{22033636}$ implications (for both arbitrary magmas and finite magmas), or $\num{4694} \times (\num{4694}-1) = \num{22028942}$ if the reflexive implications ${\E} \models \E$ are removed; while one can use properties such as the transitivity of entailment to reduce the work somewhat, this is clearly a task that requires significant automation. It was also a project highly amenable to crowdsourcing, in which different participants could work on developing different techniques, each of which could be used to fill out a different part of the implication graph. In this respect, the project could be compared with a Polymath project \cite{Gowers2009}, which used online forums such as blogs and wikis to openly collaborate on a mathematical research problem. However, the Polymath model required human moderators to review and integrate the contributions of the participants, which clearly would not scale to the ETP which required the verification of over twenty million mathematical statements. Instead, the ETP was centered around a GitHub repository in which the formal mathematical contributions had to be entered in the proof assistant language \emph{Lean}, where they could be automatically verified. In this respect, the ETP was more similar to the recently concluded Busy Beaver Challenge~\cite{bbchallenge_bb5}, which was a similarly crowdsourced project that computed the fifth Busy Beaver number $BB(5)$ to be $\num{47176870}$ through an analysis of about $180$ million Turing machines, with the halting analysis being verified in a variety of computer languages, with the final formal proof written in the proof assistant language \emph{Coq} \cite{the_coq_development_team_2024_14542673, bbchallenge_bb5}. One of the aims of the ETP was to explore potential workflows for such collaborative, formally verified mathematical research projects that could serve as a model for future projects of this nature.

Secondary aims of the ETP included the possibility of discovering unusually interesting equational laws, or new experimental observations about such laws, that had not previously been noticed in the literature; and to develop benchmarks to assess the performance of automated theorem provers and other AI tools.

\subsection{Outcomes}

The ETP achieved almost all of its primary objectives, with all of the $\num{22033636}$ implications ${\E} \models {\E'}$ and non-implications $\E \not \models {\E'}$ formalized in the proof assistant language \emph{Lean}, and can be found on the ETP GitHub repository.  See \Cref{fig:854}, \Cref{fig:1729} and \Cref{fig:longchain} for some small fragments of the implication graphs produced.
The $\num{4694}$ laws were organized into $\num{1415}$ equivalence classes, with by far the largest class being the class of $\num{1496}$ equations equivalent to the singleton law $\Eq{2}$.

For the finite implication graph ${\E} \modelsfin {\E'}$, we could similarly formalize all but two implications.  Specifically, we were unable to obtain either a human-readable or formalized proof or disproof of the implication $\Eq{677} \modelsfin \Eq{255}$ (or its equivalent dual $\Eq{2910} \modelsfin \Eq{47}$), despite extensive efforts from the participants of the project; we tentatively conjecture this implication to be false (i.e., that there exists a finite magma satisfying $\Eq{677}$ but not $\Eq{255}$), but the refutation appears to be ``immune'' to most of the techniques that we developed for the project.  (We were however able to establish that the corresponding implication $\Eq{677} \models \Eq{255}$ for arbitrary magmas was false, using the greedy construction discussed in \Cref{greedy-sec}.)

\begin{figure}
\centering
\includegraphics[width=0.85\textwidth,trim=5 0 5 0,clip]{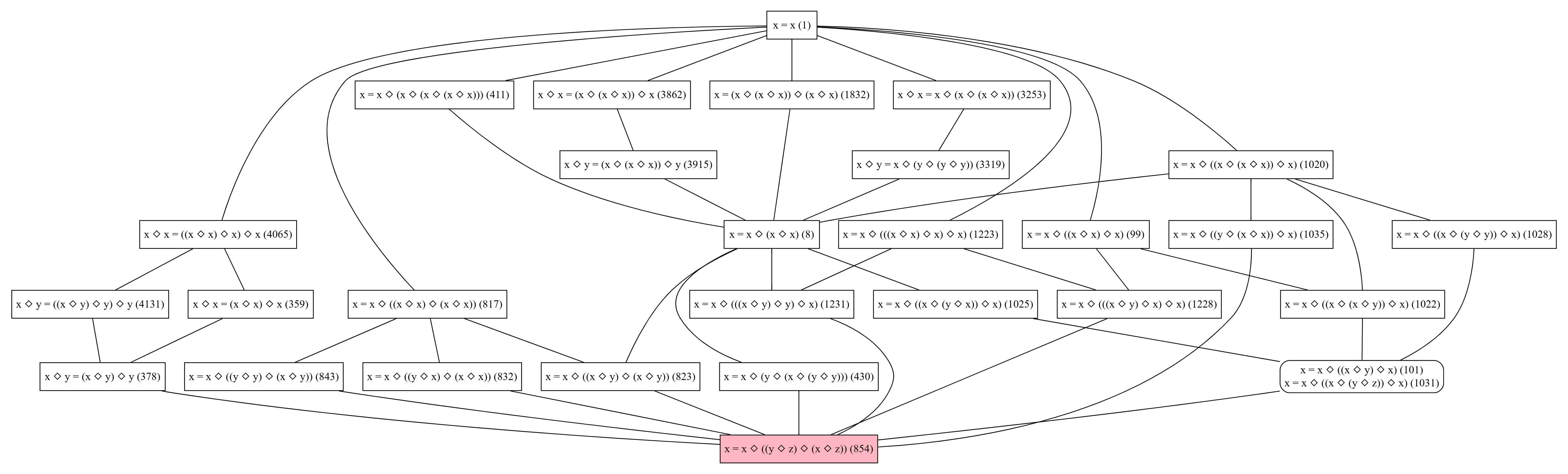}
\caption{A Hasse diagram of all the equational laws implied by $\Eq{854}$ (for unrestricted magmas).  An edge in this diagram indicates that the lower equation implies the higher one. Rounded rectangles indicate groups of equivalent laws.  This graph was produced by the visualization tool \emph{Graphiti}, which was developed for this project.}
\label{fig:854}
\end{figure}

\begin{figure}
    \centering
    \includegraphics[width=0.42\textwidth]{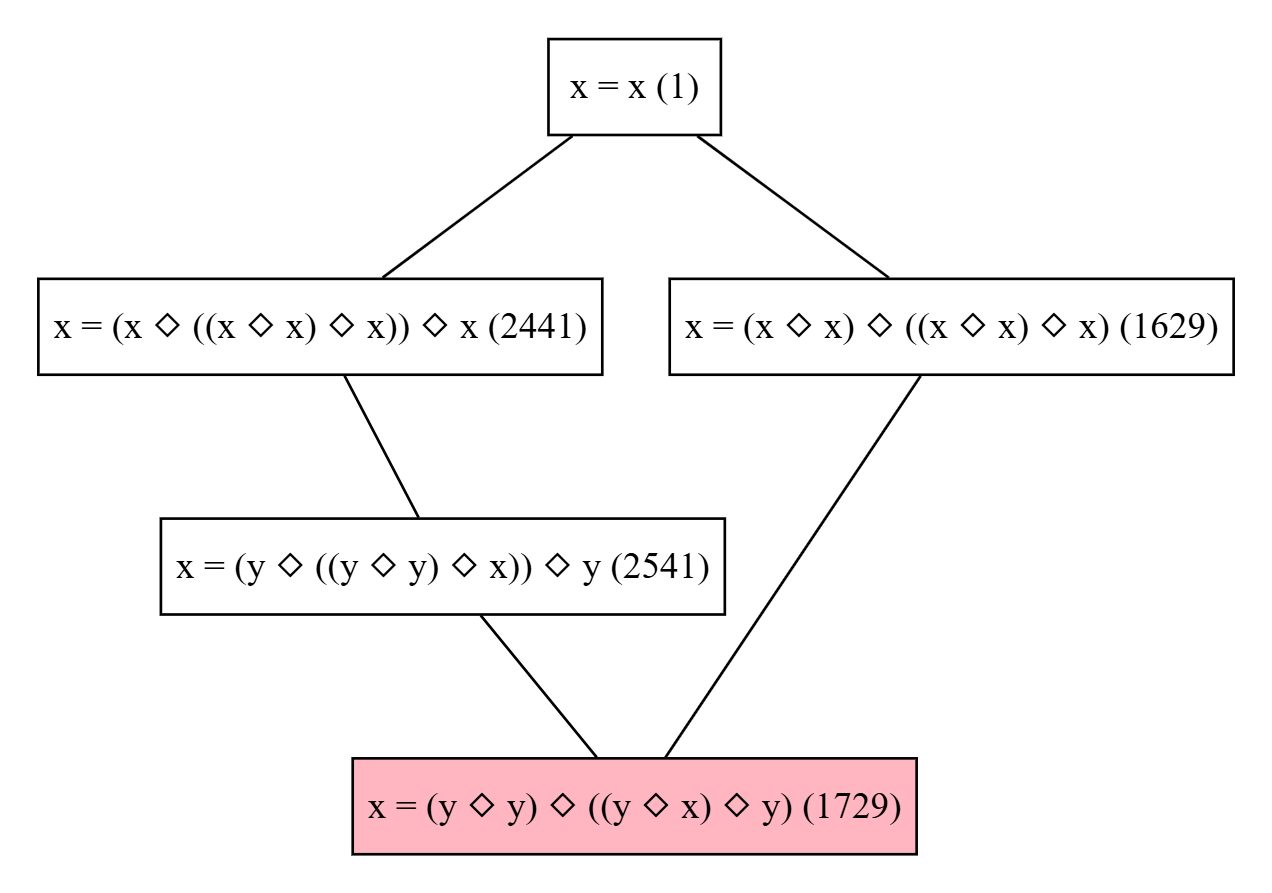}
    \includegraphics[width=0.42\textwidth]{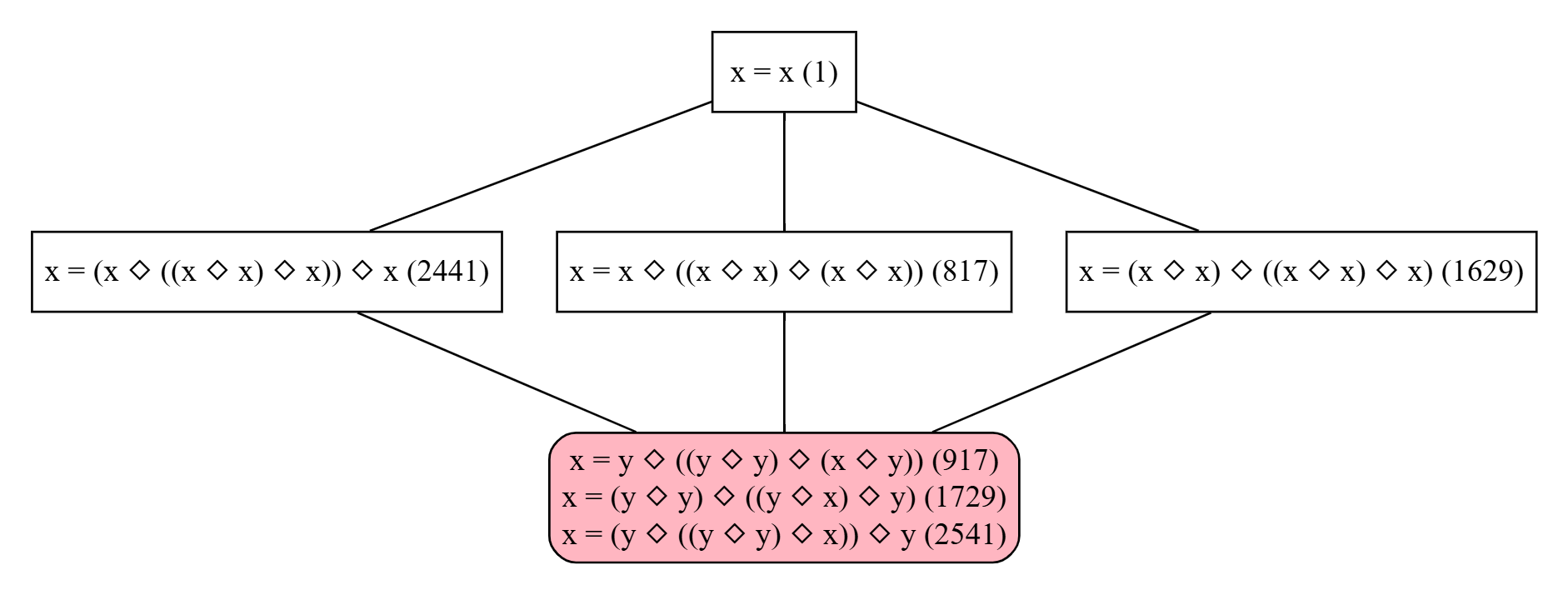}
    \caption{A Hasse diagram of all the equational laws implied by $\Eq{1729}$, both for unrestricted magmas (left) and finite magmas (right). Note the slightly larger number of implications in the latter.}
    \label{fig:1729}
\end{figure}

\begin{figure}
    \centering
    \resizebox{\textwidth}{!}{%
      \begin{tikzpicture}
        \node(1)[draw] at (0,15) {$\x \formaleq \x$ ($\Eq{1}$)};
        \node(3253)[draw] at (4,14) {$\x \op \x \formaleq \x \op (\x \op (\x \op \x))$ ($\Eq{3253}$)};
        \node(3456)[draw] at (-4,14) {$\x \op \x \formaleq \x \op ((\x \op \x) \op \x)$ ($\Eq{3456}$)};
        \node(3319)[draw] at (7,13) {$\x \op \y \formaleq \x \op (\y \op (\y \op \y))$ ($\Eq{3319}$)};
        \node(307)[draw] at (1,13) {$\x \op \x \formaleq \x \op (\x \op \x)$ ($\Eq{307}$)};
        \node(3522)[draw] at (-5,13) {$\x \op \y \formaleq \x \op ((\y \op \y) \op \y)$ ($\Eq{3522}$)};
        \node(326)[draw] at (0,12) {$\x \op \y \formaleq \x \op (\y \op \y)$ ($\Eq{326}$)};
        \node(3715)[draw] at (0,11) {$\x \op \y \formaleq (\x \op \x) \op (\y \op \y)$ ($\Eq{3715}$)};
        \node(3)[draw] at (0,10) {$\x \formaleq \x \op \x$ ($\Eq{3}$)};
        \draw[dashed](3)--(-10,10);
        \draw[dashed](3)--(10,10);
        \node(412)[draw] at (0,9) {$\x \formaleq \x \op (\x \op (\x \op (\x \op \y)))$ ($\Eq{412}$)};
        \node(48)[draw] at (0,8) {$\x \formaleq \x \op (\x \op (\x \op \y))$ ($\Eq{48}$)};
        \node(618)[draw] at (3,7) {$\x \formaleq \x \op (\x \op ((\x \op \y) \op \z))$ ($\Eq{618}$)};
        \node(418)[draw] at (-3,7) {$\x \formaleq \x \op (\x \op (\y \op (\x \op \z)))$ ($\Eq{418}$)};
        \node(9)[draw] at (0,6) {$\x \formaleq \x \op (\x \op \y)$ ($\Eq{9}$)};
        \node(824)[draw] at (-9.7,5) {$\x \formaleq \x \op ((\x \op \y) \op (\x \op \z))$ ($\Eq{824}$)};
        \node(829)[draw] at (-3.1,5) {$\x \formaleq \x \op ((\x \op \y) \op (\z \op \y))$ ($\Eq{829}$)};
        \node(1027)[draw] at (3.1,5) {$\x \formaleq \x \op ((\x \op (\y \op \x)) \op \z)$ ($\Eq{1027}$)};
        \node(1033)[draw] at (9.7,5) {$\x \formaleq \x \op ((\x \op (\y \op \z)) \op \z)$ ($\Eq{1033}$)};
        \node(103)[draw] at (0,4) {$\x \formaleq \x \op ((\x \op \y) \op \z)$ ($\Eq{103}$)};
        \node(1235)[draw] at (-9,4) {$\x \formaleq \x \op (((\x \op \y) \op \z) \op \y)$ ($\Eq{1235}$)};
        \node(1237)[draw] at (0,3) {$\x \formaleq \x \op (((\x \op \y) \op \z) \op \w)$ ($\Eq{1237}$)};
        \node(1037)[draw] at (0,2) {$\x \formaleq \x \op ((\y \op (\x \op \x)) \op \z)$ ($\Eq{1037}$)};
        \node(4)[draw] at (0,1) {$\x \formaleq \x \op \y$ ($\Eq{4}$)};
        \node(2)[draw] at (0,0) {$\x \formaleq \y$ ($\Eq{2}$)};
        \draw(2)--(4);
        \draw(4)--(1037);
        \draw(1037)--(1237);
        \draw(1237)--(1235);
        \draw(1235)--(824);
        \draw(1237)--(103);
        \draw(103)--(1027);
        \draw(103)--(1033.south west);
        \draw(103)--(824.south east);
        \draw(103)--(829);
        \draw(1027)--(9);
        \draw(1033.north west)--(9);
        \draw(829)--(9);
        \draw(824.north east)--(9);
        \draw(9)--(418);
        \draw(9)--(618);
        \draw(418)--(48);
        \draw(618)--(48);
        \draw(48)--(412);
        \draw(412)--(3);
        \draw(3)--(3715);
        \draw(3715)--(326);
        \draw(326)--(3522);
        \draw(3522)--(3456);
        \draw(3456)--(1);
        \draw(326)--(307);
        \draw(307)--(3253);
        \draw(3253)--(1);
        \draw(326)--(3319);
        \draw(3319)--(3253);
      \end{tikzpicture}%
    }
    \caption{Longest chains of implications (length $15$) between inequivalent laws in the implication graph.  The parts above/below law $\Eq{3}$ can be independently dualized. }
    \label{fig:longchain}
\end{figure}
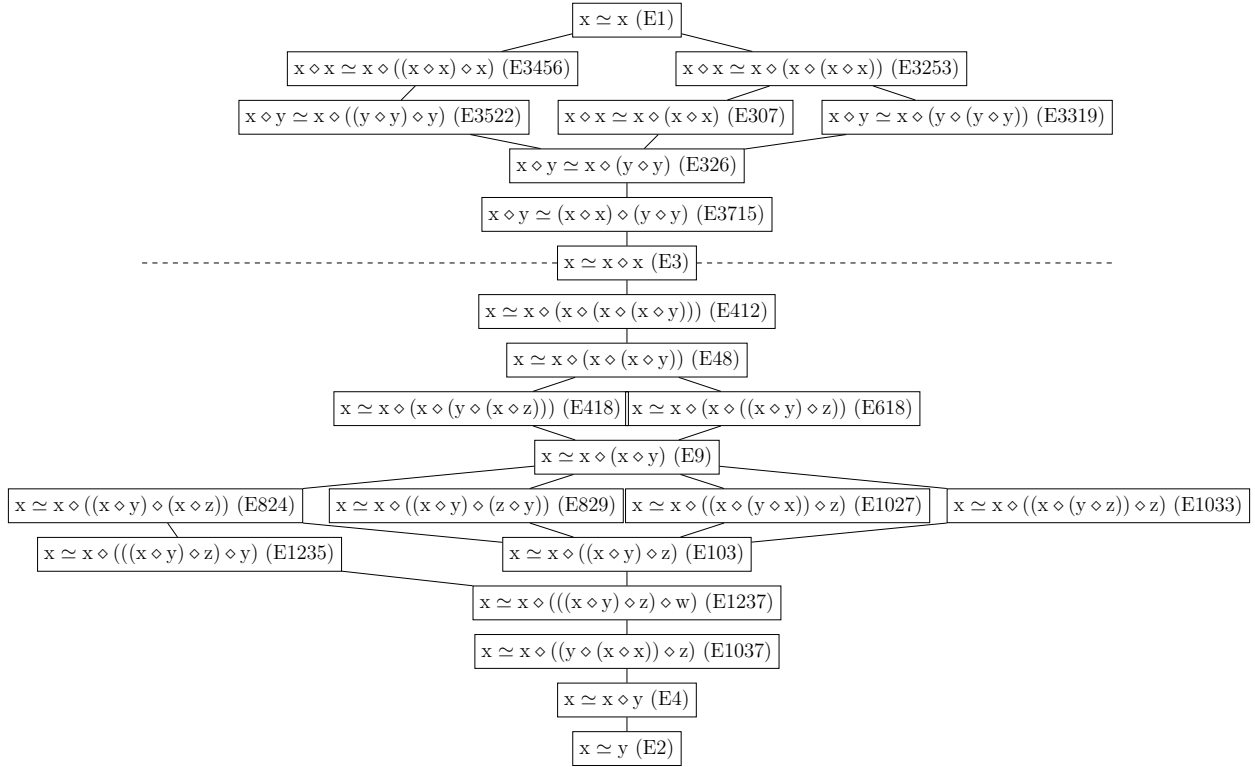

Of the $\num{22033636}$ possible implications ${\E} \models {\E'}$, $\num{8178279}$ (or $37.12\%$) would end up being true; for an additional set of either $820$ or $822$ pairs $\E,{\E'}$, the weaker implication ${\E} \modelsfin {\E'}$ also held. To establish such positive implications ${\E} \models {\E'}$ or ${\E} \modelsfin {\E'}$, the main techniques used were as follows:

\begin{itemize}
    \item A very small number of positive implications were established and \textbf{formalized by hand}, mostly through direct rewriting of the laws; but this approach would not scale to the full project.
    \item \textbf{Simple rewriting rules}, for instance based on the observation that any law of the form $\x \formaleq f(\y,\z,\dots)$ was necessarily equivalent to the singleton law $\Eq{2}$, could already reduce the size of potential equivalence classes by a significant fraction. We discuss this method in \Cref{rewrite-sec}.
    \item The preorder axioms for $\models$, as well as the ``duality'' symmetry of the preorder with respect to replacing a magma operation $x \op y$ with its opposite $x \op^{\mathrm{op}} y \coloneqq y \op x$, can be used to significantly cut down on the number of implications that need to be proven explicitly; ultimately, only $\num{10657}$ ($0.13\%$) of the positive implications needed a direct proof.
    \item To obtain additional implications for finite magmas, heavy reliance was made on the fact that for functions $f \colon M \to M$ on a finite set $M$, surjectivity was equivalent to injectivity.  Some more sophisticated variants of this idea can lead to additional implications; see \Cref{finite-sec}.
    \item \textbf{Automated Theorem Provers} (ATP) could be deployed at extremely fast speeds to establish a complete generating set of positive implications; see \Cref{automated-sec}.
\end{itemize}

More challenging were the $\num{13855357}$ ($62.88\%$) implications that were false, $\E \nmodels {\E'}$, and particularly the slightly smaller set of $\num{13854535}$ or $\num{13854537}$ implications that were false even for finite magmas, $\E \nmodelsfin {\E'}$. Here, the range of techniques needed to refute such implications were quite varied, and may be of independent interest:
\begin{itemize}
        \item \textbf{Small finite magmas}, which can be described explicitly by multiplication tables, could be tested by brute force computations to provide a large number of finite counterexamples to implications, or by ATP-assisted methods. See \Cref{finite-sec}.
        \item \textbf{Linear models}, in which the magma operation took the form $x \op y = ax + by$ for some (commuting or noncommuting) coefficients $a,b$, allowed for another large class of counterexamples to implications, which could be automatically scanned for, either by brute force or by Gr\"obner basis type calculations; many of these examples could also be made finite. See \Cref{linear-sec}.
        \item \textbf{Translation invariant models}, in which the magma operation took the form $x \op y = x + f(y-x)$ on an additive group, or $x \op y = x f(x^{-1} y)$ on a noncommutative group, reduce matters to analyzing certain functional equations; see \Cref{translation-sec}.
        \item To each equation $\E$ one can associate a ``\textbf{twisting semigroup}'' $S_{\E}$.  If $S_{\E}$ is larger than $S_{{\E'}}$, then this can often be used to disprove the implication ${\E} \models {\E'}$; see \Cref{twisting-sec}.
        \item \textbf{Greedy methods}, in which either the multiplication table $(x,y) \mapsto x \op y$ or the function $f$ determining a translation-invariant model are iteratively constructed by a greedy algorithm subject to a well-chosen ruleset, were effective in resolving many implications not easily disposed of by preceding methods. See \Cref{greedy-sec}.
        \item Starting with a simple base magma $\Magma$ satisfying both $\E$ and ${\E'}$, and either \textbf{enlarging} it to a larger magma $\Magma'$ containing $\Magma$ as a submagma, \textbf{extending} it to a magma $\MagmaN$ with a projection homomorphism $\pi: \MagmaN \to \Magma$, or \emph{modifying the multiplication table} on a small number of values, also proved effective when combined with greedy methods or with a ``\textbf{magma cohomology}'' construction. See \Cref{modify-base}.
        \item \textbf{Syntactic methods}, such as observing a ``matching invariant'' of the law $\E$ that was not shared by the law ${\E'}$, could be used to obtain some refutations.  For instance, if both sides of $\E$ had the same order, but both sides of ${\E'}$ did not, this could be used to syntactically refute ${\E} \models {\E'}$.  Similarly, if the law $\E$ was confluent, enjoyed a complete rewriting system, or otherwise permitted some understanding of the free magma associated to that law, one could decide the assertions ${\E} \models {\E'}$ for all possible laws ${\E'}$, or at least a significant fraction of such laws.  We discuss these methods, and the extent to which they can be automated, in \Cref{syntactic-sec}.
        \item Some \textbf{ad hoc models} based on existing mathematical objects, such as infinite trees, rings of polynomials, or ``Kisielewicz models'' utilizing the prime factorization of the natural numbers, could also handle some otherwise difficult cases.  In some cases, the magma law induced some relevant and familiar structures, such as a directed graph or a partial order, which also helped guide counterexample constructions. We will not detail these diverse examples here, but refer the reader to the ETP blueprint for more discussion.
        \item \textbf{Automated theorem provers} were helpful in identifying which simplifying axioms could be added to the magma without jeopardizing the ability to refute the desired implication ${\E} \models {\E'}$ or ${\E} \modelsfin {\E'}$.
\end{itemize}

While the vast majority of negative implications could be quickly resolved by one of the above techniques, either with human input or in a completely automated fashion, there were perhaps two dozen such negative results that required quite delicate and \emph{sui generis} constructions.  The hardest such implication, $\Eq{1729} \nmodels \Eq{817}$, took several months to establish and then formalize (using a combination of many of the above constructions), with the final proof in \emph{Lean} requiring just over $\num{4000}$ dedicated lines of code from multiple contributors.

In the course of completing the implication graph, some interesting new algebraic structures were discovered.  One such example concerns the magmas satisfying $\Eq{1485}$, which we refer to as \emph{weak central groupoids} as they contain the central groupoids (satisfying $\Eq{168}$) as a subclass.  In \cite{knuth} it was observed that all finite central groupoids have order equal to a perfect square $n^2$; empirically, we have found that finite weak central groupoids always have order $n^2$ or $2n^2$, although we have no rigorous proof of this claim; they also have a graph-theoretic interpretation analogous to the interpretation of central groupoids as digraphs with the unique path property.  For these and other observations we refer the reader to \href{https://teorth.github.io/equational_theories/blueprint/weak-central-groupoids-chapter.html}{the blueprint of the ETP}.

The objective of using the data from the ETP to establish well-calibrated benchmarks to evaluate ATPs remains an interesting open problem; the participants of this project did not have the required expertise to develop and test such benchmarks to the standards expected in the area.  However, in \Cref{automated-sec} we present a more informal ``field report'' of our experiences using ATPs in the project, in the hope that this will provide some useful guidance to other researchers seeking to incorporate ATPs into their own research.

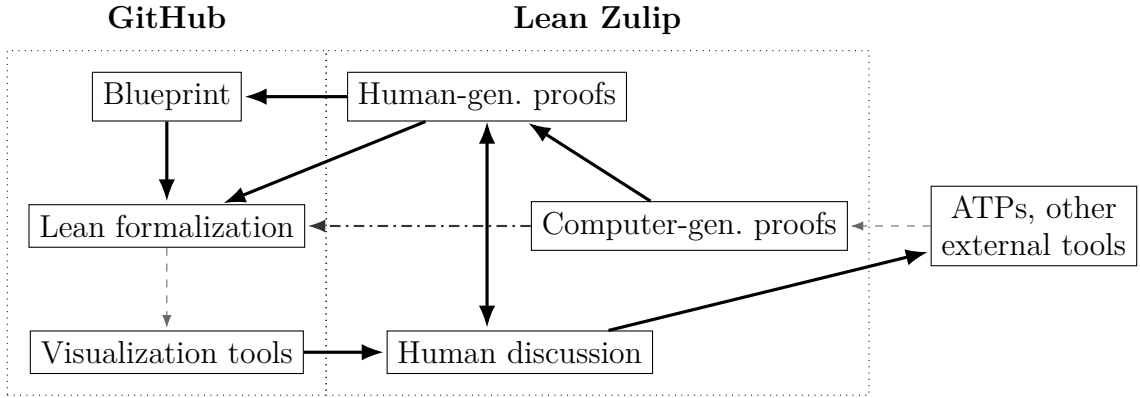
\begin{figure}
\centering
\begin{tikzpicture}[
    >=Latex,
    node distance=1.1cm,
    every node/.style={draw, rectangle, align=center},
    every label/.style={draw=none, font=\bfseries\strut}
  ]

  \node (blueprint)   {Blueprint};
  \node (formal)     [below=of blueprint] {Lean formalization};
  \node (viz)        [below=of formal]    {Visualization tools};

  \node (hproofs)    [right=1.4cm of blueprint]       {Human-gen.\ proofs};
  \node (cproofs)    [right=3cm of formal]         {Computer-gen.\ proofs};
  \node (hdisc)      [right=of viz]         {Human discussion};

  \node (atp)        [right=of cproofs]         {ATPs, other\\external tools};

  \node[draw,dotted,inner sep=8pt,fit=(blueprint)(formal)(viz), label=above:GitHub]
    (ghbox) {};
  \node[draw,dotted,inner sep=8pt,fit=(hproofs)(cproofs)(hdisc), label=above:Lean Zulip]
    (zulipbox) {};

  \draw[human,->] (blueprint)  -- (formal);
  \draw[auto,->] (formal)     -- (viz);
  \draw[human,->] (viz)        -- (hdisc);
  \draw[human,->] (hdisc)      -- (atp);
  \draw[auto,->] (atp)        -- (cproofs);
  \draw[human,->] (cproofs)    -- (hproofs);
  \draw[semi,->] (cproofs)    -- (formal);
  \draw[human,<->] (hproofs.south)   -- (hproofs.south|-hdisc.north);
  \draw[human,->] (hproofs)  -- (blueprint);
  \draw[human,->] (hproofs)  -- (formal);

\end{tikzpicture}
\caption{Some of the main dynamics in which proofs were generated, discussed within the Lean Zulip channel and then formalized in the GitHub repository.  Boldface arrows indicate human activities, such as proposing an automated attack on outstanding implications, converting a computer-generated proof into a human-readable format, formalizing a human readable proof directly, or first creating a more precise blueprint for other collaborators to work on.  Dashed arrows indicate fully automated processes, while the partly dashed line indicates a semi-automated process requiring human supervision. }
\label{fig:flow}
\end{figure}

\subsection{Further directions}

While the primary objective of the ETP was being completed, some additional related results were generated as spinoffs.  Specifically:
\begin{itemize}
\item In the blueprint on the ETP web site, we report some partial progress on classifying which of the $\num{57882}$ distinct laws of order $5$ are equivalent to the singleton law $\Eq{2}$, either with or without the requirement that the magma be finite.
\item In \Cref{spectrum-sec} we report on the determination of laws with full spectrum, i.e., with magmas satisfying them of all finite sizes.
\item In \Cref{higman-neumann} we report on classifying the laws of order $8$ that are equivalent to the Higman-Neumann law $\Eq{42323216}$.
\end{itemize}

We refer the reader to the ETP blueprint for a more in depth discussion of several of the topics covered in this report, including a detailed treatment of particularly difficult equations and complex magma constructions, further treatment of term rewriting theory and equational law invariants, a list of ``hard'' implications that required particular effort to resolve in our project, and some implementation details of how tools such as automated theorem provers were utilized.

\section{Notation and mathematical foundations}\label{notation-sec}

If $\Magma = (M,\op)$ is a magma, we define the left and right multiplication operators $L_a, R_a \colon M \to M$ for $a \in M$ by the formula
\begin{equation}\label{left-right}
    L_x y = R_y x \coloneqq x \op y.
\end{equation}
We also define the squaring operator $S \colon M \to M$ by
\begin{equation}\label{square-def}
    Sx \coloneqq x \op x = L_x x = R_x x.
\end{equation}

A \emph{homomorphism} $f \colon \Magma \to \Magma'$ between two magmas $\Magma = (M,\op)$, $\Magma' = (M',\op')$ is a function $f \colon M \to M'$ such that $f(x \op y) = f(x) \op' f(y)$ for all $x,y \in M$.  An \emph{isomorphism} is a homomorphism that is invertible (which implies that the inverse is also a homomorphism).  An \emph{endomorphism} is a homomorphism from a magma to itself.

If $X$ is an alphabet, we let $\Magma_X = (M_X, \diamond)$ denote the free magma generated by $X$, thus an element of $\Magma_X$ is either a letter in $X$, or of the form\footnote{Strictly speaking, one should use parentheses and write $(w_1 \op w_2)$ to avoid ambiguity, but to reduce clutter we shall abuse notation by omitting parentheses when no ambiguity is caused by doing so.} $w_1 \op w_2$ with $w_1,w_2 \in M_X$.  Every function $f \colon X \to M$ into a magma $\Magma = (M,\op)$ extends to a unique homomorphism $\varphi_f \colon \Magma_X \to \Magma$.  Formally, an equational law with some indeterminates in $X$ can be written as $w_1 \formaleq w_2$ for some $w_1, w_2 \in M_X$; a magma $\Magma = (M,\op)$ then satisfies this law if and only if $\varphi_f(w_1) = \varphi_f(w_2)$ for all $f \colon X \to M$.  We also define the order of a word $w \in M_X$ to be the number of occurrences of $\op$ in the word, thus letters in $X$ are of order $0$, and the order of $w_1 \op w_2$ is the sum of the orders of $w_1, w_2$, plus one.

A \emph{theory} is a collection $\Gamma$ of equational laws; we say that a magma $\Magma$ \emph{satisfies} a theory, and write $\Magma \models \Gamma$, if every law in $\Gamma$ is satisfied by $\Magma$.  If $\E$ is an equational law, we write $\Gamma \models \E$ if every magma that satisfies $\Gamma$ also satisfies $\E$.  A \emph{free magma} $\Magma_{X,\Gamma} = (M_{X,\Gamma},\diamond)$ for such a theory $\Gamma$ and an alphabet $X$ is a magma satisfying $\Gamma$ together with a map $\iota_{X,\Gamma} \colon X \to M_{X,\Gamma}$ which is universal in the sense that every function $f \colon X \to \Magma$ to a magma $\Magma$ satisfying $\Gamma$ uniquely determines a homomorphism $\varphi_{f,\Gamma} \colon \Magma_{X,\Gamma} \to \Magma$ such that $\varphi_{f,\Gamma} \circ \iota_{X,\Gamma} = f$.  This magma is unique up to isomorphism; a canonical way to construct it is as the quotient $\Magma_X/\sim_\Gamma$ of the free magma $\Magma_X$ by the equivalence relation $\sim_\Gamma$ given by declaring $w \sim_\Gamma w'$ if $\Gamma \models w \formaleq w'$ \cite[Theorem 3.5.6]{term-rewriting}.  If $\Gamma = \{\E\}$ consists of a single law $\E$, we write $\Magma_{X,\E}$, $\sim_\E$, $\varphi_{f,\E}$ for $\Magma_{X,\{\E\}}$, $\sim_{\{\E\}}$, $\varphi_{f,\{\E\}}$ respectively.

In general, the free magma $\Magma_{X,\Gamma}$ is difficult to describe in a tractable form, but for some theories, one has a simple description.  We give two simple examples here:

\begin{example}[Commutative and associative free magma]\label{semi-group} The free magma $\Magma_{X,\{\Eq{43}, \Eq{4512}\}}$ for the commutative law $\Eq{43}$ and the associative law $\Eq{4512}$ is the free abelian semigroup generated by $X$ (with $\iota_{X,\{\Eq{43},\Eq{4512}\}}$ the obvious embedding map).
\end{example}

\begin{example}[Left-absorptive free magma]\label{left-absorb}
The free magma $\Magma_{X,\{\Eq{4}\}}$ for the left-absorptive law $\Eq{4}$ is the magma with carrier $X$ and operation $x \op y = x$ (with $\iota_{X,\Eq{4}}$ the identity).
\end{example}

Every magma $\Magma$ has an opposite $\Magma^{\mathrm{op}}$, which has the same carrier but the opposite operation $x \op^{\mathrm{op}} y \coloneqq y \op x$.  A magma $\Magma$ satisfies an equational law $\E$ if and only if its opposite $\Magma^{\mathrm{op}}$ satisfies the dual law $\E^*$, defined by reversing all the operations.  For instance, the dual of
$\x \op \y \formaleq \x \op (\y \op \z)$ \eqref{eq327} is $\y \op \x \formaleq (\z \op \y) \op \x$, which in our numbering system we rewrite in normal form as $\x \op \y \formaleq (\z \op \x) \op \y$ \eqref{eq395}.

We then see that the implication graph has a duality symmetry: given two equational laws $\E_1,\E_2$, we have $\E_1 \models \E_2$ if and only if $\E_1^* \models \E_2^*$.

\section{Formal foundations}

All proofs in the ETP were ultimately formalized in the proof assistant language \emph{Lean}, though in many cases the proofs were first written in an informal human document, which was then incorporated into the human-readable \emph{blueprint}~\cite{leanblueprint} that accompanied the formalization.  Many of the computer-assisted proofs were also first generated as computer output from a source other than \emph{Lean}, such as an ATP, and later converted to a \emph{Lean} proof by a separate program custom-written for this task.

The project relied on \emph{Lean}'s extensive \emph{Mathlib} library, for instance to provide support for algebraic concepts such as the free group that arose in some of the more difficult constructions.  Additional extensions to \emph{Lean}, such as \emph{duper} or \emph{egg}, were employed by some participants in external forks of the repository, but we did not incorporate them into the master repository to simplify the version control process.  As a consequence, some manual translation of proofs produced using such extensions to a proof that avoided such extensions were needed at various stages of the project.

The concept of a magma could be modeled by existing \emph{Mathlib} classes such as \texttt{Mul}; however we chose early in the project to define a custom magma class \texttt{Magma} instead, as for some magma constructions the magma operation (which we denoted $\op$) was distinct from an existing multiplication structure $*$ on the same carrier.  Most components of the \emph{Lean} codebase were placed in namespaces to avoid collisions with each other, and with \emph{Mathlib}.

Equational laws in the project were implemented both syntactically\,---\,as a structure \texttt{LawX} containing two words in a free group\,---\,as well as semantically, as a predicate \texttt{EquationX} that could be applied to a magma. Here \texttt{X} is the number assigned to the law. The semantic formulation (\texttt{EquationX}) was more convenient for proving or refuting specific implications, while the syntactic formulation (\texttt{LawX}) was preferred for implementing metatheorems, such as the use of duality between laws. \emph{Lean}'s metaprogramming features proved to be vital to relate the two representations. A custom command, \texttt{equation}, was created for specifying equational laws. Elaborating the \texttt{equation} command generated both \texttt{EquationX} and \texttt{LawX} definitions from this description, as well as theorems relating them to each other. A similar construction was used to generate dual laws, where the dual law was given explicitly for simplicity.

To facilitate the automatic generation of an implication graph from the \emph{Lean} codebase, a custom \texttt{@{[}equational\_result{]}} tag was formed to attach to propositions in \emph{Lean} to indicate that they were proving or refuting one or more implications; see \Cref{fig:impl}.  A \texttt{conjecture} keyword was also created for implications or refutations which we wished to identify as having an informal proof that had yet to be formalized in \emph{Lean}.

\begin{figure}
\centering
\begin{lean}
@[equational_result]
theorem _root_.Equation1437_not_implies_Equation4269 :
    ∃ (G : Type) (_ : Magma G), Equation1437 G ∧ ¬ Equation4269 G := by
  use ℕ × Fin 3, ⟨op⟩
  constructor
  · intro x y z
    simp [op, add_assoc]
  · simp only [not_forall, op]
    use (0, 0), (2, 0)
    decide
\end{lean}
\caption{A sample proof of a formalized implication, in this case that $\Eq{1437} \nmodels \Eq{4269}$.}
\label{fig:impl}
\end{figure}

A single construction of a magma could satisfy multiple laws $\E_1,\E_2,\dots$ and not satisfy others $\E'_1, \E'_2, \dots$, leading to a large number of refutations of the form $\E_i \nmodels \E'_j$.  A custom \texttt{Facts} command was designed to organize such information efficiently; see \Cref{fig:facts}.

\begin{figure}
\centering
\begin{lean}
@[equational_result]
theorem «Facts from All4x4Tables [[1,2,3,4,5,0],[4,1,2,5,0,3],[3,0,5,2,1,4],
[0,5,4,3,2,1],[5,4,1,0,3,2],[2,3,0,1,4,5]]» :
  ∃ (G : Type) (_ : Magma G) (_ : Finite G), Facts G [1316, 2863] [411, 680,
  817, 1020, 1426, 2035, 2441, 2644, 2853, 2855, 2865, 2872, 2947, 3050,
  3253, 3456, 4270, 4283, 4290, 4380, 4598, 4605, 4656] :=
⟨Fin 6, «All4x4Tables [[1,2,3,4,5,0],[4,1,2,5,0,3],[3,0,5,2,1,4],[0,5,4,3,2,1],
[5,4,1,0,3,2],[2,3,0,1,4,5]]», Finite.of_fintype _, by decideFin!⟩
\end{lean}
\caption{A computer generated \texttt{Facts} theorem, using an explicit finite magma of order $6$ to refute several implications at once.}
\label{fig:facts}
\end{figure}

As an additional precaution against ``exploit''-based proofs (such as those that might be contributed by an AI tool) \emph{lean4checker} was used to ensure that no axioms were used in \emph{Lean} outside of a small trusted set.  In particular, \emph{Lean} tactics such as \texttt{native\_decide} that relied on external tools were not permitted into the codebase.

Explicitly formalizing all $\num{22028942}$ implications as theorems would lead to an infeasible compilation time in  \emph{Lean}.  Instead, a reduced generating set of $\num{10657}$ positive implications and $\num{586925}$ negative implications were formalized, with the latter in turn mostly organized into a smaller number of \texttt{Facts} theorems as discussed above.  The extension of these results to the rest of the implication graph via transitivity and duality is currently done by programs external to \emph{Lean}, although in principle one could create an ``end-to-end theorem'' which completely establishes the implication graph within \emph{Lean}.

Some lemmas generated in the project were suitable for upstreaming back to \emph{Mathlib}, as well as several technical improvements to the \emph{LeanBlueprint} software.

\section{Project management}\label{project-sec}

This project is, among other things, an experiment on how to organise large scale collaborations for mathematical work. In this section, we describe several aspects of the organisation of the collaborative effort.

\subsection{Problems of scale in mathematical collaboration}
In order to understand the scaling issues that can arise in large scale collaborations, it helps to revisit the mechanics of traditional mathematical collaborations and their limitations. While every collaboration is unique, there are some general patterns. A small number of contributors, usually under ten, who may know each other, join forces to tackle some class of problems. Typically the collaborators are almost all academics who share substantial amounts of common knowledge. They discuss the problem at hand together, typically with some shared written medium such as a whiteboard. After several rounds of discussion and refinement, different members of the collaboration come up with different pieces of a solution. These pieces are then put together via discussion and merging of write-ups over several iterations. Once the collaborators are reasonably confident about the correctness of their work, including theorem statements and proofs, they submit the paper for peer review. Thus the correctness of mathematical research relies on this basic cycle of discuss, solve, write, cross-check, and revise, followed by peer review. Ultimately the authors take responsibility for the contents of their research article. This joint responsibility for authorship is formally enshrined by mathematical societies. For instance, see point 4 of the EMS code of practice for joint responsibility \cite{EMS_code_of_practice}.

However, this project involved over fifty contributors spread across the world with diverse academic and professional backgrounds. They collaborated across several timezones and countries over the internet. The aforementioned process does not scale. Collaborators do not usually know each other nearly as well as they would in a traditional project. Thus such a collaboration does not have the same level of mutual trust. Further, as the number of contributors grows beyond the single digits, it becomes increasingly difficult to ensure the robustness of each other's results, because of the sheer volume of material produced. Even delegating responsibility for the various pieces of mathematical work and integrating them into a coherent whole becomes difficult. Concretely, the scaling challenge manifests in several ways:
\begin{itemize}
    \item Partitioning and allocating tasks to voluntary contributors, keeping track of progress on the respective subtasks, and ensuring that everybody gets a fair chance at contributing without conflicting submissions for the same subproblems.
    \item Homogenising the mathematical content generated across multiple discussions spanning various forums into a coherent piece of work.
    \item Tracking progress relative to the goals of the project.
    \item Verifying the correctness of contributions made by more than fifty people with diverse backgrounds who might not share a common mathematical vocabulary, and who collaborate across multiple timezones, using a diverse set of tools.
\end{itemize}

Of the challenges mentioned above, this section deals with the first, second, and last. We briefly address the third challenge of tracking progress, the tools for which are described in \Cref{sec:gui-sec}. We spend a lot of time on the last point of trust and verification of results for two reasons. On the one hand, use of tools like \emph{Lean} is fairly new in mathematical research, and while the community researching theorem provers is familiar with their guarantees and limitations, a clear academic exposition targeted at mathematics researchers will be a helpful resource for future reference, to fill a gap that is currently covered by online forums and folklore. We also describe the important role played by a number of other tools in the project.

\subsection{The Blueprint tool}

The formalization of proofs is an act of careful engineering. It is therefore helpful to have a blueprint with detailed natural language lemmata, definitions, and proof sketches in \emph{Lean}. In the \emph{Lean} community it has been conventional to use the \emph{Lean blueprint} tool by Patrick Massot et al.~\cite{leanblueprint}. The typical formalization project has a clearly defined set of target theorems, and the authors of the project work with a known proof, to produce a clear roadmap for the formalization. The \emph{Lean blueprint} tool is capable of linking each piece of this natural language document to its \emph{Lean} encoding, tracking the dependency of definitions and theorems, and progress through them, by producing a key coloured dependency graph. Thus the managers of the formalization project can not only organise the project to distribute tasks among contributors, but also track when various pieces of the formalization are complete.

In this project, we were entering uncharted mathematical territory. We had a clear list of tasks to accomplish, namely to prove the implication or anti-implication between every pair of equational laws, up to transitivity and duality. At the same time there was no clearly known pen and paper proof available for any of these beforehand. This meant that we could not prepare the blueprint of the project in advance and organise the formalization around it. Thus the traditional roles played by the blueprint were replaced by a number of other tools and mechanisms. In particular, the dependency graph did not play its traditional role in formalization projects. We developed a number of visual tools to track our progress in the project in terms of remaining open implications and anti-implications (see \Cref{sec:gui-sec}). Within \emph{Lean}, every equational result was tagged with the \texttt{@[equational\_result]} attribute to identify the theorem as one of the project goals, and this attribute was used to collect the status of all the goal theorems of the project. Instead of covering the dependency graph node by node, progress in the project happened as various contributors uncovered some structural ideas or heuristics that helped ATPs solve one or more pairs of laws.

The blueprint tool played a very important role in recording our progress and formalizing these classes of implications or anti-implications. It is the only comprehensive record of all the techniques that were employed in the project. Further at the level of specific implications and anti-implications, the blueprint and formalization evolved as in other projects, hand in hand. As an example, the formalization of the anti-implication  $\Eq{1729} \nmodels \Eq{817}$ proceeded through several iterations of refinement of the blueprint and formalization.

In conclusion, when using ITPs for tackling open problems, especially at scale, we observed that the role of the blueprint changed, but it still remained an important way to track and document our progress at a local level across the project.

\subsection{The project template}
When working on a formalization project, there are many moving pieces that need to work in concert. At the core level, there is the project set up by \emph{Lean}'s build and dependency management system \emph{lake}. But in addition to that, there are several pieces, including the aforementioned blueprint tool, as well as scripts that a user may choose to run to visualise various aspects of the project, or check the project in specific ways, or compile documentation automatically as the project advances. These additional tasks are accomplished by a number of external tools, and combining them in a mutually compatible way can be challenging. We side-stepped most of these issues by using the GitHub template repository of Pietro Monticone~\cite{Monticone_LeanProject_2025}. At the same time, when we began the project, the template in place was suited for more conventional formalization projects and the tooling they required. It also did not include the scripts that enabled automated project management support that we added, as well as support for deploying our visualisation tools and the paper. Over the course of the project, the \texttt{leanproject} template in turn received substantial new additions. One elementary example is the addition of git pre-push hooks, which are scripts that perform a basic sanity check on the local working copy of a contributor before pushing their contributions to the central GitHub repository.

\subsection{The \emph{Lean} Zulip chat forum}
The \emph{Lean} community traditionally congregates on the \href{https://leanprover.zulipchat.com}{leanprover Zulip chat forum}\footnote{\url{https://leanprover.zulipchat.com}}. Our project was coordinated and organised primarily from this forum. At the beginning we created a channel called \texttt{Equational}. Zulip allows the creation and management of discussion topics within the scope of a channel. We made extensive use of the Zulip channel for several purposes. In the beginning it became the gathering point for new contributors. The new contributions process was designed and discussed on this forum. Later, topics were created for each specific technical topic, including the metatheory and its formalization, specific design decisions, specific implications and anti-implications, design of tools, etc. As shown in \Cref{fig:proj_mgmt_flow}, the Zulip chat served as the beginning of the contributions process for each piece of the project. Contributors first discussed their proposed contributions or specific problems they tackled on Zulip before following the steps of claiming tasks on GitHub, writing a blueprint write up and/or formalization.

\subsection{Organising the collaboration: the precedent set by the PFR project}
When five people collaborate in person, splitting up the research on a question into subtasks and assigning them to collaborators can be accomplished by discussion and consensus. When there are more than fifty collaborators working together online, a more systematic approach is required. In previous formalization projects such as the formalization of the proof of the Polynomial Freiman--Ruzsa (PFR) conjecture \cite{PFR_Tao_Dilles_2023}, tasks were managed over the \emph{Lean} zulipchat forum. The organiser of the project, Terence Tao, posted a series of message threads. Each thread corresponded to a list of outstanding tasks. These tasks were then claimed by collaborators on Zulip. The claims were recorded on a first-come first-served basis by the organiser by tagging the respective users against the tasks. Contributors could claim any open task and disclaim tasks if they couldn't finish it, with the organiser keeping track of these requests. This system allowed contributors to take their time to flesh out their work, without worrying about competing claims to the same task. Further, it helped the organisers track the task assignment and communicate with the respective collaborators to track and ascertain progress. Unfortunately, this involved a lot of manual and time-consuming management of the task list by organisers. In this project, we automated several pieces of this approach. This freed up organisers to help contributors and review their contributions.

\subsection{Organizing the collaboration in this project}
We adopted tools that are familiar to software engineers as ticket systems but are also known in the wider world of industrial production, such as the kanban system. Our project dashboard was built using the GitHub projects feature. We were able to encode some pieces of our automation using the standard GitHub-provided interface. For the rest, we relied on \emph{continuous integration} scripts (hereon~CI\@). The exact flow of contributions is specified in the \texttt{CONTRIBUTING.md} file of the project repository~\cite{The_Equational_Theories_repository}. Briefly,

\begin{enumerate}
    \item Tasks were proposed by organisers. A contributor might start a discussion on Zulip or raise an issue on GitHub to prompt the organisers to launch tasks.
    \item Contributors could then claim tasks with a comment under the task. The CI ensured that at most one contributor could claim a task at any time.
    \item Contributors could then work on the task and propose a corresponding pull request.
    \item Upon completion of the task, the pull request received reviews, while the CI automatically checked that the project compiled and passed additional checks such as \emph{Lean}'s environment replay tool \emph{leanchecker} and the semi-external checker \emph{lean4lean}~\cite{lean4lean}.
    \item If all was well, the PR was merged onto the main branch of the project repository.
\end{enumerate}

At any point in this process, the contributor could disclaim the task or replace a proposed PR with an alternative. In addition, organisers could always step in to fix any errors that occurred and follow up with contributors. Each of the steps described above happened automatically, triggered by a well-defined set of actions described in the \texttt{CONTRIBUTING.md} file. The typical workflow of this process is shown in the flowchart in \Cref{fig:proj_mgmt_flow}. The figure omits error handling and situations where organisers might manually intervene. The user interface to this project management is the GitHub project dashboard, of which we include a snapshot in \Cref{fig:proj_dashboard}

We note that our method has since been adopted by other major formalization projects including the one to formalize Fermat's Last Theorem \cite{FLT_Lean}.
\begin{figure}[t]
    \centering
    \includegraphics[width=.86\textwidth,trim={70 110 240 70},clip]{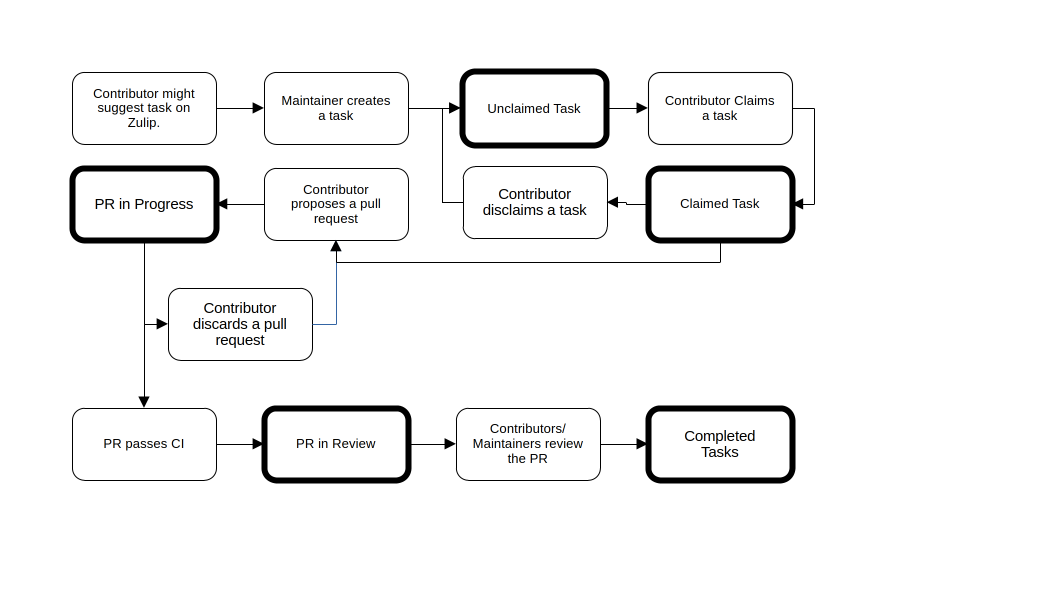}
    \caption{\label{fig:proj_mgmt_flow} A partial flowchart of the automated task management process. Each task corresponds to an issue. A pull request is created to resolve tasks and once a pull request linked to a task is merged, the task is considered complete. The thick boxes represent states of the project dashboard represented by task columns. The movement of tasks between these states is automated by the CI which is triggered upon specific actions performed by contributors on the respective GitHub issues and pull requests. A more detailed description is found in the contributions file of the GitHub file, named \texttt{CONTRIBUTING.md} by convention.}
\end{figure}

\begin{figure}[t]
    \centering
    \includegraphics[width=1.0\textwidth]{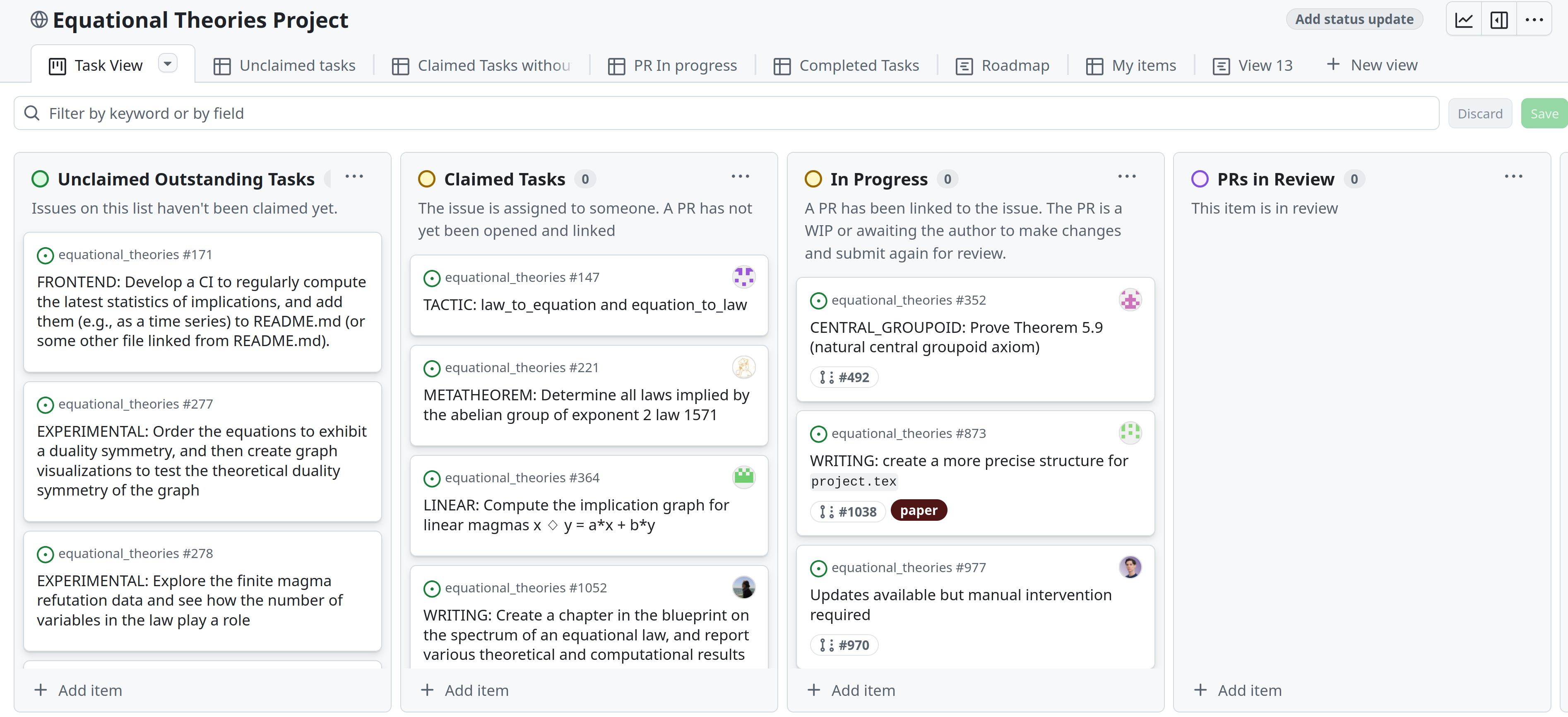}
    \caption{\label{fig:proj_dashboard} A snapshot of the project dashboard as of 30 July 2025}
\end{figure}
\subsection{Trusting ITPs to scale collaboration}
In our project we used the interactive theorem prover (hereon ITP) \emph{Lean} 4 \cite{the_lean4_paper} precisely to address these issues of scaling. The contents of this section are common knowledge in the ITP and ITP-adjacent research communities. The exposition is intended to be useful to a user of ITPs.

At its core, an interactive theorem prover implements an expressive logic, encoded in a suitable choice of mathematical foundations. \emph{Lean} 4 has the calculus of constructions extended by inductive types as its core logic. This logic is sufficiently powerful to express mathematical definitions and theorems for almost all areas of mathematical interest, while being relatively spartan and easy to write proof checkers for. Additionally, modern ITPs provide a convenient programming language which helps express mathematical ideas in a syntax closer to a mathematician's intuition than would be permitted by raw logical terms. A subset of ITPs like \emph{Lean}, \emph{Rocq} (formerly \emph{Coq}), and \emph{Isabelle} go one step further and provide the means to generate proofs through so-called \emph{tactics}. There are usually numerous tactics, each specialised for specific proof generation methods. Among other things, they search mathematical libraries, simplify expressions, and identify lemmas and hypotheses to make progress in proofs. The proofs generated by this overlying programming machinery are terms in the core logical calculus which are checked mechanically by the proof checker. But in a large project, there is more to trust. It helps to understand the nature and limits of trust one can place on ITPs.
\looseness=-1

For our purposes, \emph{Lean} consists of three pieces:
\begin{enumerate}
    \item A core proof checker called the \emph{kernel}. This checker encodes a typed $\lambda$-calculus that is sufficiently expressive for most mathematical purposes. Without getting into details, theorems are encoded as a formal specification of the intended theorem statement. Proofs are encoded as deduction trees using known lemmata, acceptable axioms, and inference rules. The kernel checks whether this deduction tree constitutes a correct deduction in the context of existing theorems and lemmata, called the \emph{environment}. Once a theorem's proof has been checked, it is added to the environment and can be used for constructing subsequent proofs. In reality, \emph{Lean} allows users some flexibility in adding declarations into its environment without checking their types.
    \item A sophisticated programming language in which users express their definitions and theorems. Programs in this language are said to be \emph{elaborated} to produce definitions and proofs in the core logic that the kernel can check.
    \item A compiler which compiles executable \emph{Lean} code into reasonably efficient C code.
\end{enumerate}

Of these three pieces, \emph{the kernel is the smallest and most trusted piece}. Programs in the programming language are translated by an \emph{elaborator} to the spartan language of the kernel. In reality the elaborator does much more, but the key takeaways are the following:
\begin{itemize}
    \item The kernel only verifies programs of a very simple barebones language. A proof verified by \emph{Lean} can be trusted \emph{modulo} the correctness of the kernel implementation. This caveat can usually be dropped because the kernel is also one of the most battle-hardened pieces of an ITP.
    \item The kernel checks proofs against a given specification. This means that if the formal theorem statement itself is flawed or incorrectly uses other definitions, a correct and verified proof of this theorem would be mathematically meaningless. The formal statement of a theorem gives expression to a mathematician's intuition and intention, and as such, the only check against mistakes in this area come from human review. \emph{Lean} cannot offer any guarantees against false statements written by its users or artificial intelligence tools.
    \item The higher level programming language can in-principle, generate arbitrary deductions. Their correctness is ultimately validated by the kernel. This gives the higher level programming language more leeway in generating possibly incorrect deductions.
    \item Ideally the environment only consists of definitions and lemmata already checked by the kernel. Thus their correct usage in the proof of subsequent theorems is valid. However this assumption is often not entirely true in practical implementations of theorem provers for efficiency reasons. In this case, trust in the kernel is restored by replaying each assumption from scratch in the kernel. This can be accomplished by external checkers. Such checkers, which are ideally independent implementations of the \emph{kernel}, can also guard against implementation mistakes in the kernel. In this project we used the environment replay tool \textit{lean4checker} and the \textit{lean4lean} \cite{lean4lean} tool. To our knowledge this was the first such use of the \textit{lean4lean} tool.

\end{itemize}

\begin{remark}
    A bug in the kernel that allows false statements to be proved is usually called a \emph{soundness} bug. Concretely, a soundness bug can be exploited to produce a kernel-certified proof of the proposition \texttt{False}. Such bugs are rare but not entirely non-existent. This is distinct from being able to prove \texttt{False} by simply assuming contradictory or false statements. A proof of the proposition \texttt{False} implies a proof of any statement by \emph{ex falso quod libet}.
\end{remark}

Users of \emph{Lean} only interact directly with the higher level programming language and usually get confirmation of the correctness of their proofs through the editor interface. Further, collaborators on a project such as ours are likely to deploy automated tools such as SAT solvers, SMT solvers, first-order theorem provers, and perhaps even modern AI tools. These tools usually produce proof certificates which are imported or inserted into a \emph{Lean} source file. Some modern AI tools are integrated into code editors which might automatically produce or edit even the statements of theorems and the definitions they deal with. Given the limits to trust mentioned above, a productive and useful collaboration using \emph{Lean} also requires a collaboration and verification infrastructure combining human effort and automated tools. It is in this context that we discuss the project infrastructure. It is a concrete answer to the questions posed by one of the authors at the beginning of this project \cite{Tao_blog_Sep_2024} that combines tools from the ITP community and the software engineering community. All these tools already exist. The goal of this exposition is to explain how they address the concerns described above.

\paragraph{\textbf{The non-Lean pieces:}} While \emph{Lean} can check proofs of theorems up to the limitations described above, a project of this scale involves the use of several non-\emph{Lean} tools. For example there are tools which extract the proven implications and anti-implications. There are tools which construct various visualisations. There are also metaprograms which call external automated theorem provers, and extract proof certificates from them to construct a \emph{Lean} proof out of them. In keeping with the garbage-in garbage-out principle, if these tools get spurious inputs and throw spurious outputs, \emph{Lean} can only tell us that the proofs are incorrect after the formal proofs have been translated to \emph{Lean}. It cannot, for instance, stop us from generating a large number of spurious conjectures that misguide contributors because of a simple index error in an array. Further, the continuous integration scripts that run \emph{Lean} and check the \emph{Lean} code with external checkers are not formally verified. Such bugs can only be uncovered by empirical testing and user reports. This highlights a basic caveat when using interactive theorem provers. These tools check something highly specific, a proof, against a specification. Contributors are responsible for the correctness of everything else. This makes the role of organisers and maintainers especially important.

\subsection{The Role of organisers and maintainers}
As mentioned before, the correctness of the project depends on a lot of moving pieces, many of which cannot be guaranteed to be correct or functional by \emph{Lean}. In this section, we give a brief description of the variety of tasks that need to be performed by maintainers. We wish to emphasize that the role played by maintainers is akin to that played by the principal investigator and senior postdoctoral researchers in a large experimental project, in that they need to understand the big picture and mathematical details of the project to a reasonable extent and be capable of making highly technical decisions, either themselves, on in consultation with subject experts. Further, they are likely to have limited time to get involved in highly specific details of the project, and while they might make technical contributions, most of their time will be spent managing and organising tasks for other contributors. The role requires a  combination of mathematical research skills and capacity with software engineering tools.

The organisers and maintainers have several tasks in a project such as ours.
\begin{itemize}
    \item They are responsible for monitoring the Zulip chat and onboarding new contributors.
    \item They are responsible for creating new tasks based on requirements and Zulip discussions, and ensuring that they are properly assigned.
    \item They are responsible for ensuring that all the project automation functions smoothly and step in when an issue is detected. It greatly helps to have a geographically distributed set of maintainers across timezones to help fix issues at any time of the day.
      \looseness=-1
    \item They are responsible for reviewing all code, both \emph{Lean} proofs and theorems, and non-\emph{Lean} scripts. This includes ensuring that the automation to build and check proofs, compile the documentation and blueprint, and test and run scripts works smoothly. When necessary they must be willing to step in and build or repair the automation.
    \item They are responsible for reviewing and developing the basic definitions and theorems. As mentioned before, \emph{Lean} takes definitions and theorems for granted. Thus maintainers need to be familiar with both the mathematical content and good ways of expressing this content in  \emph{Lean}. Being proficient in  \emph{Lean} internals is helpful for maintainers in identifying anti-patterns like the use of certain tactics that lead to trusting the  \emph{Lean} compiler, such as \texttt{native\_decide}.
    \item Maintainers are responsible for helping contributors who might get stuck in a proof. Other contributors may also assist in such matters, but ultimately it is up to the maintainers to ensure progress.
    \item Experienced maintainers might also offer suggestions and guidance on how to produce shorter or more elegant proofs.
    \item They are responsible for ensuring that some basic standards are met in proof blocks that make proofs robust to upstream changes. For instance, non-terminal uses of the \texttt{simp} tactic must be replaced with the \texttt{simp only} tactic with an explicit list of lemmas used. Otherwise, changes in the behaviour of upstream libraries can change how the tactic works and affect the correctness of the proof, when the  \emph{Lean} toolchain is updated.
    \item They are responsible for maintaining some record of the progress of the project. Projects on this scale can take a long time and it can become hard to remember how the project progressed. It would be extremely tedious to try to reconstruct real-time impressions long after the fact just from the GitHub commit history and Zulip chat archive. Thus it is extremely helpful if maintainers keep and publish regular logs of activities at a frequency that is proportionate to the scale of activity in the project.  In our case, the main organiser maintained a daily log of activities during the busiest part of the project, which became less frequent as the progress rate slowed towards the end.
\end{itemize}

In this project, we were still learning a lot of the aforementioned lessons. A key learning from this project is that it really helps to have a maintainer structure in place before the project begins, rather than inventing one on the fly. Of course, new contributors can be onboarded as maintainers as necessary. But a small maintainer team in the initial stages can hamper proper review processes and result in suboptimal design choices in the  \emph{Lean} formalization that become hard to undo later; for instance, we developed a theory of free magmas before realizing that Mathlib already had this concept, but by that point we did not feel it worth the effort to refactor the existing code to use the Mathlib version. As one example, the initial list of equations were translated and put into one  \emph{Lean} file as opposed to several. This created excessively large  \emph{Lean} files which could have been managed better with a better file organisation.

In conclusion, we wish to emphasize that as projects scale, the administrative aspects of the project assume non-trivial importance. They require people who are proficient in at least part of the research topic, technical tools, and administrative matters. Setting these processes well in advance of announcing the project and inviting contributors should lead to a smoother project.

\section{Counterexample constructions}

In this section we collect the various techniques developed in the ETP to construct counterexamples to implications $\E \models {\E'}$.

\subsection{Finite magmas}\label{finite-sec}

A finite magma $\Magma$ of size $n$ can be assumed without loss of generality to have carrier $\{1,\dots,n\}$ and described by specifying the multiplication table $\op \colon \{1,\dots,n\} \times \{1,\dots,n\} \to \{1,\dots,n\}$.  By generating a list of all the equational laws $\Eq{j}$, $j=1,\dots,4694$ satisfied by this magma, one can create refutations: if $\Magma \models \Eq{j}$ and $\Magma \not \models \Eq{k}$, then clearly $\Eq{j} \nmodelsfin \Eq{k}$ and hence also $\Eq{j} \nmodels \Eq{k}$.  (As mentioned previously, these statements were organized in Lean using the \texttt{Facts} statement.) It is feasible to brute force over all $\sum_{n=2}^4 n^{n^2} \approx 4.3 \times 10^9$ non-trivial magmas of size at most $4$ to obtain many refutations of this type.
By performing brute force over all magmas up to size $4$, a total of $\num{13632566}$ implications ($61.9\%$ of all implications, and $96.3\%$ of the false ones) can be refuted with $524$ distinct magmas. Of these implications, $\num{13345053}$ were refuted with magmas of size\footnote{For an earlier computer exploration of size $3$ magmas, see \cite{berman-burris}.} $3$, with the remaining $\num{415293}$ requiring magmas of size $4$. Performing this search took 165 CPU-hours.

However, it is not feasible to exhaustively search over the $5^{5^2} \approx 3 \times 10^{17}$ magmas of size $5$, even after quotienting out by isomorphism and symmetry (which roughly saves a factor of $5! \times 2 = 240$).  Randomly sampling such magmas did not produce significant refutations, as random magmas of size $5$ tended to satisfy few laws, and the set of laws covered were usually also exhibited by smaller magmas.  A more fruitful approach was to randomly sample from magmas with additional properties that encouraged satisfiability of a greater set of laws.  These included linear and quadratic magmas (discussed below), and cancellative magmas.  On the other hand, some classes of magmas, such as commutative magmas, ended up producing a disappointingly small number of additional refutations.

For specific refutations, it was sometimes possible to locate a finite example with an ATP, particularly if one also imposed additional axioms (e.g., an idempotence axiom $x = x \op x$) that one suspected would be useful; see \Cref{automated-sec} for further discussion.  For medium-sized magmas (of size $n=5,6,7,8$), this appeared to be a more efficient approach than brute force exhaustion of all such magmas.

It is a result of Kisielewicz \cite{Kisielewicz} that every law $\Eq{n}$ with $n \leq 4694$ is either equivalent to the singleton law $\Eq{2}$, or else has a non-trivial finite model; in other words, the implications $\Eq{n} \models \Eq{2}$ and $\Eq{n} \modelsfin \Eq{2}$ are equivalent for $n \leq 4694$.  Our brute force search revealed that in the latter case there is always a model of size $2 \leq n \leq 5$, with the lone exception of $\Eq{1286}$ (and its dual $\Eq{2301}$), for which the smallest non-trivial finite model was of size $7$, as presented in \Cref{1286-ex} below.  In fact, most of the $\num{4694}$ laws either only had trivial models, or had a size $2$ model, as shown in \Cref{size-table}.
\begin{table}
\centering
\caption{Number of laws of order at most $4$ whose smallest non-trivial model (if any) is of a given size.}\label{size-table}
\sisetup{table-format=2, table-alignment-mode=format, table-number-alignment=center}
\begin{tabular}{lS[table-format=4]SSSS[table-format=1]S[table-format=4]}
  \toprule
  Size of smallest non-trivial model & {$2$} & {$3$} & {$4$} & {$5$} & {$7$} & {Trivial only} \\
  \midrule
  Number of laws of order $0$ & 1 & & & & & 1 \\
  \phantom{Number of laws of }order $1$ & 3 & & & & & 2 \\
  \phantom{Number of laws of }order $2$ & 27 & & & & & 12 \\
  \phantom{Number of laws of }order $3$ & 229 & 14 & 2 & & & 119 \\
  \phantom{Number of laws of }order $4$ & 2876 & 18 & 12 & 14 & 2 & 1362 \\
  \cmidrule{1-7}
  \qquad\qquad\, total (order $\leq 4$) & 3136 & 32 & 14 & 14 & 2 & 1496 \\
  \bottomrule
\end{tabular}
\end{table}

\begin{remark} The laws satisfied by a given finite magma $\Magma$ need not be finitely axiomatizable.  The smallest example is the three-element magma $\{0,1,2\}$ with $1 \op 2 = 1$, $2 \op 1 = 2 \op 2 = 2$, and $x \op y = 0$ for all other $x,y$ \cite{murskii-1}.  It was also shown in \cite{murskii-2} that ``almost all'' magmas $\Magma$ (in a certain precise sense) are idemprimal\footnote{A magma is idemprimal if every idempotent function is expressed by a term.  This is weaker than being primal (in which \emph{every} function is expressible by a term), but stronger than being quasiprimal (in which the \emph{discriminator} $f(a,b,c)$, defined to equal $c$ if $a=b$ and $a$ otherwise, is expressed by a term); quasiprimality is sufficient to show that all other finite magmas satisfying the laws of $\Magma$ are isomorphic to products of submagmas of $\Magma$.  By combining these observations with \cite{Padmanabhan}, one can also show that quasiprimal finite magmas can be axiomatized by a single equation.  We thank Stanley Burris for these comments, as well as the other references and observations in this remark.}, which implies that their laws are finitely axiomatizable, and all other finite magmas satisfying these laws are isomorphic to powers of $\Magma$.
\end{remark}

\subsection{Linear models}\label{linear-sec}

As it turns out, a particularly fruitful source of counterexamples is the class of \emph{linear magmas}, where the carrier $M$ is a ring (which may be commutative or noncommutative, finite or infinite), and the operation $\op$ is given by $x \op y = ax + by$ for some coefficients $a,b \in M$; one can also generalize this slightly to \emph{affine magmas}, in which the operation is given by $x \op y = ax + by + c$, but for simplicity we shall focus on linear magmas here.  It is easy to see that in a linear magma, any word $w(x_1,\dots,x_n)$ of $n$ indeterminates also takes the linear form
$$ w(x_1,\dots,x_n) = \sum_{i=1}^n P_{w,i}(a,b) x_i$$
for some (possibly noncommutative) polynomial $P_{w,i}$ in $a,b$ with integer coefficients.  Thus, a linear magma will satisfy an equational law $w_1 \formaleq w_2$ if and only if the pair $(a,b)$ lies in the (possibly noncommutative) variety
\begin{equation}\label{variety}
  V_{w_1,w_2}(M) \coloneqq \{ (a,b) \in M \times M: P_{w_1,i}(a,b) = P_{w_2,i}(a,b) \hbox{ for all } i \} \subseteq M^2.
\end{equation}
As such, a necessary condition for such a law $w_1 \formaleq w_2$ to entail another law $w'_1 \formaleq w'_2$ is that one has the inclusion
$$ V_{w_1,w_2}(M) \subseteq V_{w'_1,w'_2}(M)
$$
for all rings $M$.  For commutative rings, this criterion can be checked in an automatable fashion by standard Gr\"obner basis techniques; in the noncommutative case one can use methods such as the diamond lemma \cite{diamond-lemma}.

\begin{example}[Commutative counterexample]\label{1286-ex} For the law $\x \formaleq \y \op (((\x \op \y) \op \x) \op \y)$ \eqref{eq1286}, the variety \eqref{variety} can be computed to be
$$ \{ (a,b) \in M \times M: 1 = ba^3+bab, 0 = a + ba^2 b + b^2 \}$$
while the variety for the idempotent law $\Eq{3}$ is
$$ \{ (a,b) \in M \times M: a+b=1 \}.$$
Thus, to show that $\Eq{1286}$ does not entail $\Eq{3}$, it suffices to locate elements $a,b$ of a ring $M$ for which one has $1 = ba^3+bab$, $0 = a + ba^2 b + b^2$, and $a+b \neq 1$.  Here one can take a commutative example, for instance when $M = \Z/p\Z$ and $(p,a,b) = (11,1,7)$.
\end{example}

\begin{example}[Noncommutative counterexample]\label{1117-ex} For the law $\x \formaleq \y \op ((\y \op (\x \op \z)) \op \z)$ \eqref{eq1117}, the variety \eqref{variety} can be computed to be
$$ \{ (a,b) \in M \times M: 1 = baba, 0 = a+ba^2, 0 = bab^2 + b^2 \}$$
while the variety for $\x \formaleq (\x \op ((\x \op \x) \op \x)) \op \x$ \eqref{eq2441} is
$$ \{ (a,b) \in M \times M: 1 = a^2 + aba^2 + abab + ab^2 + b \}.$$
Observe that if $ba = -1$, then $(a,b)$ automatically lies in the first set, and lies in the second set if and only if $(ab+1)(b-1) = 0$.  One can then show that $\Eq{1117}$ does not imply $\Eq{2441}$ by setting $a = L$, $b = -R$ where $L, R$ are the left and right shift operators respectively on the ring of integer-valued sequences $\Z^\N$.  With some \emph{ad hoc} effort one can convert this example into a less linear, but simpler (and easier to formalize) example, namely the magma with carrier $\Z$ and operation $x \op y = 2x - \lfloor y/2 \rfloor$.
\end{example}

\begin{remark} As essentially observed in \cite{austin}, if there is a commutative linear counterexample to an implication $\E \models {\E'}$, then by the Lefschetz principle this counterexample can be realized in a finite field ${\mathbb F}_q$ for some prime power $q$ (and by the Chebotarev density theorem one can in fact take $q$ to be a prime, so that the carrier is of the form $\Z/p\Z$ for some prime $p$), so that one also has $\E \modelsfin {\E'}$.  As such, we have found that an effective way to refute implications by the commutative linear magma method is to simply perform a brute force search over linear magmas $x \op y = ax + by$ in $\Z/p\Z$ for various triples $(p,a,b)$.

On the other hand, the refutations obtained by noncommutative linear constructions need not have a finite model.  For instance, consider the refutation $\Eq{1117} \nmodels \Eq{2441}$ from \Cref{1117-ex}.  The law $\Eq{1117}$ can be rewritten as $L_y R_z L_y R_z x = x$.  This implies that $R_z$ is injective and $L_y$ is surjective for all $y,z$.  For finite magmas $\Magma$, this then implies that the $L_y, R_z$ are in fact invertible, and hence we have also $R_z L_y R_z L_y x = x$, which implies $\Eq{2441}$ by setting $x=y=z$.  Thus, the refutation $\Eq{1117} \nmodels \Eq{2441}$ is ``immune'' to finite counterexamples.
\end{remark}

\begin{remark}  One can also consider nonlinear magma models, such as quadratic models $x \op y = ax^2 + bxy + cy^2 + dx + ey + f$ in a cyclic group $\Z/N\Z$.  For small values of $N$, we have found such models somewhat useful in providing additional refutations of implications $\E \modelsfin {\E'}$ beyond what can be achieved by the linear or affine models.  However, as the polynomials associated to a word $w(x_1,\dots,x_n)$ tend to be of high degree (exponential in the order of the word), it becomes quite rare for such models to satisfy a given equation $\E$ when $N$ is large.
\end{remark}

\begin{remark} One can also consider the seemingly more general linear model $x \op y = ax + by$, where the carrier $M$ is now an abelian group, and $a,b$ act on $M$ by homomorphisms, that is to say that they are elements of the endomorphism ring $\mathrm{End}(M)$.  However, this leads to exactly the same varieties \eqref{variety} (where $M$ is now replaced by the endomorphism ring $\mathrm{End}(M)$) and so does not increase the power of the linear model for the purposes of refuting implications.
\end{remark}

On the other hand, there are certainly some refutations $\E \nmodels {\E'}$ of implications that are ``immune'' to both commutative and noncommutative linear models, in the sense that all such models that satisfy $\E$, also satisfy ${\E'}$.  One such example is the refutation $\Eq{1485} \not \models \Eq{151}$, which we discuss further in \Cref{twisting-sec} below.

\subsection{Translation-invariant models}\label{translation-sec}

It is natural to look for counterexamples amongst magmas that satisfy a large number of symmetries.  One such class of counterexamples are \emph{translation-invariant models}, in which the carrier $M$ is a group, and the left translations of this group form isomorphisms of the magma $\Magma$.  In the case of an abelian group $M = (M,+)$, such models take the form
\begin{equation}\label{xop-add}
  x \op y = x + f(y-x)
\end{equation}
for some function $f \colon M \to M$; in the case of a non-abelian group $M = (M,\cdot)$, such models instead take the form
\begin{equation*}
x \op y = x f(x^{-1} y).
\end{equation*}
For such models, the verification of an equational law in $n$ variables corresponds to a functional equation for $f$ in $n-1$ variables, as the translation symmetry allows one to normalize one variable to be the identity (say). This can simplify an implication to the point where an explicit counterexample can be found.  These functional equations are trivial to analyze when $n=1$.  For $n=2$, they are not as trivial, but still quite tractable, and has led to several refutations in practice.  The method does not appear to be particularly effective for $n>2$ due to the complexity of the functional equations.

\begin{example}[Abelian example]\label{abex}  For the law $\x \formaleq (\x \op \y) \op ((\x \op \y) \op \y)$ \eqref{eq1648}, we apply the abelian translation-invariant model \eqref{xop-add} with $y=x+h$ to obtain
\begin{align*}
  x \op y &= x + f(h) \\
  (x \op y) \op y &= x + f(h) + f(h-f(h)) \\
  (x \op y) \op ((x \op y) \op y) &= x + f(h) + f(f(h-f(h)))
\end{align*}
so that this model satisfies $\Eq{1648}$ if and only if the functional equation
$$f(h) + f(f(h-f(h))) = 0$$
holds for all $h \in M$.  Similarly, the law $\x \formaleq (\x \op (\x \op \y)) \op \y$ \eqref{eq206} is satisfied if and only if
$$ f(f(h)) + f(h - f(f(h))) = 0$$
for all $h \in M$.  One can now check that the function $f \colon \Z \to \Z$ defined by $f(h) \coloneqq - \mathrm{sgn}(h)$ (thus $f(h)$ equals $-1$ when $h$ is positive, $+1$ when $h$ is negative, and $0$ when $h$ is zero) satisfies the first functional equation but not the second, thus establishing that $\Eq{1648} \nmodels \Eq{206}$.
\end{example}

\begin{example}[Non-abelian example]\label{trans-nonab}  We now obtain the opposite refutation $\Eq{206} \nmodels \Eq{1648}$ to \Cref{abex} using the non-abelian translation-invariant model.  By similar calculations to before, we now seek to find a function $f \colon M \to M$ on a non-abelian group $(M,\cdot)$ that satisfies the functional equation
\begin{equation}\label{206-eq}
 f(f(h)) f(f(f(h))^{-1} h) = 1
\end{equation}
for all $h \in M$, but fails to satisfy the functional equation
\begin{equation}\label{1648-eq}
   f(h) f(f(f(h)^{-1} h)) = 1
\end{equation}
for at least one $h \in M$.  Now take $M$ to be the group generated by three generators $a,b,c$ subject to the relations $a^2=b^2=c^2=1$, or equivalently the group of reduced words in $a,b,c$ with no adjacent letters in the word equal.  We define
$$ f(1) = 1, f(a)=b, f(b) = c, f(c) = a$$
and then $f(aw)=a$ for any non-empty reduced word $w$ not starting with $a$, and similarly for $b$ and $c$.  The equation \eqref{206-eq} can be checked directly for $h=1,a,b,c$.  If $h=aw$ with $w$ non-empty, reduced, and not starting with $a$, then $f(f(h))^{-1} = f(f(h)) = b$ and $f(f(f(h))^{-1} h) = f(baw) = b$, giving \eqref{206-eq} in this case, and similarly for cyclic permutations. Meanwhile, \eqref{1648-eq} can be checked to fail for $h=a$.
\end{example}

\begin{remark}  The construction in \Cref{trans-nonab} also has the following more ``geometric'' interpretation.  The carrier $M$ can be viewed as the infinite $3$-regular tree, in which every vertex imposes a cyclic ordering on its $3$ neighbors (for instance, if we embed $M$ as a planar graph, we can use the clockwise ordering).  For $x,y \in M$, we then define $x \op y$ to equal $x$ if $x=y$.  If $y$ is instead a neighbor of $x$, we define $x \op y$ to be the next neighbor of $x$ in the cyclic ordering.  Finally, if $y$ is distance two or more from $x$, we define $x \op y$ to be the neighbor of $x$ that is closest to $y$.  One can then check that this model satisfies \eqref{206-eq} but not~\eqref{1648-eq}.
\end{remark}

\begin{remark} These constructions are necessarily infinitary in nature, because $\Eq{206}$ and $\Eq{1648}$ can be shown to be equivalent for finite magmas. Indeed, $\Eq{206}$ can be written as $x = R_y L_x L_x y$, which implies that $R_y$ is surjective, hence injective, on a finite magma; writing $x = R_y z$ we conclude that $R_y z = R_y L_{z \op y} L_{z \op y} y$ and hence $z = L_{z \op y} L_{z \op y} y$, giving $\Eq{1648}$.  The opposite implication is similar (using $\Eq{1648}$ to show that $R_y$ is injective, hence surjective), and is left to the reader.
\end{remark}

  Some refutations $\E \nmodels {\E'}$ are ``immune'' to translation-invariant models, in the sense that any translation-invariant model that satisfies $\E$, also satisfies ${\E'}$.  One obstruction is that for such models, the squaring map $S$ is necessarily an invertible map, since $Sx = x + f(0)$ in the abelian case and $Sx = xf(1)$ in the non-abelian case. On the other hand, adding the assumption of invertibility of squares can sometimes force the implication $\E \models {\E'}$ to hold.  For instance, the commutative law $\x \op (\y \op \y) \formaleq (\y \op \y) \op \x$ \eqref{eq4482} for a square and an arbitrary element will imply the full commutative law $\Eq{43}$ for translation-invariant models due to the surjectivity of $S$, but does not imply it in general (as one can easily see by considering models where $S$ is constant).

\subsection{The twisting semigroup}\label{twisting-sec}

Suppose one has a magma $\Magma$ satisfying a law $\E$, that also enjoys some endomorphisms $T, U \colon \Magma \to \Magma$.  Then one can ``twist'' the operation $\op$ by $T,U$ to obtain a new magma operation
\begin{equation}\label{twist} x \op' y := Tx \op Uy.
\end{equation}
If one then tests whether this new operation $\op'$ satisfies the same law $\E$ as the original operation $\op$, one will find that this will be the case provided that $T,U$ satisfy a certain set of relations.  The semigroup generated by formal generators $\mathrm{T}, \mathrm{U}$ with these relations will be called the \emph{twisting semigroup} $\operatorname{Twist}_\E$ of $\E$.  This can be best illustrated with some examples.

\begin{example}  We compute the twisting semigroup of $\x \formaleq (\y \op \x) \op (\x \op (\z \op \y))$ \eqref{eq1485}.  We test this law on the operation \eqref{twist}, thus we consider whether
$$x = (y \op' x) \op' (x \op' (z \op' y))$$
holds for all $x,y,z \in M$.  Substituting in \eqref{twist} and using the homomorphism property repeatedly, this reduces to
$$x = (T^2y \op TUx) \op (UTx \op (U^2T z \op U^3y)).$$
If we impose the conditions $TU=UT=\mathrm{id}$, $T^2 = U^3$, then this equation would follow from $\Eq{1485}$ (with $x,y,z$ replaced with $TUx=UTx=x$, $T^2 y$, and $U^2 Tz$ respectively).  Thus the twisting semigroup $\operatorname{Twist}_{\Eq{1485}}$ of $\Eq{1485}$ is generated by two generators $\mathrm{T}, \mathrm{U}$ subject to the relations $\mathrm{T} \mathrm{U}=\mathrm{U} \mathrm{T} = 1$, $\mathrm{T}^2 = \mathrm{U}^3$.  This is a cyclic group of order $5$, since the relations can be rewritten as $\mathrm{T}^5 = 1$, $\mathrm{U} = \mathrm{T}^{-1}$.

Now consider $\x \formaleq (\x \op \x) \op (\x \op \x)$ \eqref{eq151}.  Applying the same procedure, we arrive at
$$x = (T^2 x \op TUx) \op (UT x \op U^2 x)$$
so the twisting group $\operatorname{Twist}_{\Eq{151}}$ is generated by two generators $\mathrm{T}, \mathrm{U}$ subject to the relations $\mathrm{T} \mathrm{U}=\mathrm{U} \mathrm{T} = \mathrm{T}^2 = \mathrm{U}^2 = 1$.  This is a cyclic group of order $2$, since the relations can be rewritten as $\mathrm{T}^2 = 1$, $\mathrm{U} = \mathrm{T}$.
\end{example}

Suppose the twisting semigroup $\operatorname{Twist}_\E$ is not a quotient of $\operatorname{Twist}_{{\E'}}$, in the sense that the relations that define $\operatorname{Twist}_{{\E'}}$ are not satisfied by the generators of $\operatorname{Twist}_\E$.  Then one can often disprove the implication $\E \models {\E'}$ by attempting the following procedure.
\begin{itemize}
\item First, locate a non-trivial magma $\Magma = (M,\op)$ satisfying the law $\E$.  Then the Cartesian power $M^{\operatorname{Twist}_\E}$ of tuples $(x_W)_{W \in \operatorname{Twist}_\E}$, with the pointwise magma operation, will also satisfy $\E$.
\item Furthermore, this Cartesian power admits two endomorphisms $T, U$ defined by
$$ T (x_W)_{W \in \operatorname{Twist}_\E} = (x_{W \mathrm{T}})_{W \in \operatorname{Twist}_\E};
U (x_W)_{W \in \operatorname{Twist}_\E} = (x_{W \mathrm{U}})_{W \in \operatorname{Twist}_\E},$$
which satisfy the relations defining $\operatorname{Twist}_\E$.
\item We now twist the magma operation $\op$ on $M^{\operatorname{Twist}_\E}$ by $T,U$ to obtain a new magma operation $\op'$ defined by \eqref{twist}, that will still satisfy law $\E$.
\item Because $T, U$ will not satisfy the relations defining $\operatorname{Twist}_{{\E'}}$, it is highly likely that this twisted operation will not satisfy ${\E'}$, thus refuting the implication $\E \models {\E'}$.  If $M$ and the twisting semigroup were finite, this approach should also refute $\E \modelsfin {\E'}$.
\end{itemize}

For instance, a non-trivial finite model for $\Eq{1485}$ is given by the finite field $\mathbb{F}_2$ of two elements with the NAND operation $x \op y \coloneqq 1-xy$.  If we twist $\mathbb{F}_2^5$ by the left shift $T(x_i)_{i=1}^5 = (x_{i+1})_{i=1}^5$ and right shift $U(x_i)_{i=1}^5 = (x_{i-1})_{i=1}^5$, where we extend the indices periodically modulo $5$, then the resulting operation
$$ (x_i)_{i=1}^5 \op' (y_i)_{i=1}^5 \coloneqq (1 - x_{i+1} y_{i-1})_{i=1}^5$$
on $\mathbb{F}_2^5$ still satisfies $\Eq{1485}$, but does not satisfy $\Eq{151}$, thus showing that $\Eq{1485} \nmodelsfin \Eq{151}$ and hence $\Eq{1485} \nmodels \Eq{151}$.  This particular implication does not seem to be easily refuted by any of the other methods discussed in this paper.

\subsection{Greedy constructions}\label{greedy-sec}

We have found \emph{greedy extension methods}, or \emph{greedy methods} for short, to be a powerful way to refute implications, especially when the carrier $M$ is allowed to be infinite.  Such constructions have a long history in model theory, with possibly the earliest\footnote{We thank Stanley Burris for this reference.} such construction due to Skolem \cite{skolem}. A basic implementation of this method is as follows.  To build a magma operation $\op \colon M \times M \to M$ that satisfies one law $\E$ but not another ${\E'}$, one can first consider \emph{partial magma operations} $\op \colon \Omega \to M$, defined on some subset $\Omega$ of $M \times M$. Thus $x \op y$ is defined if and only if $(x,y) \in \Omega$. A magma operation is then simply a partial operation which is \emph{total} in the sense that $\Omega = M \times M$.  We say that a partial magma operation is \emph{finitely supported} if $\Omega$ is finite.

In the language of first-order logic (in which functions and relations must be total), it is convenient to view a partial magma operation as a ternary relation $R(x,y,z)$ on $M$ with the axiom that $R(x,y,z) \wedge R(x,y,z') \implies z=z'$ for all $x,y,z \in M$.  The support $\Omega$ is then the set of $(x,y)$ for which $R(x,y,z)$ holds for some (necessarily unique) $z$, which one can then take to be the definition of $z = x \op y$.

We say that one partial operation $\op' \colon \Omega' \to M$ \emph{extends} another $\op \colon \Omega \to M$ if $\Omega'$ contains $\Omega$, and $x \op y = x \op' y$ whenever $x \op y$ (and hence $x \op' y$) are defined. Given a sequence $\op_n \colon \Omega_n \to M$ of partial operations, each of which is an extension of the previous, we can define the \emph{direct limit} $\op_\infty \colon \bigcup_n \Omega_n \to M$ to be the partial operation defined by $x \op_\infty y \coloneqq x \op_n y$ whenever $(x,y) \in \Omega_n$.
\looseness=-1

Abstractly, the greedy algorithm strategy can now be described as follows.

\begin{theorem}[Abstract greedy algorithm]\label{greedy-abstract} Let $\E,{\E'}$ be equational laws, and let $\Gamma$ be a theory of first-order sentences regarding a partial magma operation $\op \colon \Omega \to M$ on a carrier $M$.  Assume the following axioms:
  \setlength{\parskip}{0pt}
\begin{itemize}
  \item[(i)] (Seed) There exists a finitely supported partial magma operation $\op_0 \colon \Omega_0 \to M$ satisfying $\Gamma$ that contradicts ${\E'}$, in the sense that there is some assignment of variables in ${\E'}$ in $M$ such that both sides of ${\E'}$ are defined using $\op_0$, but not equal to each other.
  \item[(ii)]  (Soundness)  If $\op_n \colon \Omega_n \to M$ is a sequence of partial magma operations satisfying $\Gamma$ with each $\op_{n+1}$ an extension of $\op_n$, and the direct limit $\op_\infty$ is total, then this limit satisfies $\E$.
  \item[(iii)] (Greedy extension)  If $\op \colon \Omega \to M$ is a finitely supported partial magma operation satisfying $\Gamma$, and $a,b \in M$, then there exists a finitely supported extension $\op' \colon \Omega' \to M'$ of $\op$ to a possibly larger carrier $M'$, and also satisfying $\Gamma$, such that $a \op' b$ is defined.
\end{itemize}
Then $\E \nmodels {\E'}$.
\end{theorem}

We remark that this greedy method seems to be inherently infinitary in nature, and does not seem well adapted to refute finite magma implications $\E \modelsfin {\E'}$.

\begin{proof}  We work on the countably infinite carrier $\N$.  By embedding the finitely supported operation $\op_0$ from axiom (i) into $\N$, we can assume without loss of generality that $\op_0$ has carrier $\N$.  By similar relabeling, we can assume in (iii) that $M' = M$ when $M=\N$, since any elements of $M' \setminus \N$ that
appear in $\Omega'$ can simply be reassigned to natural numbers that did not previously appear in $\Omega$.  We well-order the pairs in $\N \times \N$ by $(a_n,b_n)$ for $n=1,2,\dots$.  Iterating (iii) starting from $\op_0$, we can thus create a sequence of finitely supported magma operations $\op_0, \op_1, \dots$ on $\N$ satisfying $\Gamma$, with each $\op_{n+1}$ an extension of $\op_n$, and $a_n \op_n b_n$ defined for all $n \geq 1$.  Then the direct limit $\op_\infty$ of these operations is total, and does not satisfy ${\E'}$ thanks to axiom (i).  On the other hand, by axiom (ii) it satisfies $\E$, and the claim follows.
\end{proof}

We refer to $\Gamma$ as the \emph{rule set} for the greedy extension method. To apply \Cref{greedy-abstract} to obtain a refutation $\E \nmodels {\E'}$, we have found the following trial-and-error method to work well in practice:
\begin{itemize}
\item[1.] Start with a minimal rule set $\Gamma$ that has just enough axioms to imply the soundness property for the given hypothesis $\E$.
\item[2.] Attempt to establish the greedy extension property for this rule set by setting $a \op' b$ equal to a new element $c \not \in M$, and then defining additional values of $\op'$ as necessary to recover the axioms of $\Gamma'$.
\item[3.]  If this can be done in all cases, then locate a seed $\op_0$ refuting the given target ${\E'}$, and \texttt{STOP}.
\item[4.]  If there is an obstruction (often due to a ``collision'' in which a given operation $x \op' y$ is required to equal two different values), add one or more rules to $\Gamma$ to avoid this obstruction, and return to Step 2.
\end{itemize}

As an example, we present

\begin{proposition}[$\Eq{73} \nmodels \Eq{4380}$]\label{73-4380} The law $\x \formaleq \y \op (\y \op (\x \op \y))$ \eqref{eq73} does not imply $\x \op (\x \op \x) \formaleq (\x \op \x) \op \x$ \eqref{eq4380}.
\end{proposition}

\begin{proof} To build a rule set $\Gamma$ that will imply $\Eq{73}$ when total, a natural first choice would be the single rule
\begin{itemize}
\item[1.] If $y \op (x \op y)$ is defined, then $y \op (y \op (x \op y))$ is defined and equal to $x$.
\end{itemize}
However, the greedy algorithm will fail just with this rule: if the partial operation has $x \op y$ and $z \op y$ both equal to some $w$ for some $x \neq z$, then any attempt to assign a value to $y \op (y \op w)$ will lead to a contradiction, as the above rule will force $y \op w$ to equal both $x$ and $z$.  Indeed, it is clear that $\Eq{73}$ forces all the right translation operators $R_y$ to be injective.  We therefore add this as an additional rule:
\begin{itemize}
\item[2.] If $x \op y$ and $z \op y$ are defined and equal, then $x=z$.
\end{itemize}
To avoid some unwanted edge cases, it is also convenient to impose the additional rule
\begin{itemize}
  \item[3.] If $x \op y$ is defined, it is not equal to $y$.
\end{itemize}
Unlike Rule 2, this rule is not forced by $\Eq{73}$, but can be enforced as part of the greedy construction.

The ruleset clearly satisfies the soundness axiom (ii) of \Cref{greedy-abstract}.  We now verify the greedy extension axiom (iii).  Let $\Omega,a,b$ be as in that axiom. We may assume that $a \op b$ is undefined, since we are done otherwise. We adjoin a new element $c$ to $M$ to create $M'$, and set $a \op' b = c$.  If we also have $b = d \op a$ for some $d$ (unique by Rule 2, and only possible for $a \neq b$ by Rule 3), set $a \op' c = d$ (this is necessary to retain Rule 1).  Of course, we also set $x \op' y = x \op y$ whenever $x \op y$ is already defined.

Since $c \not \in M$, it is clear that $\op'$ is a finitely supported partial magma operation on $M'$.  It is also clear that $\op'$ satisfies Rule 2 and Rule 3.   Now we case check Rule 1:
\begin{itemize}
\item Case 1: $x=c$ or $y=c$.  Not possible since no left multiplication with $c$ is defined.
\item Case 2: $x \op' y = c$.  Only possible when $x = a$, $y = b$, but then $y \op' (x \op' y)$ is undefined since $y = b \neq a$ if $d$ is defined.
\item Case 3. $y \op' (x \op' y) = c$.  Only possible when $y=a$ and $x=d$, and holds in this case.
\item Case 4: $x, y, x \op' y, y \op' (x \op' y) \neq c$: In this case $\op'=\op$ on all pairs, so the claim reduces to Rule 1 for $\op$, which holds by the induction hypothesis.
\end{itemize}
To conclude, we need to locate a seed $\op_0$ satisfying Rules 1,2,3 and containing a counterexample to $\Eq{4380}$.  One simple example is the carrier $\{0,1,2,3\}$ with $0 \op_0 0 = 1$, $0 \op_0 1 = 2$, $0 \op_0 2 = 0$, $1 \op_0 0 = 3$.
\end{proof}

This method is not guaranteed to halt in finite time, as there may be increasingly lengthy sets of rules one has to add to $\Gamma$ to avoid collisions.  However, in practice we have found many of the refutations that could not be resolved by simpler techniques to be amenable to this method (or variants thereof, as discussed below).

One can automate the above procedure by using ATPs (or SAT solvers) to locate new rules that are necessary and sufficient to resolve any potential collision (and which, \emph{a posteriori}, can be seen to be necessarily consequences of the law $\E$).  The seed-finding step (Step 3) is particularly easy to automate, and can also often be done by hand.  In some cases, the SAT solver calculations provided by these methods were difficult to formalize efficiently in Lean, and so we elected in some cases to replace computer-generated rulesets with shorter human-generated versions in preparation for the formalization step.

However, in some cases we have found it necessary to add ``inspired'' choices of rules that were not forced by the initial hypothesis $\E$, but which simplified the analysis by removing problematic classes of collisions from consideration.  We were unable to fully automate the process of guessing such choices; however, we found ATPs very useful for testing any proposed such guess.  In particular, if an ATP was able to show that the existing ruleset, together with a proposed new rule $A$, implied ${\E'}$, then this clearly indicated that one should not add $A$ to the rule set $\Gamma$.  Conversely, if an ATP failed to establish such an implication, this was evidence that this was a ``safe'' rule to impose.

We also found that human verification of the greedy extension property was a highly error-prone process, as the case analysis often included many delicate edge cases that were easy to overlook.  Both ATPs and the Lean formalization therefore played a crucial role in verifying the human-written greedy arguments, often revealing important gaps in those arguments that required either minor or major revisions to the rule set.

The greedy method can also be combined with the translation-invariant method, both in abelian and non-abelian settings. For instance, we can modify the proof of \Cref{greedy-abstract} to obtain the following variant:

\begin{theorem}[Noncommutative translation-invariant greedy algorithm]\label{nc-greedy-abstract} Let $F,F'$ be functional equations on groups, and let $\Gamma$ be a theory of first-order sentences regarding a partial function $f \colon \Omega \to G$ on a group $(G,\cdot, \cdot^{-1},1)$.  Assume the following axioms:
  \setlength{\parskip}{0pt}
  \begin{itemize}
    \item[(i)] (Seed) There exists a finitely supported partial function $f_0 \colon \Omega_0 \to G$ satisfying $\Gamma$ that contradicts $F'$, in the sense that there is some assignment of variables in $F'$ in $G$ such that both sides of $F'$ are defined using $f_0$, but not equal to each other.
    \item[(ii)]  (Soundness)  If $f_n \colon \Omega_n \to G$ is a sequence of partial functions satisfying $\Gamma$ with each $f_{n+1}$ an extension of $f_n$, and the direct limit $f_\infty$ is total, then this limit satisfies $F$.
    \item[(iii)] (Greedy extension)  If $f \colon \Omega \to G$ is a finitely supported partial function satisfying $\Gamma$, and $a \in G$, then there exists a finitely supported extension $f' \colon \Omega' \to G'$ of $f$ to a possibly larger group $G'$, and also satisfying $\Gamma$, such that $f'(a)$ is defined.
  \end{itemize}
  Then $F \nmodels F'$.
\end{theorem}

One can of course also develop an abelian analogue of the above theorem, in which $(G,+,-,0)$ and $(G',+,-,0)$ are now required to be abelian.  We can then give an alternate proof of \Cref{73-4380} as follows:

\begin{proof}[Second proof of \Cref{73-4380}] (Sketch)  The functional equations associated to $\Eq{73}$ and $\Eq{4380}$ are
$f^2(h^{-1} f(h)) =h^{-1}$ and $f^2(1) = f(1) f(f(1)^{-1})$ respectively.  We apply \Cref{nc-greedy-abstract} with the following ruleset:
\begin{itemize}
  \item[1.]  If $f(h^{-1} f(h))$ is defined, then $f^2(h^{-1} f(h))$ is defined and equal to $h^{-1}$.
  \item[2.]  If $h^{-1} f(h)$ and $k^{-1} f(k)$ are defined and equal, then $h=k$.
  \item[3.]  If $f(h)$ is defined, it is not equal to $h$.
\end{itemize}
Axiom (ii) is clear.  To verify axiom (iii), we can assume $f(h)$ is undefined, then adjoin an element $c$ freely to $G$ to create a larger group $G'$, and set $f'(h) = c$.  If $h = k^{-1} f(k)$ for some $k$ (which is unique by Rule 2, and only possible for $h \neq 1$ by Rule 3), then also set $f'(c) = k^{-1}$.  One can then check that axiom (iii) is satisfied.  For axiom (i), take $G$ to be a free cyclic group with one generator $a$, and set $f(1) = a$, $f(a) = a^3$, $f(a^3) = 1$, $f(a^{-1}) = a^3$ (say).
\end{proof}

More complex (and \emph{ad hoc}) variants of the greedy algorithm are possible.  In some cases, we were not able to preserve the finitely supported nature of the partial operation or partial function, and needed to extend that partial object at an infinite number of values at each step.  In other cases, one also had to add additional temporary data during the greedy process to record tasks that one wished to attend to at a later stage of the process, but could not handle immediately because it was awaiting some other operation to become well-defined.  We will not attempt to survey all possible variants of this method here, but refer the reader to the ETP blueprint for further examples.

\subsection{Modifying base models}\label{modify-base}

A general technique that we have found useful in obtaining a refutation such as $\E \nmodels {\E'}$ is to start with a simple base model $\Magma = (M,\op)$ that satisfies both $\E$ and ${\E'}$, and modify it in various ways to preserve $\E$, but create a violation of ${\E'}$.  There are many such possible modifications, but three general ways that have proven effective are as follows:

\begin{itemize}
  \item[(i)]  Modify the magma operation $\op \colon M \times M \to M$ on a small set in order to violate ${\E'}$, and then make further modifications as needed to recover $\E$.
  \item[(ii)]  Construct an \emph{extension} $\MagmaN$ of $\Magma$, equipped with a surjective magma homomorphism $\pi: \MagmaN \to \Magma$, and defined in terms of some additional data.  Then solve for that data in such a way that $\MagmaN$ satisfies $\E$ but not ${\E'}$.
  \item[(iii)]  Construct an \emph{enlargement} $\Magma' = (M',\op')$ of $\Magma = (M,\op)$, which is a magma that contains $\Magma$ as a submagma.  One needs to construct the multiplication table $\op$ on $(M' \times M') \backslash (M \times M)$ in order to retain $\E$ but disprove ${\E'}$.
\end{itemize}

One appealing case of (ii) that our project discovered, involving a ``magma cohomology'' analogous to (abelian) group cohomology, is that of an \emph{affine} extension of a magma ${\mathcal G} = (G,\op_G)$ by another magma $(M,\op_M)$ which is an abelian group $M$ with a linear magma operation $s \op_M t \coloneqq as + bt$ for some endomorphisms $a,b \in \mathrm{End}(M)$.  One can then consider extensions with carrier $G \times M$ and magma operation
\begin{equation}\label{xsyt}
 (x, s) \op (y, t) \coloneqq (x \op_G y, s \op_M t + f(x,y))
\end{equation}
for some function $f \colon G \times G \to M$.  If $(M,\op_M)$ and $(G,\op_G)$ already satisfy a law $\E$, then this extension will also satisfy $\E$ if and only if $f$ satisfies a certain ``cocycle equation'', which is a linear equation on $f$.  One can then sometimes use linear algebra to locate an $f$ that satisfies the cocycle equation for one law $\E$ but not another ${\E'}$, thus refuting the implication $\E \models {\E'}$.  An example is as follows:

\begin{proposition}[$\Eq{1110} \nmodels \Eq{1629}$]\label{1110-1629} The law $\x \formaleq \y \op ((\y \op (\x \op \x)) \op \y)$ \eqref{eq1110} does not imply $\x \formaleq (\x \op \x) \op ((\x \op \x) \op \x)$ \eqref{eq1629}, even for finite magmas.
\end{proposition}

\begin{proof}[Proof sketch] Using the linear ansatz, we find that $\Eq{1110}$ has a model $\Magma$ with carrier $\F_5$ (the finite field $\Z/5\Z$) with operation $x \op y = 3x-y$.  We then apply the ansatz \eqref{xsyt} with $G=M$.  One then finds that this operation satisfies $\Eq{1110}$ if $f \colon \F_5 \times \F_5 \to \F_5$  the cocycle equation
  $$3f(x,x) - 3f(y,2x) - f(3y-2x,y) + f(y,3y-x) = 0$$
for all $x,y \in \F_5$, and satisfies $\Eq{1629}$ if $f$ satisfies the cocycle equation
$$ f(2x,0) - f(2x,2x) = 0$$
for all $x \in \F_5$.  A routine symbolic algebra package computation reveals that the space of $f$ that satisfies the former equation is a six-dimensional vector space over $\F_5$, which is not contained in the solution space of the latter equation, giving the claim.
In fact, since these equations preserve the space of homogeneous polynomials of a fixed degree, one can use linear algebra to locate an example that is a homogeneous polynomial; one explicit choice is
\[
f(x,y) = y^5 +xy^4 + x^2y^3 +3x^3 y^2 + 3x^4 y. \qedhere
\]
\end{proof}

It may be of interest to develop this theory of ``magma cohomology'' further, for instance by defining higher order magma cohomology groups.

Now we give an example of how method (ii) can be combined with method (i).

\begin{proposition}[$\Eq{1659} \nmodels \Eq{4315}$]\label{1659-4315} The law $\x \formaleq (\x \op \y) \op ((\y \op \y) \op \z)$ \eqref{eq1659} does not imply $\x \op (\y \op \x) \formaleq \x \op (\y \op \z)$ \eqref{eq4315}.
\end{proposition}

\begin{proof}  There are two simple models for $\Eq{1659}$: the model $G$ with carrier $\Z/2\Z$ and operation $x \op y = x+1$, and the model $\Magma$ with carrier $\Z$ and operation $x \op y = x$.  Using the ansatz \eqref{xsyt}, one can soon discover that one obtains a magma operation $\op: (G \times \Z) \times (G \times \Z) \to G \times \Z$ with $f(0,0)=f(1,0)=0$, $f(0,1)=-1$, and $f(1,1)=1$.  This model still satisfies $\Eq{4315}$. However, we can create a modification $\op'$ of $\op$ as follows.  We will seek to violate $\Eq{4315}$ at $x = (0,0)$, $y = (0,0)$, $z = (1,0)$, thus we want
$$ (0,0) \op' ((0,0) \op' (0,0)) \neq (0,0) \op' ((0,0) \op' (1,0)).$$
We have $(0,0) \op (0,0) = (1,0)$ and $(0,0) \op (1,0) = (1,-1)$.  One can try to force the counterexample by setting $(0,0) \op' (1,0)$ to equal $(0,0)$ instead of $(1,-1)$. However, if this is the only change we make, then we no longer satisfy $\Eq{1659}$, since
$$ (1,0) \neq ((0,0) \op' (1,0)) \op' (((1,0) \op' (1,0)) \op' (1,t))$$
for any $t \in \N \backslash \{0\}$. But these are the only counterexamples created involving elements in the subset $G \times \N$ of $G \times \Z$; and if one then sets $(0,0) \op' (1,t) = (0,0)$ for \emph{all} $t \in \N$, and then restricts to $G \times \N$ (which is now closed under $\op'$), then one can check that the modified operation $\op'$ on this submagma now satisfies $\Eq{1659}$ but not $\Eq{4315}$ as required.
Incidentally, this submagma is isomorphic to the magma $\Magma'=(\Z,\op'')$ with $m\op''n=-m$ for $n<0$ and $m\op''n=-m-1$ for $n\geq 0$ under the bijection that maps $(0,s)$ to $s$ and $(1,s)$ to $-s-1$.
\end{proof}

The specific law $\x \formaleq \x \op ((\y \op \z) \op (\x \op \z))$ \eqref{eq854} turned out to be somewhat ``mutable'', in the sense that one can often change a small number
 of entries in the multiplication table of a finite magma satisfying this law, or add rows and columns to the table,
 in ways that preserve the law $\Eq{854}$.  This makes the law amenable to methods (i) and (iii) to construct new models of this equation that
 refute various implications  $\Eq{854} \nmodelsfin \E$, for instance by starting with a model that already refuted some stronger law ${\E'}$, and then attempting to modify it
 (possibly with ATP assistance) by some combination of methods (i) and (iii) to produce a model that violates $\E$.

 Some heuristics loosely inspired by discrete-time dynamical systems proved helpful.
 The idea is to generate a sequence of magmas, each of which is generated from the previous entry by applying various operations expected to increase the likelihood of
 finding a refutation.  This is similar to the greedy methods in \Cref{greedy-sec}, except that we require our resulting magma to be finite and completely
 defined, and our transformations need not be deterministic.

 Since our goal is simply to find a finite model\,---\,and any candidate can be checked directly\,---\,we are not limited to transformations that can be rigorously justified.
  SAT solvers like
 \emph{Glucose}~\cite{DBLP:conf/ijcai/AudemardS09,DBLP:conf/cp/AudemardS12} inspired by the earlier \emph{MiniSAT}~\cite{DBLP:conf/sat/EenS03},
 via convenient interfaces like \emph{PySAT}~\cite{imms-sat18, itk-sat24},
 other fast SAT solvers such as \emph{Kissat}~\cite{BiereFazekasFleuryHeisinger-SAT-Competition-2020-solvers},
 counterexample finders like \emph{Mace4}~\cite{prover9-mace4}, and more general ATPs like
 \emph{Vampire}~\cite{DBLP:conf/cav/KovacsV13} which can be used as solvers, all succeeded in finding useful magmas following this approach.

 For example, suppose our goal is to show $\Eq{854} \nmodelsfin \E$, for some $\E$.  We start with a magma table which
 satisfies both equations, and remove a random subset of the table entries (creating a partially defined magma).  We then
 ask the ATP to find a magma which satisfies $\Eq{854}$, filling in the unspecified values and potentially adding new elements.
 If in fact $\E$ is not implied by $\Eq{854}$, the ATP might succeed in finding a magma refuting the implication.
 We may also directly insert a violation of $\E$ into the magma cells we have emptied, hoping that a consistent completion still exists.
 Another method, by analogy with slowly introducing forcing terms into numerical integrations while preserving adiabatic invariants,
 is to impose selected implications of a given equation without directly enforcing the equation itself.
 This can gradually drive the magmas in the sequence toward satisfying the equation without immediately imposing it.

Combining several of these techniques allowed us to find proofs of
$$\Eq{854} \nmodelsfin \Eq{413}, \Eq{1045}, \Eq{1055}, \Eq{3316}, \Eq{3925}.$$

As an example of the process, we started with a magma that satisfied $\Eq{854}$ and all the above equations,
but had an $\Eq{10}$ violation.  We suspected that $\Eq{10}$ might be relevant (and later found an
argument showing multiple $\Eq{10}$ violations would be required in a magma which showed $\Eq{854} \nmodelsfin \Eq{1055}$),
and so used \emph{Kissat} and \emph{Vampire} to grow $\Eq{854}$ magmas with such violations.  These larger magmas were then used as
seeds for the random evolution process, and the ATPs eventually succeeded in finding magmas which satisfied $\Eq{854}$ but refuted
the implications in question.

Note that these approaches have their limitations.  To be effective, it must be easy to transition between different
states, which usually involves finding a magma satisfying the equation of interest or at least some related one.
Equations whose finite magmas are difficult to find (e.g.\ $\Eq{677}$), whether because of absolute rarity or the numerical challenges
that our ATPs have in finding them without taking advantage of a structural ansatz, appear resistant to these methods in practice.

Another way to utilize (iii), which proved useful for laws that involved the squaring operator $S$, was to adopt a ``squares first'' approach in which one selected a base model $S\Magma = (SM,\op)$ to serve as the set of squares, then extend it to a larger model $\Magma$ with carrier $M = SM \uplus N$ by first determining what the multiplication map should be on the diagonal $\{ (x,x): x\in N\}$ (i.e., to determine the squaring map $S \colon N \to SM$), together with the values on the blocks $SM \times N$, $N \times SM$, and then finally resolve the remaining values on the $N \times N$ block.  Often, versions of the greedy algorithm are useful for each of these stages of the construction.  The precise details are quite technical, particularly for the law $\x \formaleq (\y \op \y) \op ((\y \op \x) \op \y)$ \eqref{eq1729}, which was the last of the equations whose implications were settled by the ETP.  We refer the reader to the ETP blueprint for further details.

\section{Syntactic arguments}\label{syntactic-sec}

Many proofs or refutations of implications (or equivalences) between two equational laws $\E,\E'$ can be obtained from the syntactic form of the equation.  We discuss some techniques here that were useful in the ETP.

\subsection{Simple rewrites}\label{rewrite-sec}

Many equational laws $\E'$ can be formally deduced from a given law $\E$ by applying the \emph{Lean} \texttt{rw} tactic to rewrite $\E'$ repeatedly by some forward or backward application of $\E$ applied to arguments that match some portion of $\E$.  For instance, the commutative law $\Eq{43}$ clearly implies $\x \op (\y \op \z) \formaleq (\y \op \z) \op \x$ \eqref{eq4531}
by a single such rewrite.  A brute force application of such rewrite methods is already able to directly generate about $\num{15000}$ such implications, including many equivalences to the singleton law $\Eq{2}$ and the constant law $\Eq{46}$.  After applying transitive closure, this generates about four million further such implications.

A simple observation that already generates a reasonable number of equivalences is that any equation of the form $\x \formaleq f(\y,\z,\dots)$ necessarily is equivalent to the trivial law $\x \formaleq \y$, by transitivity; similarly, an equation of the form $f(\x,\y) \formaleq g(\z,\w,\dots)$ implies $f(\x,\y) \formaleq f(\x',\y')$; and so forth.  Equivalences of this form were useful early in the project by cutting down the number of distinct equivalence classes of laws that needed to be studied.

\subsection{Matching invariants}

Fix an alphabet $X$. A \emph{matching invariant} is an assignment $I \colon M_X \to {\mathcal I}$ of an object $I(w) \in {\mathcal I}$ in some space ${\mathcal I}$ to each word $w \in M_X$ with the property that if an equational law $w_1 \formaleq w_2$ has matching invariants $I(w_1)=I(w_2)$, then the same matching $I(w'_1) = I(w'_2)$ holds for any consequence $w'_1 \formaleq w'_2$.  In particular, if one has $I(w_1)=I(w_2)$ and $I(w'_1) \neq I(w'_2)$, then the law $w_1 \formaleq w_2$ does not imply the law $w'_1 \formaleq w'_2$.

A simple example of a matching invariant is the multiplicity $(n_x)_{x \in X}$ of variables of a word: if $w_1,w_2$ have all variables $x$ appear the same number of times $n_x$ in both words, then any rewriting of a word $w$ using the law $w_1 \formaleq w_2$ will preserve this property.  Hence, if $w'_1, w'_2$ do not have that each variable appear the same number of times in both words, then $w_1 \formaleq w_2$ cannot imply $w'_1 \formaleq w'_2$.  For instance, the commutative law $\Eq{43}$ cannot imply the left-absorptive law $\Eq{4}$.

One source of matching invariants comes from the free magma $\Magma_{X,\Gamma}$ of a theory:

\begin{proposition}[Free magmas and matching invariants]\label{free-inv}  Let $\iota_{X,\Gamma} \colon X \to M_{X,\Gamma}$ be the map associated to the free magma $\Magma_{X,\Gamma}$ for a theory $\Gamma$.  Then the map $I \colon M_X \to M_{X,\Gamma}$ defined by $I(w) \coloneqq \varphi_{\iota_{X,\Gamma}}(w)$ is an invariant.
\end{proposition}

\begin{proof}  Suppose that $w_1 \formaleq w_2$ entails $w'_1 \formaleq w'_2$, and that $I(w_1) = I(w_2)$.  For any $f \colon X \to \Magma_{X,\Gamma}$, the two maps $\varphi_f, \varphi_{f,\Gamma} \circ \varphi_{\iota_{X,\Gamma}} \colon \Magma_X \to \Magma_{X,\Gamma}$ are both homomorphisms that extend $f$, hence agree by the universal property of $\Magma_X$, as displayed by the following commutative diagram:
\[\begin{tikzcd}
	&& X \\
	\\
	{\Magma_X} && {\Magma_{X,\Gamma}} && {\Magma_{X,\Gamma}}
	\arrow[hook, from=1-3, to=3-1]
	\arrow["{\iota_{X,\Gamma}}", pos=0.7, left, from=1-3, to=3-3]
	\arrow["f", pos=0.7, above, from=1-3, to=3-5]
	\arrow["{I = \varphi_{\iota_{X,\Gamma}}}", pos=0.6, above, from=3-1, to=3-3]
	\arrow["{\varphi_f}"', curve={height=18pt}, pos=0.6, below, from=3-1, to=3-5]
	\arrow["{\varphi_{f,\Gamma}}", pos=0.5, above, from=3-3, to=3-5]
\end{tikzcd}\]
In particular, the hypothesis $I(w_1)=I(w_2)$ implies that $\varphi_f(w_1) = \varphi_f(w_2)$ for all $f \colon X \to M_{X,\Gamma}$; that is to say, the magma $\Magma_{X,\Gamma}$ satisfies the law $w_1 \formaleq w_2$, and hence also $w'_1 \formaleq w'_2$ by hypothesis.  Thus $\varphi_{\iota_{X,\Gamma}}(w'_1) = \varphi_{\iota_{X,\Gamma}}(w'_2)$, which gives $I(w'_1) = I(w'_2)$ as required.
\end{proof}

\begin{example}  If we take $\Gamma = \{\Eq{4}\}$ to be the theory of the left-absorptive law $\Eq{4}$ as described in \Cref{left-absorb}, then the matching invariant $I(w)$ produced by \Cref{free-inv} is the leftmost letter of the alphabet $X$ appearing in the word; for instance $I((\x \op \y) \op \z) = \x$.  Thus, for example, the left-absorptive law $\Eq{4}$ cannot imply the right-absorptive law $\Eq{5}$.
\end{example}

\begin{example}  If we take $\Gamma = \{\Eq{43}, \Eq{4512}\}$ to be the theory of the commutative law $\Eq{43}$ and the associative law $\Eq{4512}$, then by \Cref{semi-group}, the associated invariant $I(w) = \sum_{x \in X} n_x e_x$ is the formal sum of all the generators $e_x$ appearing in the word $w$, in the free abelian semigroup generated by those generators.  This recovers the preceding observation that the multiplicities $(n_x)_{x \in X}$ form a matching invariant.
\end{example}

\begin{example}  Let $n \geq 1$ be a positive integer, and consider the theory $\Gamma = \{\Eq{43}, \Eq{4512}, \E_n\}$ consisting of the previous theory $\{\Eq{43}, \Eq{4512}\}$ together with the order-$n$ law $L_x^n y = y$.  One can check that the free magma $\Magma_{X,\Gamma}$ can be described as the free abelian group of exponent $n$ with generators $e_x, x \in X$, with associated map $\iota_{X,\Gamma} \colon x \mapsto e_x$.  The associated matching invariant $I(w) = \sum_{x \in X} n_x e_x$ is essentially the multiplicities $(n_x \hbox{ mod } n)_{x \in X}$ modulo $n$, which gives a slightly stronger criterion than the preceding matching invariant for refuting implications.  For example, the cubic idempotent law $\x \formaleq (\x \op \x) \op \x$ \eqref{eq23}
has matching invariants $e_x = 3e_x$ in the $n=2$ case, and hence does not imply the idempotent law $\x \formaleq \x \op \x$ \eqref{eq3} since $e_x \neq 2e_x$ in the $n=2$ case.
\end{example}

In practice, we found that these invariants could be used to establish a significant fraction of the non-implications in the implication graph, although in most cases these non-implications could also be established by other means, for instance through consideration of small finite counterexamples, especially small models of~$\Gamma$.

\begin{remark}  One can also obtain matching invariants from the free objects associated to theories that involve additional operations beyond the magma operation $\op$, such as an identity element or an inverse operation.  We leave the precise generalization of \Cref{free-inv} to such theories to the interested reader.
\end{remark}

\subsection{Canonization}\label{canon-sec}

One possible way of obtaining refutations of a given implication $\E \models \E'$ between equational laws is by building a special kind of syntactic model, via certain involutions on elements of the free magma $\Magma_X$ we call \emph{canonizers}.

\begin{definition}
  A function $C\colon M_X\rightarrow M_X$ is a \emph{canonizer} for an equation $\E$ if
  \setlength{\parskip}{0pt}
  \begin{enumerate}
    \item $C$ is computable.
    \item if $w_1\sim_\E w_2$, then $C(w_1) = C(w_2)$.
  \end{enumerate}
\end{definition}

In fact, a canonizer is simply a matching invariant with target in $M_X$.

We describe a concrete strategy for building such $C$s, which will require a number of definitions.

\begin{definition}\label{def:canon}
Let $R \colon M_X \to M_X$ be an arbitrary function.
\setlength{\parskip}{0pt}
\begin{itemize}
\item  We say that $R$ is \emph{weakly collapsing} if for every word $w$, $R(w)$ is a (not necessarily proper) sub-word of $w$.
\item  A function $\theta\colon X\rightarrow M_X$ is called a \emph{substitution}, and we can extend $\theta$ to an endomorphism~$\varphi_\theta \colon \Magma_X\rightarrow\Magma_X$ in the usual way. We write $\varphi_\theta w$ instead of $\varphi_\theta(w)$ for application of substitutions.
\item  If $\E$ is the equation $l\simeq r$, and $l$ is not a variable, we say that $R$ is a \emph{weak canonizer} if for any substitution $\theta$, $R(\varphi_\theta l)=\varphi_\theta r$.
\item  Finally we say that $R$ is \emph{non-overlapping} for $\E$ if for every word $w\in M_X$ which is a strict sub-word of $l$ that is not in $X$, and any substitution $\theta$, $R(\varphi_\theta w) = \varphi_\theta w$.
\end{itemize}
  We can then define $C_R\colon M_X\rightarrow M_X$ as follows:
\begin{align}
  C_R(\x) &\coloneqq \x\\
  C_R(w\op w') &\coloneqq R(C_R(w)\op C_R(w'))
\end{align}
\end{definition}

The following lemma is readily proven by induction on the structure of terms.

\begin{lemma}\label{lem:subsubst}
  If $t$ is a strict sub-word of $l$, and $R$ is non-overlapping, then $C_R(\varphi_\theta t) = \varphi_{C_R\circ \theta} t$.
\end{lemma}
We now have the following theorem.
\begin{theorem}\label{thm:canon}
  Whenever $R$ is weakly collapsing, a weak canonizer, and non-overlapping, then $C_R$ is a canonizer.
\end{theorem}
\begin{proof}
  Assume $w\sim_{\E}w'$. We proceed by induction over the proof of equality.

  The only nontrivial case is $w = \varphi_\theta l$ and $w'=\varphi_\theta r$ for some substitution $\theta$.  The case $l=r$ is clear, so we may assume that $l \neq r$.
  Then we have
  \begin{align*}
    C_R(\varphi_\theta l) &= R(\varphi_{C_R\circ\theta}l) \\
                 &= \varphi_{C_R\circ\theta}r\\
                 &= C_R(\varphi_\theta r)
  \end{align*}
  where the first and third equalities follow from the Lemma, observing that $r$ is a strict sub-term of $l$ (since $R$ is weakly collapsing and a weak canonizer), and the second from weak canonicity.
\end{proof}

We mention an example to show why this is a useful theorem. Take $\E$ to be the equation
\[
  \y \op (\x \op (\y \op (\y \op \y)))\formaleq \x.
\]
We can take $R$ to be the transformation which sends a term of the form $w \op (v \op (w \op (w \op w)))$ to $v$ for any two words $v, w$, and leaves all other words unchanged.
It is then somewhat easy to show that this transformation satisfies the conditions of \Cref{thm:canon}, and so we have a canonizer $C_R$.
This can be used, e.g.\@ to refute the implication of $\x \formaleq (\x \op \x) \op (\x \op (\x \op \x))$ from $\E$.

As a conclusion for this section, we note that a very general strategy for building canonizers comes from the theory of \emph{rewrite systems}, see e.g.\@ Baader \& Nipkow \cite{term-rewriting}.
In that setting one defines rewriting as a transformation on words (or terms), and if this transformation is \emph{terminating} and \emph{confluent} (intuitively, rewrites cannot go on forever, or diverge forever), then one may simply pick the transformation which sends a term to its normal form as a canonizer.

Though we note that the non-overlapping criterion seems very similar to the notion of orthogonality in rewriting, we leave the investigation of the precise relationship of the classical theory with the above technique as future work.

\subsection{Unique factorization}

In general, the free magma $\Magma_{X,\E}$ for a given equational law $\E$, which we can canonically define as $\Magma_X/\!\sim_\E$, is hard to describe explicitly; indeed, from the undecidability of implications between equational laws, such a magma cannot be computably described for arbitrary $\E$.  Nevertheless, for some laws it is possible to obtain some partial understanding of $\Magma_{X,\E}$ from a syntactic perspective.  For instance, if we can refute the equivalence $w'_1 \sim_\E w'_2$ by constructing a counterexample magma $\Magma$ that satisfies $\E$ but not $w'_1 \formaleq w'_2$, then this implies that the representatives $\iota_{X,\E}(w'_1), \iota_{X,\E}(w'_2)$ of  $w'_1, w'_2$ in $\Magma_{X,\E}$ are distinct.

We illustrate this approach with equations $\E$ of the left-absorptive form
\begin{equation}\label{left-absorptive}
\x \formaleq \x \op f(\x,\y,\z)
\end{equation}
for some word $f(\x,\y,\z)$, that are also known to imply the right-idempotent law $\Eq{378}$.  An illustrative example is the law $\Eq{854}$ depicted in \Cref{fig:854}. Other examples are listed in \Cref{fig:854-like}.

\begin{figure}
  \centering
  \includegraphics[width=0.85\textwidth]{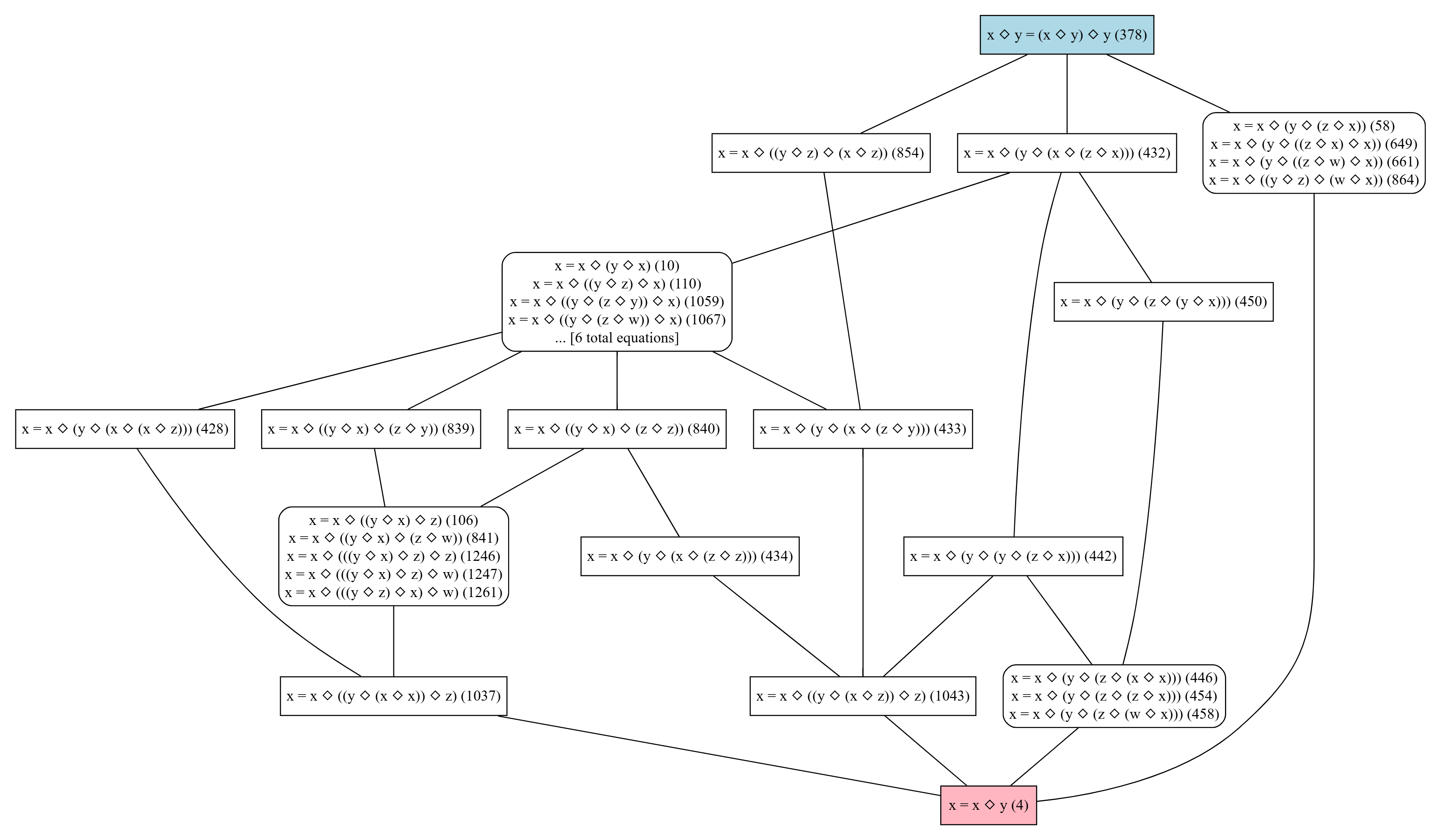}
  \caption{Equations similar to $\Eq{854}$ that are of the form \eqref{left-absorptive} (possibly involving a fourth indeterminate $\w$) and imply $\Eq{378}$.  For brevity, $70$ equations equivalent to $\Eq{4}$ have been omitted.}
  \label{fig:854-like}
  \end{figure}

\begin{lemma}\label{854} Equation $\Eq{854}$ is of the form \eqref{left-absorptive} and implies $\Eq{378}$.
\end{lemma}

\begin{proof}  Clearly we have \eqref{left-absorptive} with $f(\x,\y,\z) \coloneqq (\y \op \z) \op (\x \op \z)$.  From \eqref{left-absorptive} we have in any magma satisfying $\Eq{854}$ that
$$x = x \op f(x,x,S^2 x) = x \op S(x \op S^2 x) = x \op S(x \op f(x,x,x)) = x \op Sx.$$
This implies from a further application of \eqref{left-absorptive} that
$$ y = y \op f(y,x,y) = (y \op Sy) \op ((x \op y) \op Sy) = f(x \op y, y, Sy)$$
and hence by \eqref{left-absorptive} again
$$ (x \op y) \op y = x \op y$$
giving $\Eq{378}$.
\end{proof}

Let $\E$ be a law of the form \eqref{left-absorptive} that implies $\Eq{378}$. We define a directed graph $\to_\E$ on words in $\Magma_X$ by declaring $w' \to_\E w$ if $w \sim_\E w'' \op w'$ for some $w' \in M_X$.  By $\Eq{378}$ (applied to the quotient magma $\Magma_{X,\E} = \Magma_X/\sim_\E$), this is equivalent to requiring that $w \sim_\E w \op w'$. In particular, from \eqref{left-absorptive} we have $f(x,y,z) \to x$ for all $x,y,z$.  Furthermore, the relation $\to_\E$ factors through $\sim_\E$: if $w \sim_\E \tilde w$ and $w' \sim_\E \tilde w'$, then $w' \to_\E w$ if and only if $\tilde w' \to_\E \tilde w$.

Call a word $w \in M_X$ \emph{irreducible} if it is not of the form $w = w_1 \op w_2$ with $w_2 \to_\E w_1$.  We can partially understand the equivalence relation $\sim_\E$ on irreducible words:

\begin{theorem}[Description of equivalence]\label{irred-desc}  Let $\E$ be an equation of the form \eqref{left-absorptive}.  Let $w$ be an irreducible word, and let $w'$ be a word with $w \sim_\E w'$.
  \setlength{\parskip}{0pt}
  \begin{itemize}
    \item[(i)] If $w$ is a product $w = w_1 \op w_2$, then $w'$ takes the form
$$ w' = ({}\dots((w'_1 \op w'_2) \op v_1) \op \dots{}) \op v_n$$
for some $w'_1 \sim_\E w_1$, $w'_2 \sim_\E w_2$, some $n \geq 0$, and some words $v_1, \dots, v_n$ such that for all $0 \leq i < n$, $v_{i+1}$ is of the form
$$ v_{i+1} \sim_\E f(x_i,y_i,z_i)$$
for some $x_i, y_i, z_i$ with
$$ x_i \sim_\E ({}\dots((w'_1 \op w'_2) \op v_1) \op \dots{}) \op v_i.$$
In particular, $v_{i+1} \to_\E x_i$.
  \item[(ii)] Similarly, if $w \in X$ is a generator of $\Magma_X$, then $w'$ takes the form
$$ w' = ({}\dots(w \op v_1) \op \dots {}) \op v_n$$
for some $n \geq 0$, and some words $v_1, \dots, v_n$ such that for all $0 \leq i < n$, $v_{i+1}$ is of the form
$$ v_{i+1} \sim_\E f(x_i,y_i,z_i)$$
for some $x_i, y_i, z_i$ with
$$ x_i \sim_\E ({}\dots(w \op v_1) \op \dots{}) \op v_i.$$
In particular, $v_{i+1} \to_\E x_i$.
\end{itemize}
Conversely, any word of the above forms is equivalent to $w$.
\end{theorem}

\begin{proof}  We just verify claim (i), as claim (ii) is similar.  The converse direction is clear from \eqref{left-absorptive} (after quotienting by $\sim_\E$), so it suffices to prove the forward claim. By the Birkhoff completeness theorem, it suffices to prove that the class of words described by (i) is preserved by any term rewriting operation, in which a term in the word is replaced by an equivalent term using \eqref{left-absorptive}.  Clearly the term being rewritten is in $w'_1$ or $w'_2$ then the form of the word is preserved, and similarly if the term being rewritten is in one of the $v_i$.  The only remaining case is if we are rewriting a term of the form
$$ x_i = ({}\dots((w'_1 \op w'_2) \op v_1) \op \dots{}) \op v_i.$$
If $i>0$ we can rewrite this term down to $x_{i-1}$, and this still preserves the required form (decrementing $n$ by one).  If $i=0$ then we cannot perform such a rewriting because of the irreducibility of $w_1 \op w_2$ and hence $w'_1 \op w'_2$.  Finally, we can rewrite $x_i$ to $x_i \op v$ where $v$ is of the form
$$ v_i = f(x_i,y,z),$$
and after some relabeling we are again of the required form (now incrementing $n$ by one). This covers all possible term rewriting operations, giving the claim.
\end{proof}

Specializing to the case where $w,w'$ are both irreducible, we conclude

\begin{corollary}[Unique factorization]\label{unique-factorization}  Two irreducible words $w, w'$ are equivalent if and only if they are either the same generator of $X$, or are of the form $w = w_1 \op w_2$, $w' = w'_1 \op w'_2$ with $w_1 \sim_\E w'_1$ and $w_2 \sim_\E w'_2$.
\end{corollary}

As an application of this corollary, we establish

\begin{proposition}[$\Eq{854} \nmodels \Eq{3316}$]\label{854-3316} Equation $\Eq{854}$ does not imply $\Eq{3316}$.
\end{proposition}

\begin{proof}[Proof sketch]
  We work in the free magma $\Magma_X$ on two generators $X = \{\x,\y\}$.  It suffices to show that
$$  \x \op \y \not \sim_{\Eq{854}} \x \op (\y \op (\x \op \y)).$$
Suppose this were not the case, then by \Cref{unique-factorization} one of the following statements must hold:
\begin{itemize}
\item[(i)] $\y \to_{\Eq{854}} \x$.
\item[(ii)] $(\y \op (\x \op \y)) \to_{\Eq{854}} \x$.
\item[(iii)] $\y \op (\x \op \y) \sim_{\Eq{854}} \y$.
\end{itemize}
If (i) holds, then we have $\x \op \y = \x$ must hold in $\Magma_X/\sim_\E$, hence $\Eq{854}$ would imply $\Eq{4}$.  However, the magma $\mathbb{Z}/2\mathbb{Z}$ with $x \op y = x(1-y)$ refutes this implication.

Similarly, if (iii) held, then $\Eq{854}$ would have to imply $\Eq{10}$, but this can be refuted by a finite magma with $11$~elements.

Finally, if (ii) held, then the claim
$$  \x \op \y \sim_{\Eq{854}} \x \op (\y \op (\x \op \y))$$
to refute simplifies to
$$  \x \op \y \sim_{\Eq{854}} \x$$
and we are back to (i), which we already know not to be the case.
\end{proof}

\section{Proof automation}\label{automated-sec}

In this project we used proof automation in two ways: automated theorem provers (ATPs) and \emph{Lean} tactics.
ATPs are generally stand-alone tools that implement a (semi-)decision procedure for a given formal language or related set of languages.
For example, \emph{Vampire}~\cite{DBLP:conf/cav/KovacsV13,TheVampireDiary} is an ATP focused primarily on first-order logic using superposition, which we used extensively in this project.  We also made extensive use of \emph{Prover9} and \emph{Mace4}~\cite{prover9-mace4}.

ATPs are complex software that can contain bugs.
Instead of trusting ATP output, we used proof certificates, which many ATPs can produce, to reconstruct proofs in \emph{Lean}.
The details of proof reconstruction depend on the form of the proof certificate produced by the ATP\@.
We expand on this in \Cref{sec:proof-reconstruction}.

Tactics in \emph{Lean}, on the other hand, are meta-programs~\cite{DBLP:journals/pacmpl/EbnerURAM17} that build proofs.
In other words, they essentially take \emph{Lean} code as input and produce \emph{Lean} code as output.
In this manner, they look like another keyword in the language, and are tightly integrated by producing proofs directly.
Under the hood, their implementation can be arbitrarily complex, from syntactic sugar to full decision procedures.
The \texttt{duper} tactic~\cite{DBLP:conf/itp/CluneQBA24}, for example, implements a superposition calculus similar to \emph{Vampire}'s, but for dependent types\,---\,\emph{Lean}'s underlying logical foundation.

In the rest of this section we describe the different proof automation techniques used in this project.
We first discuss the different proof methods used: primarily superposition and equational reasoning; we then discuss the integration in \emph{Lean}, and finally we report some basic empirical results from this project.

\subsection{Proof techniques}

The two main families of ATPs and tactics we used are based on superposition/saturation and equational reasoning.
In this context we also include SMT solvers, which combine specific decision procedures for theories, like congruence closure for equational reasoning, with satisfiability (SAT) solving \cite{deMoura-Bjorner-2009}.
Finally, we also used \texttt{aesop}~\cite{DBLP:conf/cpp/LimpergF23}, which implements a version of tableau search.
This was used mainly to help specific constructions in refutations, and is not specific to proving or disproving magma implications in this sense.
We describe our use of \texttt{aesop} in \Cref{sec:proof-reconstruction} below.

\subsubsection{Saturation}
Most of the ATPs used extensively in this project rely primarily on saturation procedures in the superposition calculus.
For example, this is the case for \emph{Vampire}~\cite{DBLP:conf/cav/KovacsV13}.\footnote{See also~\cite{DBLP:journals/cacm/BentkampBNTVW23} for a gentler exposition.}
The core idea of these provers is that they take a set of assumptions and a conjecture, expressed in (say) first-order logic.
The conjecture is negated and added to the set of assumptions, which are all put into a normal form.
The ATP then tries to refute the negation by applying rules of an underlying calculus, until a proof of false (a contradiction) is derived.
In this case, the conjecture was (classically) true, and the ATP has found a proof by contradiction, often called a ``refutation'' or ``saturation'' proof.

The underlying calculi vary from system to system, but they often have a variant of a resolution clause of the form:
\[\infer{C \lor D}{C \lor L \quad D \lor \neg L} \]
This can be read as $C \lor L$ with $D \lor \neg L$ implies $C \lor D$, where $C, D, L$ are formulas in e.g. first-order logic.
Superposition calculi have a variant of this rule that deals with equality directly, and thus are more efficient at reasoning about equality.

In this project we used \emph{Vampire}~\cite{DBLP:conf/cav/KovacsV13}, \emph{Duper}~\cite{DBLP:conf/itp/CluneQBA24} and \emph{Prover9} and \emph{Mace4}~\cite{prover9-mace4} which are all based on variants of saturation for proving.

\subsubsection{Equational Reasoning} As already discussed in \Cref{canon-sec}, equational reasoning is a type of reasoning that is based\footnote{More precisely, one can formalize this reasoning using Birkhoff's five rules of inference (reflexive, symmetric, transitive, replacement, and substitution); see, e.g., \cite{burris}.} on equational logic and rewriting with congruence~\cite{term-rewriting}.
In general, an equational reasoning procedure takes a series of equations and tries to determine whether another equation can be deduced from it.
A core tool in equational reasoning are \emph{e-graphs}, a data structure used to represent congruence classes of terms.
By themselves, e-graphs provide an efficient means of implementing a decision procedure for congruence closure over ground equations (i.e., equations without variables).
Extensions to this procedure, for example by quantifier instantiation via e-matching \cite{DBLP:conf/cade/MouraB07}, also allow for a semi-decision procedure for congruence closure over non-ground equations.

SMT solvers like \emph{Z3}~\cite{DBLP:conf/tacas/MouraB08} use equational reasoning for deciding the theory of equality with uninterpreted functions~\cite{DBLP:series/txtcs/KroeningS16,DBLP:conf/cade/MouraB07}.
On the other hand, equality saturation~\cite{DBLP:journals/pacmpl/WillseyNWFTP21} uses e-graphs by extending congruence closure to a more controlled search, enabling optimization and conditional rewriting.
One of the main advantages of equational reasoning for implications of magma laws is that we get very explicit proofs: a proof that $l \models l'$ is given by a sequence of rewrites that starts at the left-hand side of $l'$ and arrives at the right-hand side through applications of $l$.

In this project we used \emph{Z3}~\cite{DBLP:conf/tacas/MouraB08}, \emph{Prover9} and \emph{Mace4}~\cite{prover9-mace4}, a custom ATP \emph{MagmaEgg} for magmas based on \emph{egg}~\cite{DBLP:journals/pacmpl/WillseyNWFTP21}, and the \emph{Lean} \texttt{egg} tactic~\cite{DBLP:journals/pacmpl/KoehlerGBGTS24,egg_tactic}, which all work with equational logic. We have also reasoned with manual (custom written) heuristics about simple rewrites.

\subsection{Integration of automation procedures}
\label{sec:proof-reconstruction}

While ATPs are very useful for proving theorems in this project, they do not integrate with \emph{Lean} out of the box.
ATPs may produce unsound proofs, or worse, derive incorrect results.
Thus, by default, theorems in \emph{Lean} cannot be proven by deferring to the result of an ATP\@.
Instead, the results of an ATP can be used to reconstruct a proof of the form required by \emph{Lean}.
Thus, in general, integration of ATPs requires two steps.
First, there is the invocation of the ATPs by translating the problem from \emph{Lean} into the languages and logics they use.
And second, there is the reconstruction of the ATPs' results as a (persistent) \emph{Lean} proof.
These two aspects present different challenges, and require different strategies, depending mostly on the kind of proof strategy the ATP uses.

More generally, we have observed that there are multiple ways of integrating decision procedures within \emph{Lean}, with different levels of integration.

\begin{enumerate}
    \item Using a \emph{Lean} tactic, which calls a decision procedure written in \emph{Lean} (like \texttt{aesop} or \texttt{duper}).
    \item\label{inter} Using a \emph{Lean} tactic, which calls an existing (external) ATP and reconstructs a proof term from the ATP's result (like \texttt{bv\_decide} or \texttt{egg}).
    \item\label{external} Using an external script which calls an existing ATP and generates a source file \texttt{.lean} which captures the result explicitly.
\end{enumerate}

This project primarily used the least integrated approach, Option~\ref{external}, as it was the fastest to implement and imposed no additional technical requirements on other contributors.
The matter of technical requirements caused problems, for example, when integrating the \texttt{egg} tactic (Option~\ref{inter}) as it initially expected certain software on the user's machine.
Such trade-offs between Option~\ref{inter} and Option~\ref{external} are, however, mutual, as the higher upfront cost of integrating a proof tactic in Option~\ref{inter} makes the decision procedure easier to use than with Option~\ref{external}.
Additionally, Option~\ref{inter} can benefit from \emph{Lean}'s meta-programming capabilities when encoding the problem for use with an ATP, and when reconstructing a \emph{Lean} proof from the result.

\subsubsection{Proof Reconstruction}
The relative simplicity of the objects used in this project benefits the implementations of proof reconstruction.
By focusing on the given problem domain, difficult reconstruction issues, like complex dependent types, could be ignored.

For saturation proofs with \emph{Vampire}, we implemented analogs of the \emph{superpose}, \emph{resolve}, and \emph{subsumption} steps in \emph{Lean}.
Proofs can then be reconstructed as sequences of these steps (and additional technicalities) as shown in \Cref{fig:vampire-example}.

\begin{figure}
\centering
\begin{lean}
@[equational_result] theorem Equation999_implies_Equation86
    (G : Type*) [Magma G] (h : Equation999 G) : Equation86 G := by
  by_contra nh
  simp only [not_forall] at nh
  obtain ⟨sK0, sK1, sK2, nh⟩ := nh
  have eq9 (X0 X1 X2 X3 : G) : (X1 ◇ ((X2 ◇ X3) ◇ (X1 ◇ X0))) = X0 :=
    mod_symm (h ..)
  have eq10 : sK0 ≠ (sK1 ◇ (sK2 ◇ (sK1 ◇ sK0))) := mod_symm nh
  have eq11 (X3 X4 X5 : G) : (X4 ◇ (X3 ◇ (X4 ◇ X5))) = X5 := superpose eq9 eq9
  have eq21 : sK0 ≠ sK0 := superpose eq11 eq10
  subsumption eq21 rfl
\end{lean}
\caption{Example of a proof reconstructed from output of \emph{Vampire}. Note how the proof proceeds by contradiction and uses the \texttt{superpose} and \texttt{subsumption} steps implemented in \emph{Lean}.}
\label{fig:vampire-example}
\end{figure}

For equational proofs from external provers, like \emph{MagmaEgg}, we also used a tailored version of reconstruction.
Specifically, the \emph{MagmaEgg} implementation turns \emph{explanations} \cite{nieuwenhuis2005proof} from \emph{egg} into \emph{Lean} proofs by simple applications of the defining properties of equality as shown in \Cref{fig:magma-egg-example}.

\begin{figure}
\centering
\begin{lean}
@[equational_result] theorem Equation3973_implies_Equation4023
    (G : Type _) [Magma G] (h : Equation3973 G) : Equation4023 G := fun x y z =>
  let v0 := M z x
  let v1 := M z v0
  let v2 := M v1 y
  let v3 := M z v1
  have h4 := R v2
  have h5 := R z
  have h6 := h z x v1
  let v7 := M x (M v1 z)
  T (T (T (T (h x y z) (h (M y v0) z v2)) (C (C h5 (C h4 (T (T (h y v0 v2)
  (C (C (T (T (T h6 (h v7 v1 v0)) (C (C (R v1) (T (h v0 v7 z) (C (S h6) h5)))
  (R v0))) (S (h z v1 v0))) (R (M v2 y))) h4)) (S (h y v3 v2))))) h4))
  (S (h (M y v3) z v2))) (S (h v1 y z))
\end{lean}
\caption{Example of a proof reconstructed by \emph{MagmaEgg}. Note the proof only uses reflexivity, symmetry, transitivity, and congruence of equality.}
\label{fig:magma-egg-example}
\end{figure}

In the case of the \texttt{egg} tactic, which also reconstructs proofs from \emph{egg} explanations, the proof could be converted into a more human-readable form by using the \texttt{calcify}\footnote{\url{https://github.com/nomeata/lean-calcify}} tactic, as shown in \Cref{fig:egg-example}

\begin{figure}
  \centering
  \includegraphics[width=\textwidth]{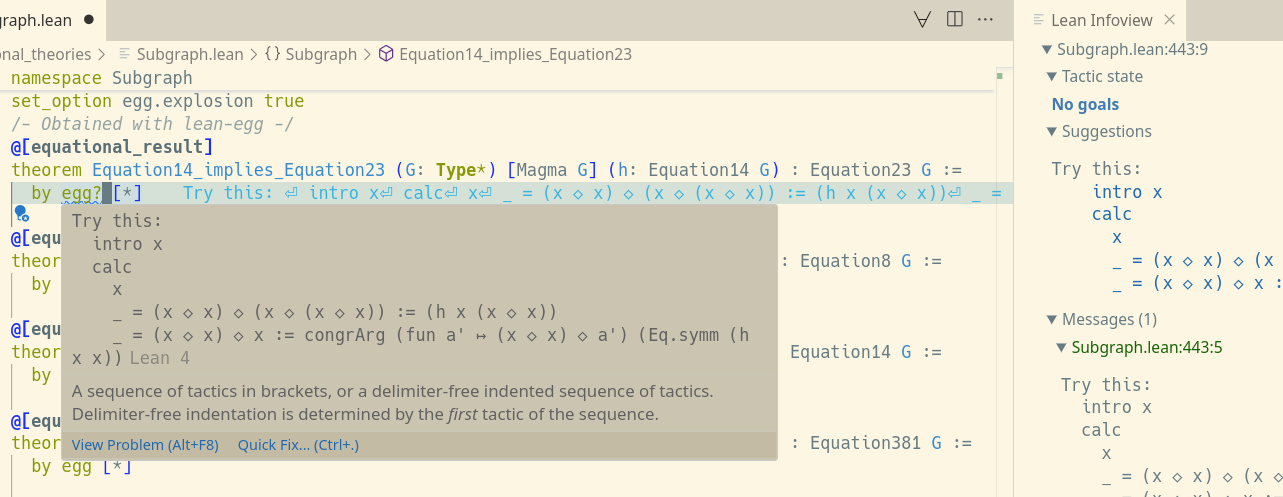}
  \caption{Example of the \texttt{egg} tactic reconstructing a proof in human-readable form with the help of \texttt{calcify} (invoked by the special syntax \texttt{egg?}).}
  \label{fig:egg-example}
\end{figure}

\subsubsection{Semi-Automated Counterexample Guidance}  Another use of ATPs has been in a semi-automatic fashion, to find counterexamples.
The general strategy was to use ATPs to find counterexamples to implications by building magmas iteratively.
If we want to build a counterexample to $l \models l'$, we want to construct a magma where $l$ holds but $l'$ does not.
In this method, we iteratively strengthen a construction with additional hypotheses, and use the ATP to check whether these hypotheses are not too strong (to imply $l'$) or unsound (to disallow $l$).


While equational reasoning can also be used in a semi-automatic fashion to prove equations~\cite{DBLP:journals/pacmpl/KoehlerGBGTS24}, the positive implications in the main implication graph of the project were all simple enough that we did not need a semi-automatic approach for them.

\subsection{ATP usage}\label{ATP_usage}

When the project started, contributors had varying degrees of ATP knowledge and skills, with some of us having to start using them from the ground up. With hindsight, several of the project computations could have been approached in better ways, their difficulty diminished due to better ATP expertise. Accordingly, in this section\footnote{Consult the ETP site for a substantially expanded version of these notes.} we provide some facts and hints about ATP usage, primarily with the algebraist working with (unsorted) equational theories in mind, based on our experience and available evidence\footnote{The timings presented here cannot be taken as benchmarks. Different experiments were executed on different machines, with heterogeneous software and OS environments (Ubuntu, Windows 10...), with varying numbers of parallel processes\,---\,both internal and external to the experiment\,---\,and, in many cases, under the Linux \Verb|nice| command.} within the ETP\@. We restrict our attention to \emph{Vampire} and \emph{Prover9}-\emph{Mace4}, which were the main tools used for exploration of implications and anti-implications along the ETP\@.

\subsubsection{Employing several ATPs}
In general, given a batch of problems of interest, it is useful to employ several ATPs when they complement each other on the batch\,---\,i.e., when they can solve different sets of problems. This is often the case with \emph{Vampire} and \emph{Prover9}: e.g., in a 2024 study~\cite{LPAR2024:Prover9_Unleashed_Automated_Configuration}, it is shown that from a batch of around 770 TPTP problems solved with \emph{Vampire} and \emph{Prover9} together (with some restrictions), around $60\%$ are solved by both, $20\%$ are solved only by \emph{Vampire}, and the remaining $20\%$ are solved only by \emph{Prover9}.\footnote{In contrast, in said study the E~prover solves only a single problem not solved by \emph{Vampire}.}

Although older and virtually discontinued, \emph{Prover9} is still very useful with equational logic and algebraic problems. Within the ETP we likewise identified several problems that were easier to solve with \emph{Prover9} than with \emph{Vampire}, and even some solved by \emph{Prover9} but not by \emph{Vampire} (with the configurations we tried). The most salient example is the proof that $\Eq{102744082}$ implies the injectivity of the right multiplication map, used in the Higman-Neumann side project (see \Cref{higman-neumann}), which was originally found by \emph{Prover9} in several hours with parameters chosen to produce a big search space; an optimized choice of \emph{Prover9} options lowers the runtime to 0.2s. Upon contact with the \emph{Vampire} development team, after several attempts they were able to provide a \emph{Vampire} configuration (inspired by the \emph{Prover9} optimization) which produces a proof in 3s.

In addition, \emph{Mace4}'s non-SAT algorithm generally makes it faster at finding models than generic SAT-based finite model builders, such as those implemented in \emph{Vampire}. For ETP problems in particular, a well-configured \emph{Mace4} is faster than \emph{Vampire}'s finite model builder in both the task of finding a model of a given size and that of exhausting a size with no models. For example, for $\Eq{677}$ models, \emph{Mace4} is able to exhaust size 8 in 1.5s, and to find a model of size 9 in 16s, while \emph{Vampire}'s finite model builder is unable to perform any of these tasks in 10 minutes. By contrast, \emph{Vampire}'s \verb|casc_sat| mode can be quicker at finding \emph{some} model of some size for a given problem.

\subsubsection{Basic usage}
When running an ATP or finite model builder on a nontrivial problem, the two key recommendations are: 1) Conduct several runs with different search-parameter settings, and 2) provide enough computational resources for each search: memory\footnote{If a memory limit is not specified, \emph{Vampire} will try to use as much RAM as possible, but by default \emph{Prover9}-\emph{Mace4} sets memory limits which are rather low for current standards, so we advise raising them before starting a long computation.}, number of processing cores\footnote{\emph{Vampire} can make use of as many cores as the user specifies. To take advantage of multicore processors with \emph{Prover9}-\emph{Mace4}, one can run several terminal or GUI instances. Moreover, each GUI window allows to run \emph{Prover9} and \emph{Mace4} simultaneously on the same problem, and they run on different cores.}, number of user instructions executed, and especially, time. Along the ETP, some proofs or models that could not be found in under 1 second with a given configuration, could be obtained in under (say) 8 seconds without altering the configuration; other implications initially required runs lasting several minutes, or even hours. Once a proof is found by some configuration, a short and quick proof can typically be found by tweaking that configuration.

The configuration of the search parameters is a difficult art that has become more sophisticated as ATPs have grown in complexity. To address difficult problems, we strongly recommend acquiring a solid understanding of the available configuration options and their effects, enabling the search space to be tailored as closely as possible to the problem at hand. Currently, options allow control over numerous aspects of each proving stage, including search limits, preprocessing steps, inference rules, formulas ordering and weighting, etc.

\subsubsection{Flow control}\label{flow control}
\emph{Vampire} is equipped with a user-friendly and powerful standard mode, the CASC mode (or CASC SAT mode for finite model building) that in many situations avoids the need of configuring specific search options for the given problem. This mode invokes a sequence of strategies (different option configurations) with assigned time limits. The time given to each strategy is important: counterintuitively, giving less time to the CASC mode may end up producing a faster proof: there are proofs that can be found when the time limit is set to less than $5\%$ of the time \emph{Vampire} needs to find a proof without the time limit~\cite{DBLP:conf/cav/KovacsV13}. At least two factors contribute to this effect: the correct strategy is explored sooner because less time is spent on each strategy, and the behaviour of \emph{Vampire}'s default saturation algorithm, the limited resource strategy (LRS) algorithm.

On the other hand, \emph{Prover9} itself does not have the capability of running several schedules in a row, although it lets the user implement some rudimentary strategies through commands called \emph{actions}. Also, a separate program called \emph{FOF-Prover9} includes a preprocessing step that attempts to reduce the problem to independent subproblems, possibly reducing the overall time significantly.

\subsubsection{Input}\label{input}

Both \emph{Prover9}-\emph{Mace4} and \emph{Vampire} present pitfalls regarding quantifier scope, which can lead to mistakes for unwary algebraists.
Let us illustrate this with \emph{Prover9}, using \verb|*| for the operation~$\op$.  In this ATP, free variables are universally quantified at the outermost level of the whole formula, regardless of its parenthesization.
For example, to formulate the property that right multiplications are injective (\texttt{y*x = z*x -> y=z}) if and only they are surjective (\texttt{exists w w*u=v}), one must include universal quantifiers explicitly on both sides of the binary connective \texttt{<->} (quantifiers have higher precedence than \texttt{<->}),
\begin{equation*}
  \texttt{all x all y all z (y*x = z*x -> y=z) <-> all u all v exists w w*u = v.}
\end{equation*}
Omitting the \texttt{all} quantifiers would quantify all variables outside the equivalence, while writing the left-hand side as \texttt{(all x all y all z y*x = z*x -> y=z)} would amount to $(\forall x,y,z, y\op x=z\op x) \implies y'=z'$, with $y'$ and~$z'$ being free variables, hence universally quantified.
We have in fact committed such mistakes over the course of the ETP, thereby giving rise to false proofs of $\Eq{677}\modelsfin\Eq{255}$, our only open implication.

Frequently, it is useful not to run an ATP alone, but to run it in a larger environment permitting multiple calls with different input files, strategies, etc.\@ so that we can provide semi-automated guidance. Integrating the ATP into a computer algebra system further allows to leverage the latter's mathematical capabilities, such as preparing input files with operations from sophisticated algebraic structures. For example, in the ETP we have integrated \emph{Prover9}-\emph{Mace4} with Python and SAGE, and (among several other applications) used SAGE to access GAP's small groups library to search, via \emph{Mace4}, for translation-invariant countermodels on specific groups (see \Cref{translation-sec}).

An important factor to consider when generating an input file for proving a conjecture is, for both \emph{Vampire} and \emph{Prover9}, that owing to the way the saturation algorithm operates, the original order in which formulas are presented is actually important: it is perfectly possible to have a collection of axioms which produces a proof in some order but not in another. \emph{Vampire} includes the option \verb|--normalize| to prevent this effect. For \emph{Prover9}, one can use a larger environment to automate the permutations of the input file and make several tries with time restrictions. In contrast, the order of the formulas does not affect \emph{Mace4}'s response.

Moreover, the inclusion of additional formulas redundant with the premises may significantly speed up the proof search, if those formulas are not quickly derived by the ATP algorithm (e.g., evaluations mapping different variables to the same one are routinely included).

Finally, the use of demodulators can greatly simplify and speed up the proof search. In \emph{Prover9}, the \verb|assign(eq_defs, fold)| command allows to substitute any specified subexpression by a user-defined symbol, simplifying all further expressions generated by the ATP\@. For instance, a quick proof that $\Eq{102744082}$ implies the injectivity of right multiplications is obtained by setting \verb|eq_defs| and adding the formula \verb|s(x) = x*x| (where \verb|*| stands for $\op$) to the premises (together with a good weight limit, see \Cref{weight_strategy}). In \emph{Vampire}, deactivating the \verb|--function_definition_elimination| mode can serve a similar purpose.

\subsubsection{The weight-sos limit strategy}\label{weight_strategy}

In \emph{Prover9}, each clause is assigned a weight depending on its length (and other customizable parameters), with newly generated clauses exceeding the \verb|max_weight| limit being discarded. In addition, the size of the list of clauses awaiting processing (the SOS list) is controlled with the \verb|sos_limit| parameter, with most newly generated clauses being discarded once the list is full. Consequently, the size and shape of the search space can be partially controlled through these parameters. Ideally, we would want to use the smallest values of \verb|max_weight| and \verb|sos_limit| that still guarantee a proof to be contained in the search space. Thus smaller values are preferable, provided they are not so low that the search is exhausted without finding a proof.

Since there are more parameters affecting the search, and different proofs may be reachable depending on the configuration, the system's behaviour with respect to \verb|max_weight| and  \verb|sos_limit| is not straightforward, with several casuistics arising. When \verb|sos_limit| is fixed, the proof-finding time tends to rise with \verb|max_weight| until reaching a plateau, with the length of this increasing phase extending for higher \verb|sos_limit| values (see \Cref{figure:weight_time}). For this reason we recommend lowering \verb|sos_limit| from its default value of 20000 to substantially smaller values.

\begin{figure}
\centering
\includegraphics[width=225pt]{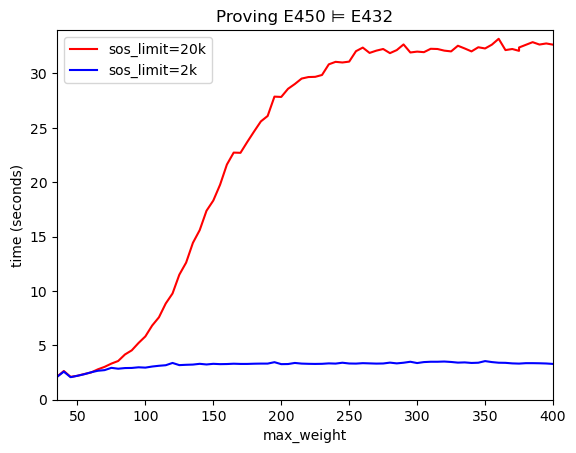}
\includegraphics[width=230pt]{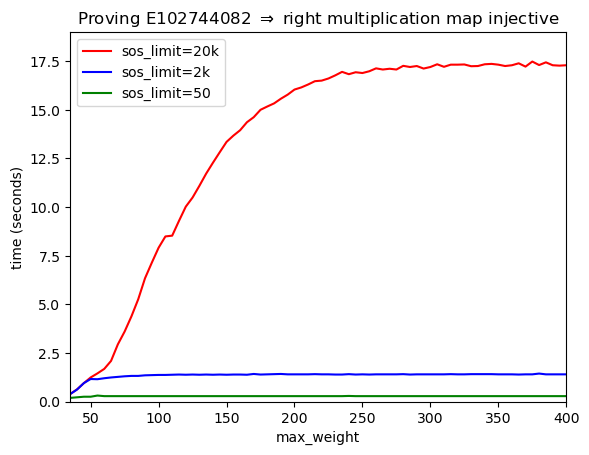}
\caption{Proof-finding time as a function of \texttt{max\_weight}, for several values of \texttt{sos\_limit}.}
\label{figure:weight_time}
\end{figure}

On the other hand, when \verb|max_weight| is fixed, the smallest \verb|sos_limit| values that still yield proofs typically produce a chaotic transient phase with higher proof-finding times, after which a better-defined relation between proof-finding time and \verb|sos_limit| emerges, which tends to follow either a near-plateau pattern (\Cref{figure:sos_limit_time}a)), or a monotonically increasing trend (\Cref{figure:sos_limit_time}b)). We recommend avoiding an excessively low \verb|sos_limit|, in order to prevent the onset of the transient phase.

\begin{figure}
\centering
\includegraphics[width=305pt]{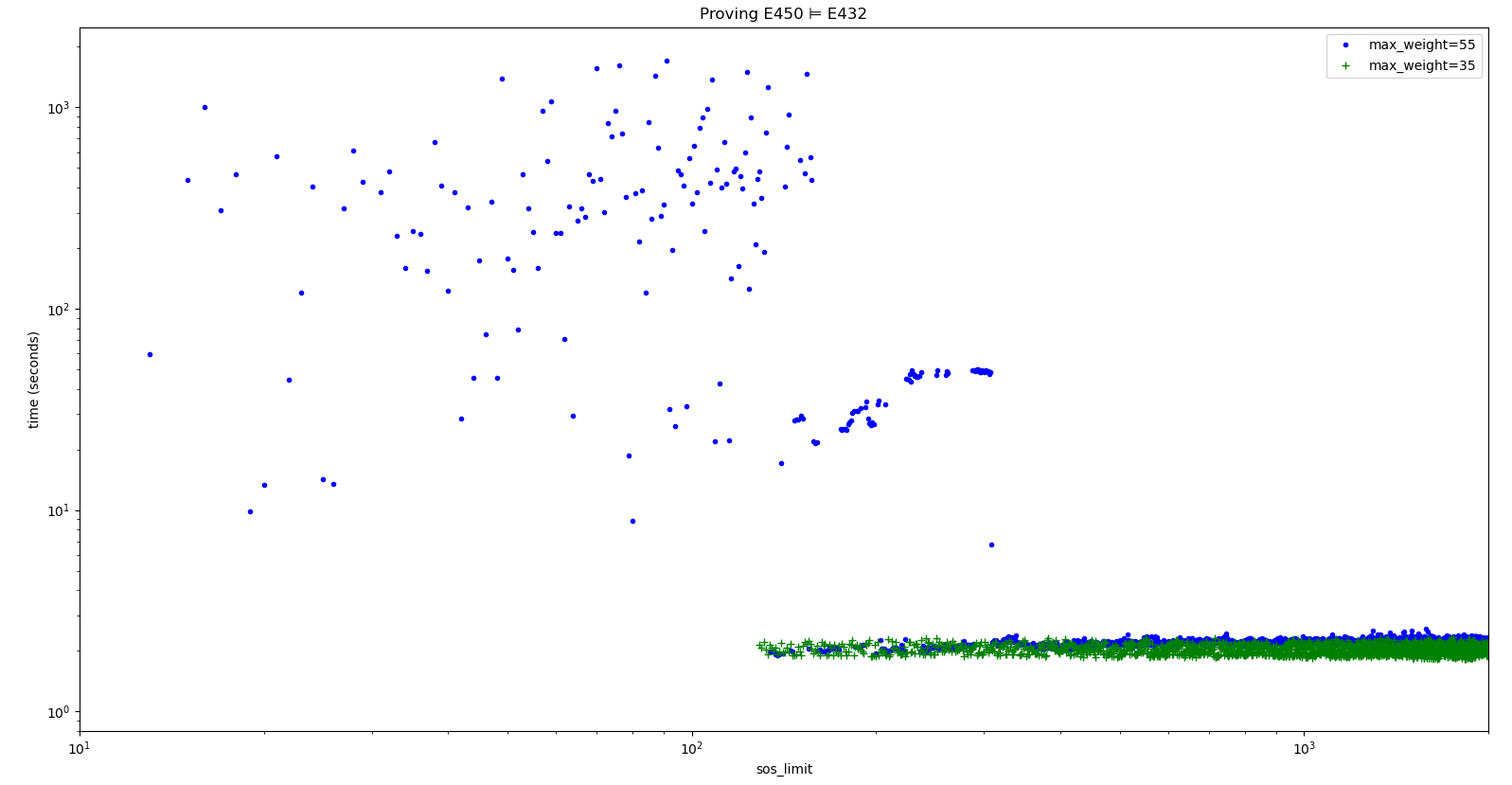}
\includegraphics[width=310pt]{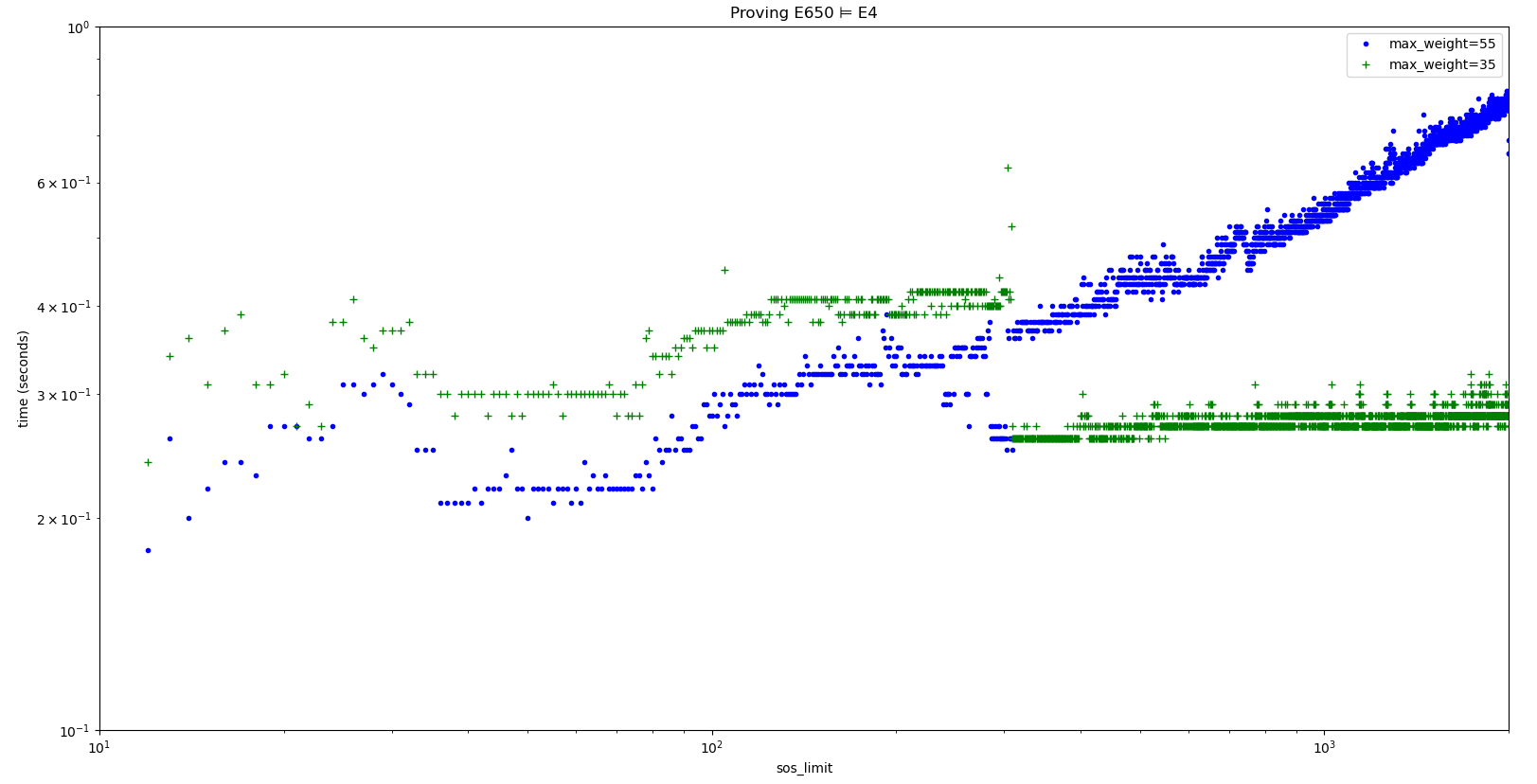}
\caption{Proof-finding time as a function of \texttt{sos\_limit}, for different values of \texttt{max\_weight}. a) Stabilization as plateau, same for different weights. b) Different behaviours for different weights, with one trend monotonically increasing.}
\label{figure:sos_limit_time}
\end{figure}

Other relevant characteristics related to the complexity of the resulting proof (such as proof level, length, weight, etc.\@) may also vary in nontrivial ways depending on the chosen \verb|sos_limit| value (see \Cref{figure:sos_limit_proof_data}).

\begin{figure}
\centering
\includegraphics[width=230pt]{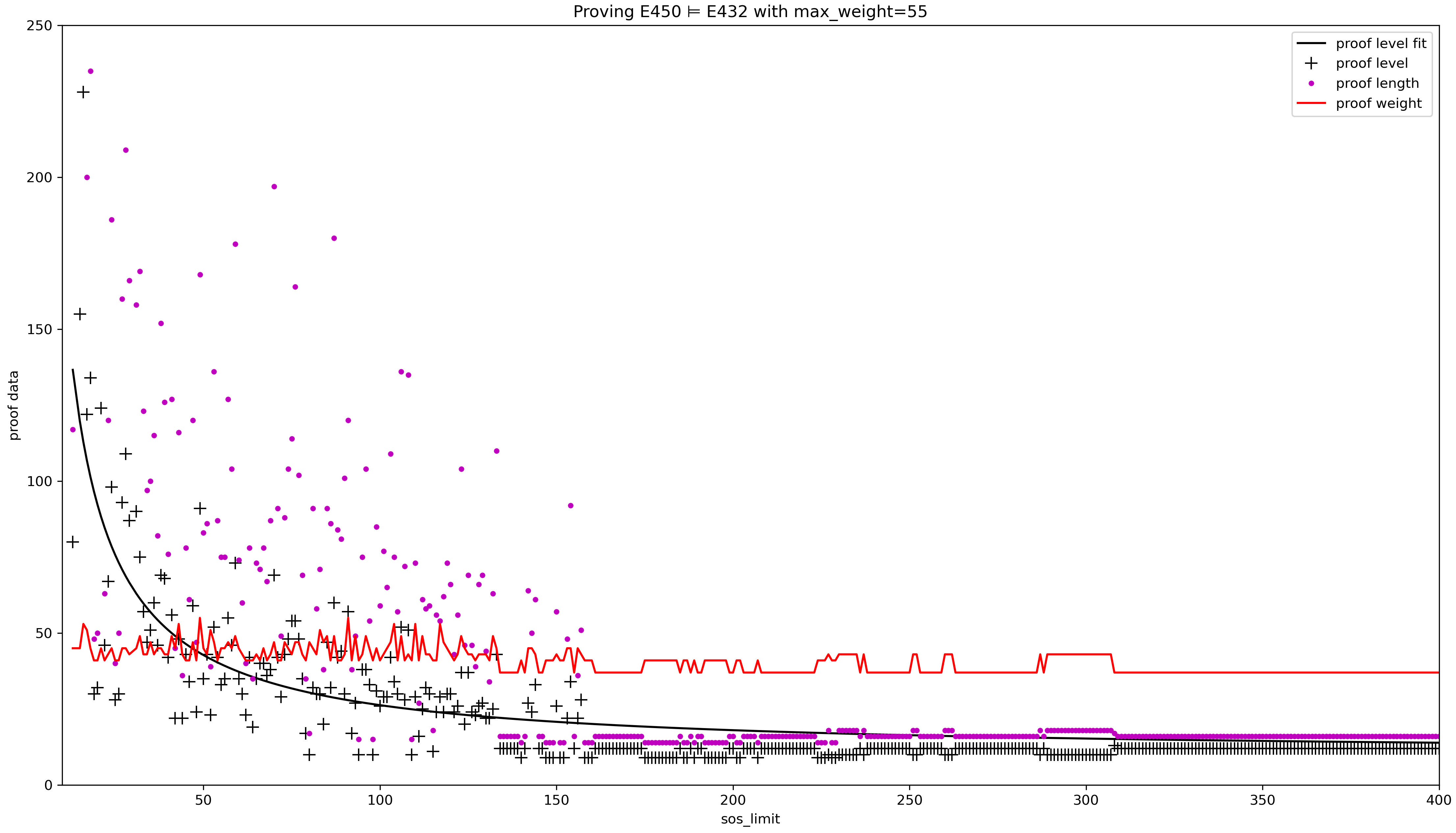}
\includegraphics[width=230pt]{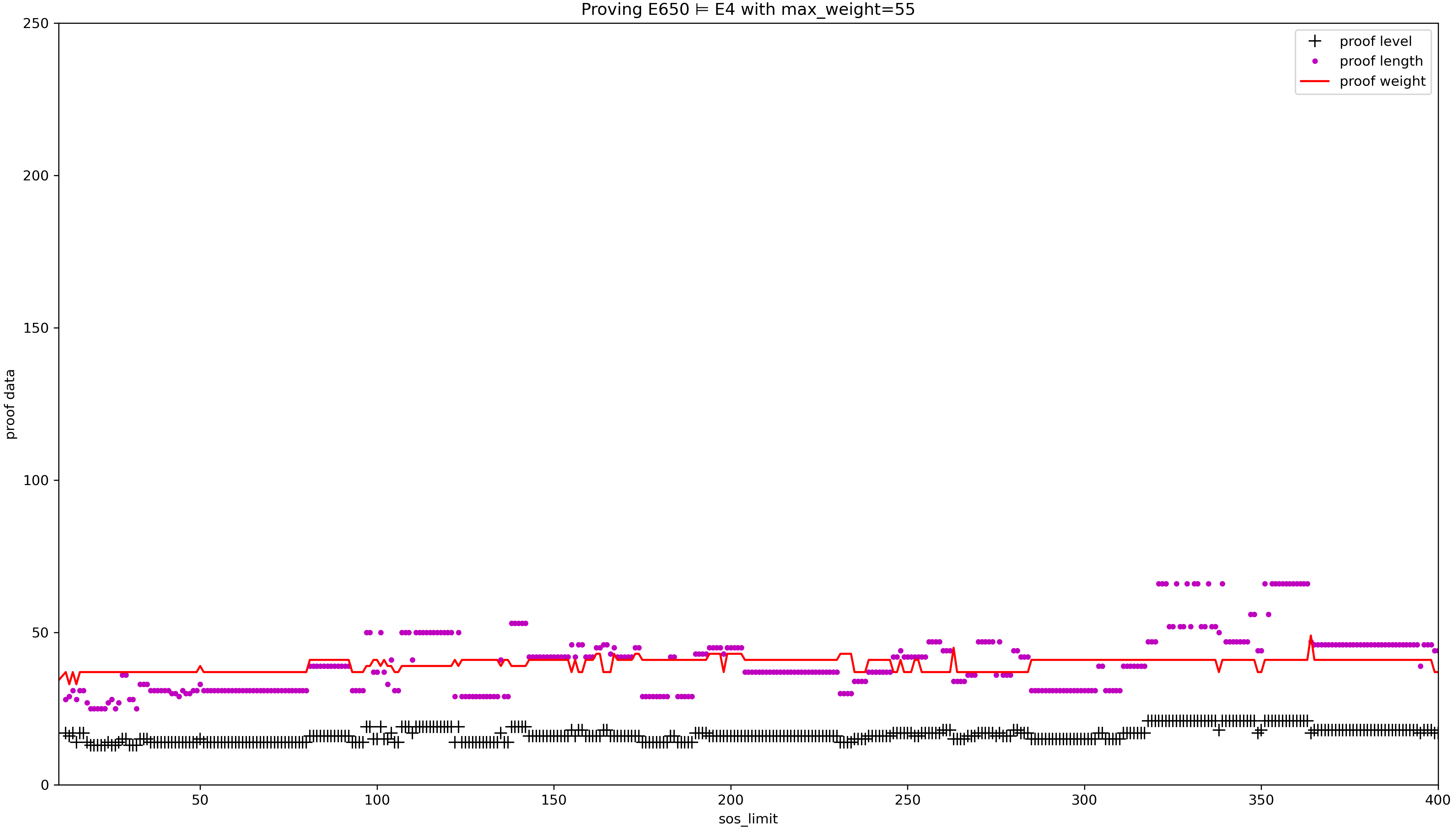}
\caption{Proof-complexity indicators as a function of \texttt{sos\_limit}.}
\label{figure:sos_limit_proof_data}
\end{figure}

Remarkably, \emph{Prover9} can establish all positive implications between laws of order up to 4 using \verb|max_weight| 55 and \verb|sos_limit| 20000, with \verb|max_weight| 25 being sufficient for almost every case. With default parameters (\verb|max_weight| 100, \verb|sos_limit| 20000), all consequences of each equation can be proven in at most $1$ second per equation, with the exception of laws $\Eq{450}$ and $\Eq{650}$ (and their duals), which take somewhat longer\footnote{For details, see \url{https://leanprover.zulipchat.com/\#narrow/channel/458659-Equational/topic/Which.20implications.20are.20harder.20for.20ATPs/near/547089094}.}.

\emph{Vampire} does not allow the imposition of a user-defined maximum weight, although its default LRS algorithm already implements a dynamic weight limit. According to \cite{janota2025experimentalresultsvampireequational}, \emph{Vampire}~5.0 can prove all true implications from the ETP, and can prove 99.97\% of them using less than 500 instructions per proof.

\subsubsection{Mace4}\label{Mace4}

The configuration of \emph{Mace4} likewise has a substantial impact on its running time, with different parameter settings resulting in differences of orders of magnitude. In particular, model-computation time is greatly affected by the \verb|selection_measure| and \verb|skolems_last| parameters. The \verb|selection_order| parameter is also relevant, but our empirical data indicate that \verb|selection_order| 2 is by far the best general choice. It is outperformed by \verb|selection_order| 0 only in a few cases, and even then the improvement in computing time is not that substantial. The effect of setting \verb|skolems_last| can range from slightly detrimental to highly advantageous: as an example, a model of $\Eq{272260}\nmodels\Eq{42323216}$ of size $7$ is found in $0.0$s with (2,4,Y), and in $4.0$s with (2,4,N). Here and in the following we use notation (a,b,S) to mean the configuration \verb|selection_order| a, \verb|selection_measure| b, and \verb|skolems_last| set (S$=$Y) or clear (S$=$N).

There is no universally optimal configuration (a,b,S): for each (potential) implication or theory, configurations may be ranked differently. However, in the ETP context, certain configurations have performed consistently better, roughly ordered as (2,3,Y), (2,1,Y), (2,4,Y), (2,3,N), (2,1,N), (2,4,N), (0,1,N), (0,1,Y), with the remaining configurations being largely avoidable. Problems involving several algebraic operations, or presenting some kind of symmetry (see \Cref{table:weakcentralgroupoids} on weak central groupoids), tend to perform better with (2,4,N/Y). Problems may also be considerably sensitive to changes in formulas that appear minor, for example adding $f^n(x)=x$ with different values of~$n$ can result in different optimal configurations.
\looseness=-1

\begin{table}
\centering
\caption{Largest size of $\Eq{1485}$ (weak central groupoids) models exhausted by each configuration in under 300s. The \texttt{skolems\_last} parameter turns out to be irrelevant for this problem.}
\label{table:weakcentralgroupoids}
\begin{tabular}{cccccc}
  \toprule
  Configuration &(2,4,Y/N) & (0,4,Y/N) & (2,3,Y/N) & (2,2,Y/N) & (2,1,Y/N) \\\midrule
  Size &14 (in \SI{127}{s}) & 14 (in \SI{166}{s}) & 13 (in \SI{290}{s}) & 9 (in \SI{33}{s}) & 7 (in \SI{8}{s}) \\
  \bottomrule
\end{tabular}
\end{table}

In addition, while in the ETP we typically searched for a single model to contradict a given implication, in some cases we sought all models of a certain size for some given theory, either to better understand a collection of base examples in order to extend them, or just to improve our understanding. Accordingly, it should be noted that a configuration that is fast for finding one model of a given size may be slower than others for exhausting the whole size. A notable example of this is (2,2,Y) for models of $\Eq{677}$ of size 9, which finds a single model in 160s, but is unable to exhaust the whole size after 20h (see \Cref{table:677models}).

When undertaking a long-running model search at a large size, it is advisable to determine in advance the potentially optimal configuration(s) by examining smaller sizes. Throughout the ETP, we have employed the following procedure. If time permits, exhaust the previous size with all configurations to make a choice. If that process is prohibitively time-consuming, then set the target number of desired configurations (e.g., for parallel runs), initialize the pool of configurations with all possibilities, and iterate:
\begin{enumerate}
\item Pick the next available smallest size.
\item Run all configurations on that size under a reasonable time limit and collect their running times, whether for finding a single model, a predetermined number of models, or for exhausting the size.
\item Based on the smallest time found, determine a statistically significant time threshold.
\item Remove from the pool those configurations whose time differs from the best one by more than the threshold (if any). End the loop once the number of configurations contained in the pool is the desired one.
\end{enumerate}
Be aware that this algorithm assumes that once a configuration outperforms another, this advantage carries over to larger sizes. This is not always the case (see \Cref{table:Schneider} for an example). Additionally, the algorithm may be further refined by considering combinations of the configurations with different subsets of formulas.

\begin{table}
\caption{Time to find one model of $\Eq{677}$ of size~$9$ and to exhaust the same size for different configurations. Some finite-setting formulas were included to speed up the search (see \Cref{finite setting}). The configurations not listed here were unable to find a model in \SI{4000}{s}.}
\label{table:677models}
\centering
\small
\sisetup{table-format=7.2, table-alignment-mode=format, table-number-alignment=center, table-align-text-post=false, table-align-text-pre=false}
\begin{tabular}{cSS}\toprule
Configuration & {One model} & {Exhaust size} \\\midrule
(2,1,N) & 16 \,\unit{s} & 190\,\unit{s} \\
(2,1,Y) & 16 \,\unit{s} & 210\,\unit{s}\\
(0,1,Y) & 20 \,\unit{s} & 149\,\unit{s}\\
(0,1,N) & 32 \,\unit{s} & 169\,\unit{s}\\
(2,2,Y) & 154 \,\unit{s} & {$>$}72000\,\unit{s} \\
(2,2,N) & 180 \,\unit{s} & {$>$}43000\,\unit{s} \\
(2,3,N) & 557 \,\unit{s} & 877\,\unit{s}\\
(2,3,Y) & 587 \,\unit{s} & 941\,\unit{s}\\
(2,4,Y) & 596 \,\unit{s} & 761\,\unit{s}\\
(2,4,N) & 620 \,\unit{s} & 730\,\unit{s}\\
(0,3,N) & 3746 \,\unit{s} & 4151\,\unit{s}\\
(0,3,Y) & 3777 \,\unit{s} & 4254\,\unit{s}\\
\bottomrule
\end{tabular}
\end{table}

\begin{table}
\caption{Comparison of the two best configurations for finding certain special $\Eq{677}$ models, both for exhausting each size and for finding a single model, with best times in boldface. Note that initially (2,4,N) seems the best configuration, but in the long run, (2,3,Y) outperforms it. The comparison is particularly misleading at size 16 when searching for a single model.  All times are in seconds.}
\label{table:Schneider}
\centering
\small
\sisetup{table-format=4.4, detect-weight=true, table-align-text-post=false}
\begin{tabular}{cS@{}SS@{}S}\toprule
& \multicolumn{2}{c}{Exhaustion} & \multicolumn{2}{c}{Finding one model} \\
\cmidrule(lr){2-3}\cmidrule(lr){4-5}
Size & {(2,3,Y)} & {(2,4,N)} & {(2,3,Y)} & {(2,4,N)}\\\midrule
9 & 0.14 & \bfseries 0.04 & 0.02 & \bfseries 0.0\\
10 & 0.32& \bfseries 0.12& {---} & {---}   \\
11 & 0.82& \bfseries 0.37& 0.15& \bfseries 0.01\\
12 & 2.17& \bfseries 1.2& {---} & {---}   \\
13 & 6.18& \bfseries 4.35& 2.74& \bfseries 0.91\\
14 & 15.47& \bfseries 14.39& {---} & {---}   \\
15 & \bfseries 38.82& 50.87& {---} & {---}   \\
16 & \bfseries 120& 175& 11.95& \bfseries 0.02\\
17 & \bfseries 306& 660& {---} & {---}   \\
18 & \bfseries 883& 2382& {---} & {---}  \\
19 & \bfseries 2300& {\!\llap{$>$}}3600 & \bfseries 846& 2082\\\bottomrule
\end{tabular}
\end{table}

\subsubsection{Search of compatible properties} When in search of a model, either theoretically or using a model builder, it is useful to be able to identify additional properties to impose on the model, strong enough to simplify the search, yet weak enough to guarantee compatibility with the original axioms. In particular, when seeking a countermodel to implication $\E\models\E'$, we should avoid properties that, together with $\E$, would imply $\E'$. This search can be greatly aided by an ATP: if we have an educated guess that some property $P$ is compatible with the problem $\E\nmodels\E'$, we can run the ATP on $\E\wedge P$ under various strategies to attempt a proof of $\E'$. If, after allowing ample time, no proof is found, this provides heuristic evidence that a countermodel satisfying $P$ exists. Note that even if no proof actually exists, it may well be that there are infinite countermodels but not finite ones.

In the ETP, we successfully applied this approach to several of the outstanding implications. Most notably, we were able to construct an infinite model of $\Eq{1323}\nmodels\Eq{2744}$ after heuristically verifying compatibility\footnote{Specifically, for each property, we ran \emph{Prover9} for 20 minutes with the default configuration and \emph{Vampire} for 999s in CASC mode.} with a unit element and with closure of the operation over the set of square elements.

\subsubsection{Finite setting}\label{finite setting}

There are situations in which one wants to work in the finite setting, for example, when proving a finite implication. Indeed, model builders typically search for finite models; therefore, when using them, we can usually assume we are operating in the finite setting. In this context, injective (resp. surjective) maps are bijective, bijections are periodic maps, and so on. In addition, any law of the form $x = f(y)\op g(x,y)$ for some maps $f,g$ has left multiplication map $L_{f(y)}$ surjective, hence bijective. Thus by applying the substitution $x\mapsto f(y)\op x$ and simplifying $f(y)$ on the left, the law finitely implies $x = g(f(y)\op x,y)$.

By this and similar approaches we typically find some multiplication or related map (squaring, etc.) to be injective and surjective, properties that we may add to our initial formulas. In hindsight, adding injectivity to a finite model builder usually increases performance, whereas surjectivity (and other existentially quantified formulas) tend to diminish it. Moreover, \emph{Prover9}-\emph{Mace4} performs better with operations and equalities than with non-equational formulas. For this reason it is advantageous to convert an injectivity condition from an implication into a new operation. For example, injectivity of the left multiplication map is captured by $x \backslash (x * y) = y$ (where $*$ stands for $\op$); this substitution typically yields a threefold speedup. In contrast, this conversion appears to slow down \emph{Vampire} slightly.

\begin{example}
As a case study, let us examine the model search for $\Eq{1518}\nmodels\Eq{47}$ using \emph{Mace4} (with $*$ standing for $\op$). It is now known that the smallest models have size~$15$. Initially, this search exceeded our ATPs expertise, hence we ended up constructing a theoretical model of size~$232$. If we restrict to the original formulas $\Eq{1518}$ and $\neg\Eq{47}$, the optimal configuration is (2,3,N), taking $2.5$ minutes to exhaust size~$10$ and already exceeding $7$~hours to exhaust size~$11$. Putting $S(x):=x\op x$, since $\Eq{1518}$ is $x = (y \op y) \op  (x \op  (y \op  x))$ we observe that $L_{S(y)}$ is surjective, hence injective. Accordingly we add \verb|(y*y)*x = (y*y)*z -> x = z| to \emph{Mace4} and then size 11 is exhausted in less than $5$ minutes. In addition, from the implication graph we observe that $\Eq{1518}$ implies $\Eq{359}$, which states $S(x) \op x = S(x)$; when added as \verb|(x*x)*x = x*x|, \emph{Mace4} exhausts size 11 in under 1s, with the optimal configuration shifting to (2,1,N). Size 12 is exhausted in 18s. We also note that if $S(x)=S(y)$ then $S(x)\op x = S(x) = S(y) = S(y)\op y$, yielding $x=y$ by injectivity of $L_{S(x)}$. Therefore $S$ is injective, we can add \verb|x*x = y*y -> x = y|, and exhaust size 12 in 13.5s. Also, in $\Eq{1518}$ we can substitute $x\mapsto S(y)\op x$ and cancel $S(y)$ on the left to get $\Eq{320858}$, with which we exhaust size 12 in under $2$s and size 13 in under $1$ minute. Moreover, as $S$ is surjective every element is a square, so the injectivity of $L_{S(x)}$ implies that of $L_x$, and we can add \verb|x*y = x*z -> y = z| to the search. But this actually worsens the performance! After some experimentation, we observe that we should either add the injectivity of $L_{S(x)}$ or that of $L_x$, but not both. It turns out that for size 13, injectivity of $L_{S(x)}$ performs slightly better, so we retain it. Additionally, the implication formula for this injectivity can be replaced with a new operation~\texttt{\textbackslash} by writing \texttt{x\textbackslash((x*x)*y) = y}. This further reduces the exhausting time for size 13 to $33$s. Now size 14 is exhausted in 30 minutes, and size 15 is well within reach.
\end{example}

\subsubsection{Discriminators for Isofilter}
\emph{Mace4}'s output may be redundant, in that many of the models found in a search may (and often will) be isomorphic to each other. To solve this problem, \emph{Mace4}'s output can be fed to another tool called \emph{Isofilter}, which returns a representative model of each isomorphism class. \emph{Isofilter} compares permutations of the models, restricted according to some discriminators: by default the only one used is the frequency with which each domain element appears in the operation tables\footnote{This discriminator is entirely ineffective with models having same number of copies of each element in their tables, such as quasigroups.}, but the user can include any number of other discriminators. \emph{Isofilter} ``as is'' handles models of size up to around 10 with ease; but the combinatorial explosion soon makes computations unfeasible for larger sizes, unless discriminators well suited to the problem are chosen. For example, \emph{Mace4} provides 10 models of size 15 for $\Eq{1518} \not\models \Eq{47}$, of which $6$ belong to different isomorphism classes. Vanilla \emph{Isofilter} requires 10 days and processes $3\cdot10^{13}$ permutations in order to determine these classes, while adding the discriminators \verb|x*x=x. (x*y)*z = x*(y*z). x*y=y. x*y=y*x.| reduces the task to $10^5$ permutations, completed in 0.02s.

\section{Implications for finite magmas}\label{austin-sec}

Many of the techniques used to determine the graph of implications $\E \models \E'$ can also be used to determine the graph of finite implications $\E \modelsfin \E'$, with the notable exception of the greedy construction, which appears to be inherently infinitary in nature.  On the other hand, when the magma $\Magma$ is finite, one can prove additional implications by using the fact that any function $f \colon M \to M$ which is surjective is necessarily injective, or vice versa.  We could establish about 200 new implications by providing these two axioms to \emph{Vampire} or to the \emph{Lean} package \emph{Duper}, though in the latter case some human rewriting of the proof was needed to formalize it in the base installation of \emph{Lean}.  A small number of additional implications could be resolved by more complicated facts about functions $f,g \colon M \to M$, such as the fact that $f = f \circ f \circ g$ implies $f = f \circ g \circ f$.  We refer the reader to the blueprint for examples of such arguments, which were obtained by \emph{ad hoc} methods.

In the end, we were able to establish $820$ new implications $\E \modelsfin \E'$ for which $\E \not \models \E'$; and for most other anti-implications $\E \not \models \E'$, we were able to strengthen the anti-implication to $\E \not \modelsfin \E'$.  However, there was (up to duality) precisely one finite implication which we could not settle, which we leave as an open problem:

\begin{problem}  Does the law $\x \formaleq \y \op (\x \op ((\y \op \x) \op \y))$ \eqref{eq677} imply the law $\x \formaleq ((\x \op \x) \op \x) \op \x$ \eqref{eq255} for finite magmas?
\end{problem}

This problem appears to be ``immune'' to many of our constructions, such as the linear magma construction or the magma cohomology construction; the greedy construction does show that $\Eq{677} \not \models \Eq{255}$, but the construction is inherently infinite in nature.  We tentatively conjecture that $\Eq{677} \not \modelsfin \Eq{255}$; we refer the reader to the blueprint for several partial results in this direction.

\section{Spectrum of equational laws}\label{spectrum-sec}

Given an equational law~$\Eq{n}$, one can ask for its spectrum, namely the set of cardinalities of its finite models.\footnote{The infinite spectrum is uninteresting (assuming the axiom of choice).  Let us show that if $\Eq{n}$ is not equivalent to the singleton law $\Eq{2}$ then it has models of all infinite cardinalities~$\kappa$.  The free magma on $\Eq{n}$ with $\kappa$ generators is such a model.  Indeed, the generators are distinct in this magma (otherwise $\Eq{n}$ would imply $\Eq{2}$) so its cardinal $\mu$ is at least $\kappa$.  Conversely, $\mu\leq\sum_T\kappa^{|T|}$ where the range sums over finite binary trees, and this sum is bounded above by $\aleph_0\cdot\kappa=\kappa$.}
The spectrum $\operatorname{Spec}(\Eq{n})$ is a multiplicative subset of $\mathbb{Z}_{>0}$ since the direct product of models is a model.
We focus here on the most basic question, that is, which laws (of order up to~$4$) have spectrum equal to $\mathbb{Z}_{>0}$?

Several extensions will be described in a separate publication: determining the spectrum and not only whether it is full; the spectrum of simple magmas or (sub)directly irreducible magmas; tracking multiplicity, namely counting (or bounding asymptotically) how many finite magmas exist of each size in the spectrum; this is referred to as the \emph{fine spectrum} in \cite{taylor}.  These detailed considerations reveal profound differences in how much an equational law constrains the magma operation, and may help organize the implication graph into different families.

An ATP run shows that $\num{1558}$ laws of order up to~$4$ have no model of size~$2$ and $\num{62}$ have a model of size~$2$ but none of size~$3$.  Pushing the search to higher model sizes does not resolve the question for any of the remaining $\num{3074}$ laws.  As we explain next, all of these laws actually have full spectrum.\footnote{Further investigations show that the lowest-numbered law with models of sizes $2$ and $3$ but not full spectrum is $\Eq{80887}$, namely $\x \formaleq \y \op (\y \op (\y \op (((\y \op \y) \op \x) \op \y)))$, of order~$6$.}

Our main tool by far to show that a law has full spectrum is to consider the carrier set $\mathbb{Z}/n\mathbb{Z}$, with a linear operation $x\op y = ax+by$ with $a,b\in\{-1,0,1\}$.  If the law holds for some choice of $a,b$ then the law has full spectrum.
\begin{itemize}
\item For $(a,b)=(0,0)$ the operation is the constant operation, which is a model of $\Eq{1}$ and any law whose sides both have positive order.  Equivalently, these laws are consequences of the constant law $\Eq{46}$.
\item For $(a,b)=(1,0)$ the operation is a projection, which is a model of any law whose sides start with the same first variable, equivalently the consequences of $\Eq{4}$.  (The choice $(a,b)=(-1,0)$ is a model of fewer laws hence is not useful.)  Likewise $(a,b)=(0,1)$ shows that laws whose sides end with the same last variable (consequences of $\Eq{5}$) have full spectrum.
\item For $(a,b)=(1,-1)$ the operation is abelian group subtraction, characterized by Tarski's axiom $\Eq{543}$, which shows that any law implied by $\Eq{543}$ has full spectrum.  Likewise, backwards subtraction $(a,b)=(-1,1)$ provides models for $\Eq{1090}$ (equivalent to the dual of $\Eq{543}$) and its consequences.
\item The operation for $(a,b)=(1,1)$ is abelian group addition, which is useless as any law it satisfies for all~$n$ is also satisfied by the constant operation $(a,b)=(0,0)$.
\item Finally, the operation $(a,b)=(-1,-1)$ is a model of some more laws, such as the semi-symmetric quasigroup law $\Eq{14}$ and totally symmetric quasigroup law $\Eq{492}$.
\end{itemize}
These considerations account for $\num{3068}$ laws, and there remain three dual pairs of laws to treat.  This is done through ad hoc models: a piecewise linear model for $\Eq{1682}$ and its dual, and models whose operation table is mostly constant for the remaining laws $\Eq{1482}$, $\Eq{1523}$, and their duals.

Since a law implied by a full spectrum law has full spectrum itself, the implication graph reduces significantly the number of laws for which it is useful to formalize the full spectrum property.  Accounting for duality and implications, we found it sufficient to formalize the proof that $32$ laws have no magma of size~$2$ or none of size~$3$, and the explicit construction of magmas of all finite sizes for the $7$~laws $\Eq{4}$, $\Eq{41}$, $\Eq{492}$, $\Eq{543}$, $\Eq{1482}$, $\Eq{1523}$, and $\Eq{1682}$.  In conclusion, we prove that $\num{3074}$ laws ($65\%$) have full spectrum $\operatorname{Spec}(\Eq{n})=\mathbb{Z}_{>0}$ and $\num{1620}$ ($35\%$) do not (including $\num{1496}$ laws equivalent to $\Eq{2}$).  These percentages remain roughly stable at higher orders, with $60\%$ of laws of order up to $9$ having full spectrum, as will be reported elsewhere.

\section{Higman--Neumann laws}\label{higman-neumann}

\subsection{Describing groups as magmas}

The ETP is focused exclusively on magmas, which only feature a single (binary) operation.  Many mathematical structures traditionally defined using several operations can nevertheless be fully described as magmas with a well-chosen combined operation, from which the whole structure can be reconstructed.  The first example is how Boolean algebras defined in terms of three operations $(\land,\lor,\lnot)$ were equivalently described in 1913 in terms of the Sheffer stroke $x\op y\coloneqq \lnot(x\land y)$~\cite{sheffer}.  Once such a single operation is found, a separate endeavor is to determine which laws it must satisfy to get the desired structure, and, in favorable cases find a single law that encapsulates the whole structure, or even find all equivalent laws of minimum order.  The earliest such example is Tarski's description of abelian groups in terms of subtraction $x\op y\coloneqq x+(-y)$, subject to a single axiom $\x \formaleq \y \op (\z \op (\x \op (\y \op \z)))$ \eqref{eq543}, found in 1938~\cite{Tarski1938}.  It then took three decades~\cite{higman-neumann,Sholander01021959,Padmanabhan_1969} to sort out the full equivalence class of $\Eq{543}$ among laws of order~$4$.  For Boolean algebras, a minimum-order single-law description was only found in~\cite{mccune_et_al}, nine decades after Sheffer's work.

We plan to report elsewhere on other examples such as modules over Eisenstein integers $\mathbb{Z}[\omega_3]$ or Gaussian integers $\mathbb{Z}[\omega_4]$, with $\omega_k$~a primitive $k$th root of unity, which can be described by the operation $x \op y \coloneqq x + \omega_k y$ subject to the order-$6$ laws $\Eq{85914}$ and~$\Eq{86082}$, respectively.

Here, we describe the case of groups.  The binary operation~$*$, unary operation~$(\cdot )^{-1}$, and nullary operation~$e$ (identity element) can be repackaged into a single division operation $x \op y \coloneqq x*y^{-1}$, from which the original operations are easily reconstructed: for instance $x*y=x\op((y\op y)\op y)$.  A group equipped with division, called a Ward quasigroup, is a magma $(G, \op)$ satisfying the unipotence law $\x \op \x \formaleq \y \op \y$ \eqref{eq40}, the right-unit squares law $\x \formaleq \x \op (\y \op \y)$ \eqref{eq11}, and a version of the associativity law dubbed the half-group law, $\x \op \y \formaleq (\x \op \z) \op (\y \op \z)$ \eqref{eq3737}, from which group axioms are easily derived.
These three laws are equivalent to a single law $\E_{\mathrm{HN}}\coloneqq \Eq{42323216}$ of order~$8$, found by Higman and Neumann~\cite{higman-neumann},
\[
\E_{\mathrm{HN}} \colon \x \formaleq \y \op \Bigl(\bigl(((\y \op \y) \op \x) \op \z\bigr) \op \bigl(((\y \op \y) \op \y) \op \z\bigr)\Bigr).
\]
McCune found two more laws equivalent to this one and of the same order \cite{mccune1993single}, $\Eq{42302852}$ and $\Eq{147976245}$.  A natural question is to find all characterizations of Ward quasigroups (groups equipped with division) with minimum order.
Throughout our exploration, we used two criteria: the law must be satisfied by group division, and must fail for magmas that are not Ward quasigroups.

\subsection{Basic constraints}

There are $\num{298012537}$ laws of order up to~$8$, and running an ATP on all of them is too slow, so one needs efficient ways to filter them beforehand.  Let us begin with restrictions on the shape of any law equivalent to~$\E_{\mathrm{HN}}$.
The law must take the form $\x\formaleq\dots$ as otherwise it would be satisfied by the constant operation on any set.
The law must be satisfied when evaluated with all variables set to the same element (say,~$1$) in the Ward quasigroup $\mathbb{Z}$ equipped with subtraction.  In particular the law must have even order.
This reduces from $\num{3470}$ shapes of order up to~$8$, down to just $548$~shapes.

Next come some restrictions on the variables.
The right-hand side must not start nor end with the variable~$\x$, as otherwise the projection operations $x\op y \coloneqq x$, $x\op y\coloneqq y$ respectively would satisfy the law.
The law must have at least three variables: otherwise it is satisfied by division in any diassociative loop (such as a Moufang loop), namely a quasigroup with identity element in which every $2$-generated submagma is a group.
Each variable must appear an even number of times, so that the law holds in Boolean groups (abelian groups of exponent~$2$).
These basic constraints leave $54$, $\num{9000}$, and $\num{1841910}$~candidate laws of orders $4$, $6$, and~$8$, respectively, which can be efficiently enumerated since the conditions so far constrain separately the shape and list of variables (refer to \Cref{numbering-app} for the relevant definitions).

Imposing further that the law is satisfied by division in a free non-abelian group (with one generator per variable) reduces these numbers of laws to $0$, $59$, and $\num{5692}$ at these same orders.  All of the laws coming out of these filters are consequences of~$\E_{\mathrm{HN}}$; accordingly, one must determine which of these candidates imply said law.

\subsection{Using automated theorem provers}

We repeatedly whittled down the list of candidates by accumulating a collection of finite countermodels, namely magmas that satisfy a candidate law while violating one of the laws $\Eq{11}$, $\Eq{40}$ and~$\Eq{3737}$ characterizing Ward quasigroups.  Automated searches of small magmas (of size up to~$8$) with \emph{Mace4} or \emph{Vampire} gave many countermodels (of which $12$ are enough).  A second source was that of linear models $x \op y \coloneqq ax+by$ on $\mathbb{Z}/n\mathbb{Z}$ with $(a,b)\neq(1,n-1)$: the largest one we used is $x \op y \coloneqq 261x + 33y \bmod 307$ to rule out the candidate law $\Eq{68185620}$, $\x \formaleq (\y \op \y) \op (\y \op ((\x \op (\z \op \y)) \op ((\x \op \x) \op \z)))$.  Finally, we introduced some models that are ``almost'' Ward quasigroups: the $7$-element smallest non-associative inverse loop (equipped with division), the $10$-element smallest non-associative Steiner loop (commutative loop in which divisions coincide with multiplication), and the $16$-element Moufang loop of unit octonions over~$\mathbb{Z}$.

These steps eliminated all candidates of order less than~$8$,\footnote{In particular we recover the nonexistence of laws of order~$6$ characterizing Ward quasigroups, already announced by McCune and Kinyon~\cite{mccune-webpage-gtsax}.} and left only $213$ laws of order~$8$ that could be equivalent to~$\E_{\mathrm{HN}}$.  These laws come in $31$ families consisting of a ``parent'' $5$-variable law and some specializations with pairs of variables being identified.  The lowest-numbered law in this list is McCune's law
\[
\x \formaleq \y \op \Bigl(\bigl(((\x \op \x) \op \x) \op \z\bigr) \op \bigl(((\x \op \x) \op \y) \op \z\bigr)\Bigr) \ \eqref{eq42302852},
\]
which is in the same family as the Higman--Neumann law.  Another common feature is that all $213$ candidate laws include at least one subexpression of the form $v \op v$ for some variable~$v$.

For $179$ candidate laws~$\E$, we showed the implication $\E\models\E_{\mathrm{HN}}$ using the ATP \emph{Prover9}.  For equations of this order, the ATP computation times increase significantly compared to order-$4$ laws, with some proofs taking $20$~times longer than checking with \emph{Prover9} all $\num{8178279}$ positive implications of the main project.  The choices of parameters bounding the ATP search (such as the parameter \texttt{max\_weight} limiting clause complexity in \emph{Prover9}) were particularly crucial, with different values being optimal in different proofs.  Another important speed-up was obtained by seeking proofs of a simple property such as $\Eq{11}$, $\Eq{40}$, $\Eq{3737}$, or injectivity/surjectivity of left or right multiplications, then seeking proofs that the candidate law together with that property implies some other property, and so on, until proving all three laws characterizing Ward quasigroups.  The reverse approach also proved useful, namely finding which property would allow the proof to succeed, then seeking a proof of that property from the candidate law.  The law $\Eq{102744082}$ was a particularly difficult instance: together with injectivity of right multiplications it easily implies the Higman--Neumann law, but the proof that $\Eq{102744082}$ does imply injectivity took several hours to obtain in a sweeping search with general parameters; an optimized choice of \emph{Prover9} options trims this time down to $0.2$~seconds.  The two characterizations $\Eq{42302946}$ and~$\Eq{89176740}$ of division in groups deserve particular mention for being nicely expressed in terms of the right-cubing map $C(x) \coloneqq (x \op x) \op x$:
\[
  \x \formaleq \y \op ((C(\x) \op \z) \op (C(\y) \op \z)), \qquad
  \x \formaleq C(\y) \op ((C(\x) \op \z) \op (\y \op \z)).
\]

Among the $34$ remaining candidates, we showed the finite implication $\E\modelsfin\E_{\mathrm{HN}}$ for $21$~laws~$\E$, which means that the law $\E$ characterizes Ward quasigroups among finite magmas.  Let us illustrate the proof technique for $\x \formaleq (\y \op \y) \op (\y \op ((\x \op \z) \op (((\x \op \x) \op \y) \op \z)))$ \eqref{eq67953597}.  In a finite magma, one gets $x = L_{y\op y} \circ L_y \circ f_{y,z}(x)$ in terms of the function $f_{y,z}\colon x \mapsto (x\op z) \op (((x \op x) \op y) \op z)$.  Finiteness implies that the composition of several functions can only be a bijection if all of them are bijections, thus left multiplications are bijective.  By selecting $y=L_{x\op x}^{-1}(w)$ one gets that $(x\op z)\op (w\op z)$ equals the $z$-independent expression $L_y^{-1} \circ L_{y\op y}^{-1}(x)$.  Taking $w=x$ yields that the square of $L_x(z)$ is $z$-independent, hence (by surjectivity of~$L_x$) all squares are equal.  A routine ATP run then concludes.
While the resulting proofs of finite implications are relatively short and have been successfully ported to \emph{Lean}, our automated search involved thousands of \emph{Vampire} runs.
Indeed, rather than the condition that a bijective composition implies bijectivity of its constituents, we had to use the more concrete property that injectivity is equivalent to surjectivity for various collections of specific functions $f\colon M\to M$ such as left or right multiplications, cubing, etc.\@, with a brute-force search over which functions to include in a given run.

The remaining $13$~candidates have proven to be quite resistant to both proof and countermodel attacks with a wide range of parameter options, target clauses, additional clauses, and given time (up to 10 hours for each candidate and given experiment).  We have shown with \emph{Prover9} that a model of any of the remaining candidates~$\E$ that has a right-identity element, or that satisfies the right-unit squares law~$\Eq{11}$ or unipotence law~$\Eq{40}$, is a Ward quasigroup.  As such, any putative countermodel to the implication $\E\models\E_{\mathrm{HN}}$ must be far from being a Ward quasigroup, in the sense that it must violate these laws.\footnote{In fact, we find that countermodels must violate many laws satisfied by group division: $\Eq{11}$, $\Eq{40}$, $\Eq{823}$, $\Eq{835}$, $\Eq{842}$, $\Eq{846}$, $\Eq{1049}$, $\Eq{1637}$, $\Eq{1673}$, $\Eq{1718}$, $\Eq{1835}$, $\Eq{1876}$, $\Eq{3282}$, $\Eq{3323}$, $\Eq{3662}$, $\Eq{3665}$, $\Eq{3677}$, $\Eq{3684}$, $\Eq{3721}$, $\Eq{3729}$, $\Eq{3737}$, $\Eq{3761}$, $\Eq{3823}$, $\Eq{3870}$, $\Eq{3891}$, $\Eq{3943}$, $\Eq{4270}$, $\Eq{4590}$ of order up to~$4$, and many more of higher order, including the associative right division law $\Eq{912704}$, $\x\op(\y\op(\z\op\w)) \formaleq(\x\op(\w\op\z))\op\y$.}  In particular it cannot be a loop; this discards, e.g., the $32$-element loop of unit sedenions over~$\mathbb{Z}$ as a potential countermodel.

In summary, out of the $\num{298012537}$ laws of order up to~$8$, we found $179$ laws characterizing Ward quasigroups, $21$~characterizing them among finite magmas but perhaps not infinite ones, and $13$~candidates for which we have neither a counterexample nor a proof even for finite magmas. No efforts have been made yet to find infinite countermodels with the techniques developed in the main project. The results of this section\footnote{\url{https://github.com/teorth/equational_theories/blob/main/data/Higman-Neumann.json}} have not been formalized in \emph{Lean} yet. More details and further progress on this side project will be published elsewhere.

\section{AI and Machine Learning contributions}\label{ml-sec}

As discussed in \Cref{automated-sec}, the ETP made extensive use of automated theorem provers in completing the primary goal of determining and then formalizing all the implications between the specified equational laws.  In contrast, we were only able to utilize modern large language models (LLMs) in a fairly limited fashion.  Such models were useful in writing initial code for graphical user interfaces that we discuss further in \Cref{sec:gui-sec}, as well as performing some code autocompletion (using tools such as \emph{GitHub Copilot}) when formalizing an informal proof in \emph{Lean}.  In one instance, \emph{ChatGPT} was used\footnote{\url{https://chatgpt.com/share/670ce7db-8a44-800d-a5dc-8462c12eca3b}} to guess a complete rewriting system for the law $\x \op \y \formaleq \x \op ((\y \op \y) \op \z)$ \eqref{eq3523} which could then be formally verified, thus resolving all implications from this equation. However, in most of the difficult implications that resisted automated approaches, we found that LLMs did not provide useful suggestions beyond what the human participants could already propose.

On the other hand, we found that machine learning (ML) methods showed some promise of being able to heuristically predict the truth value of portions of the implication graph; we shall now discuss a convolutional neural network approach.\footnote{For some discussion of other machine learning experiments performed during the Equational Theories project, see \url{https://leanprover.zulipchat.com/\#narrow/channel/458659-Equational/topic/Machine.20learning.2C.20first.20results} for a (vectorized) transformer neural network approach, and see \url{https://leanprover.zulipchat.com/\#narrow/channel/458659-Equational/topic/Graph.20ML.3A.20Directed.20link.20prediction.20on.20the.20implication.20graph} for directed link prediction on the implication graph using Graph Neural Network (GNN) autoencoders.}

\subsection*{Convolutional neural network model for the implication graph}

To model the implication graph, we used a convolutional neural network (CNN)\@. For each pair of equations $(p,q)$, the input of the CNN consisted of a character-level tokenization (no vectorization) of the two equations, and the output of the CNN was a yes/no label depending on whether $p$ implies $q$. The CNN processed the input data using a 5-layer architecture, each layer composed of a 1-D convolution, followed by batch normalization and rectified linear unit activation functions~\cite{Goodfellow-et-al-2016}. After the last convolutional layer, a flattening layer and a softmax activation function were used to obtain the output of the network, i.e.\@, the prediction of the implication for the input pair of equations. Note that we trained our models with the 16 October 2024 version of the (infinite) implication graph, for which 362 of the hardest implications (less than 0.002\% of the total) were still unknown.

\smallskip

Prior to training the CNN, we divided the data into training (60\%), validation (20\%), and test (20\%) subsets. The CNN was implemented in TensorFlow 2.9.0~\cite{tensorflow2015-whitepaper} and trained on an NVIDIA 3080 Ti GPU with the following configuration: the binary cross-entropy as the loss function to minimize, the Adam method with an initial learning rate of $10^{-3}$ for the adjustment of network weights, a batch size of 1024 with random shuffling, a learning rate reduction by a factor of 2 after 15 epochs without improvement in the validation loss, and early stopping if no improvement occurred for 40 epochs, with the model with the lowest validation loss being retained as the final CNN\@.
The final CNN was evaluated on the test set, reaching a prediction accuracy of 99.7\%, which means that the model misclassified around 66k of the 22 million implications. Since this accuracy was somewhat surprising for such a small and “simple” model, as a control we generated a random label (yes/no) for each pair of equations, then trained the CNN on this data with 60\%/20\%/20\% training/validation/test percentages, resulting in a 49.99\% accuracy, as expected from an unbiased model.

\smallskip

It could be the case that the high accuracy of our CNN model was mostly due to it learning the transitivity of the implication relation, as opposed to it discovering patterns in the identities. To clarify this point, it was proposed to train our CNN model either on a random poset or on the equational poset with vertices and labels permuted, and check whether a similar accuracy was achieved. This experiment has not been performed yet.

\smallskip

In any case, since there are around 600k explicit implications from which the rest can be derived by transitivity (2.7\% of the total), if the CNN was learning transitivity it should perform well with a very small training dataset. Accordingly, we trained and assessed a CNN model with a 5\%/5\%/90\% training/validation/test proportion, with a resulting 99.6\% accuracy on test (the process finishing in under 20 minutes). But the high accuracy is maintained with even smaller training datasets, as evidenced in \Cref{table:prediction-accuracy}.
\begin{table}
\centering
\caption{Prediction accuracy as a function of the size of the training set}
\begin{tabular}{cc}
\toprule
  Training/validation/test proportion (\%) & Prediction accuracy (\%) \\\midrule
  60/20/20 & 99.7 \\
  5/5/90 & 99.6 \\
  1/1/98 & 99.3 \\
  0.5/0.5/99 & 98.9 \\
  0.1/0.1/99.8 & 92.2 \\
  \bottomrule
\end{tabular}
\label{table:prediction-accuracy}
\end{table}

\noindent Since these training datasets were significantly smaller than the subset of explicit implications, and were not carefully chosen from the poset extremes but taken randomly, we can conclude that even if the CNN were learning transitivity, that by itself is probably insufficient to explain the high accuracy achieved by the CNN model.

\smallskip

Since sometimes machine learning is announced as a form of data compression, let us now comment on the level of data compression achieved by our CNN model. In \Cref{table:encoding-sizes} we compare the sizes of three different encodings of the implication graph: a) As the simplest approach, we can encode the full implication graph in one file as labelled pairs of equations of the form ($p$, $q$, yes/no). b) On the other extreme, we can encode it as a bit table containing neither the explicit equation expressions nor their numbers, but just a 1/0 label for each point with coordinates ($p$,$q$), together with a table mapping each number to its corresponding equation expression and a small script to recover the file in a). c) Finally, we can encode it in the complete files of the CNN model produced by TensorFlow 2.9.0. In addition, we can either consider these files in their raw form, or we can highly (but losslessly) compress them to achieve a rough comparison of their actual information content; accordingly, in \Cref{table:encoding-sizes} we also include the sizes of the encoding models when compressed with 7-zip LZMA2 ultra compression with a 1536MB dictionary size and a 273 word size.

\begin{table}
\centering
\caption{Sizes of different encodings for the implication graph}
\begin{tabular}{ccc}
  \toprule
  Encoding model & Uncompressed size & Compressed size \\\midrule
  Labelled pairs of equations & \SI{1.5}{\giga\byte} & \SI{9}{\mega\byte} \\
  Bit table & \SI{42}{\mega\byte} & \SI{40}{\kilo\byte} \\
  CNN (99.7\% accuracy) & \SI{1.34}{\mega\byte} & \SI{700}{\kilo\byte} \\
  \bottomrule
\end{tabular}
\label{table:encoding-sizes}
\end{table}

\smallskip

\noindent As we see, the information in the CNN model is more than 13 times less than in the labelled pairs model, while it is more than 18 times that of the bit table model. We also note that the CNN model in its raw form is already quite incompressible, and more than 30 times smaller than the raw bit table.

\smallskip

Lastly, note that our CNN model does not only encode the implication graph up to order 4 (with 0.03\% of noise), but a priori may also be able to predict it for higher orders with a significant accuracy. Thus it could be used to guide and speed up the determination of the implication graph up to order 5, by letting ATPs focus first on the CNN’s predicted status of the studied implication.

\section{User Interfaces}\label{sec:gui-sec}

A number of custom web applications were developed as part of the ETP. While many past Lean formalization projects have primarily relied on the Lean blueprint tool to organize tasks and track progress, the large volume of (transitive) implications tracked by the ETP, along with the research-oriented nature of the project, necessitated the development of custom tools to complement the blueprint tool. These web applications also made information more accessible to project participants and other interested parties, including those unfamiliar with Lean or the custom software developed for the project. The project features four primary interfaces:

\begin{enumerate}
  \item The \textbf{ETP dashboard}\footnote{\url{https://teorth.github.io/equational_theories/dashboard/}} displays the high-level overview of the project: the total number of resolved, conjectured, and unknown implications for the general and finite implication graphs. The dashboard also includes links to other tools, data, and visualizations about the implication graphs.
  \item The \textbf{Equation Explorer}\footnote{\url{https://teorth.github.io/equational_theories/implications/}} is the primary tool to navigate the implication graph. For a given equation, it displays its inbound and outbound implications, as well as other members of its equivalence class. The explorer allows navigating either the general or finite implication graphs. The explorer also features custom commentary for a given equation (when available), serving as a repository for information and links. It also links to Graphiti visualizations and an example of its smallest satisfying magma, if one exists. \Cref{fig:screenshot-equation-explorer} shows an example view of the explorer.
  \item \textbf{Graphiti}\footnote{\url{https://teorth.github.io/equational_theories/graphiti/}} visualizes the implication graph as a Hasse diagram, where downward edges represent subset relationships, and upward edges represent implications. Equivalence classes are collapsed into single nodes for clarity. Graphiti supports search parameters to visualize specific subsets of the graph. It can also display the entire implication graph, though the complete graph is large and challenging to navigate. \Cref{fig:854-like} is an example of a Graphiti visualization.
  \item The \textbf{Finite Magma Explorer}\footnote{\url{https://teorth.github.io/equational_theories/fme/}} tests which equations a given finite magma satisfies or fails to satisfy. Users input finite magmas as Cayley tables. The tool is aware of the finite implication graph, so if an input magma witnesses an unknown refutation, it notifies the user and provides instructions for contributing it to the GitHub repository.
\end{enumerate}

\begin{figure}
  \centering
  \includegraphics[width=1.0\textwidth]{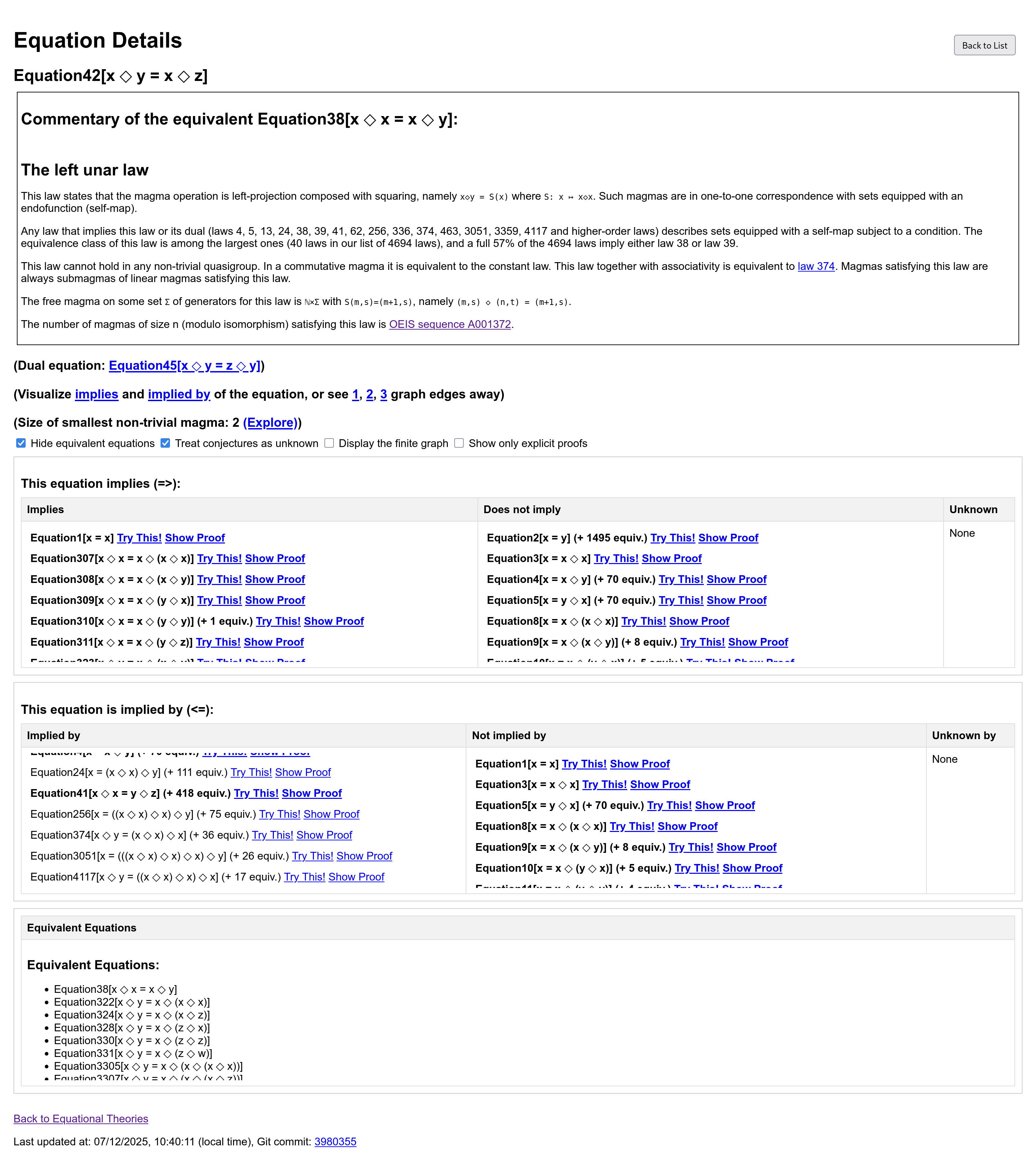}
  \caption{An example of the information displayed by the Equation Explorer for a specific equation.}
  \label{fig:screenshot-equation-explorer}
\end{figure}

The data for these tools is extracted directly from the Lean-formalized proofs in the project's GitHub repository, ensuring it always faithfully reflects the current state of progress. Additionally, the data is automatically updated with each code change using continuous integration (CI), eliminating the need for manual updates.

\section{Data management}

All the data and formalizations generated by the project are placed in the GitHub repository\footnote{\url{https://github.com/teorth/equational_theories}}, while the discussion is almost entirely contained in a dedicated channel on the Lean Zulip\footnote{\url{https://leanprover.zulipchat.com/\#narrow/channel/458659-Equational}}.  The implication graph can be downloaded from Equation Explorer\footnote{\url{https://teorth.github.io/equational_theories/implications/}}, and can also indicate the individual \emph{Lean} theorems required to establish or refute any given implication, although currently we have only formalized a generating set of implications and refutations in \emph{Lean}, rather than the entirety of the implication graph.

\section{Conclusions and future directions}

This project successfully demonstrated that large-scale explorations of a space of mathematical statements (in this case, the implications or non-implications between selected equational laws) can be crowdsourced using modern collaboration platforms and proof assistants.  No single tool or method was able to study the entirety of this space, and many informal proofs generated contained non-trivial errors; but there were multiple techniques that could treat significant portions of the space, and through a collaborative effort combined with the proof validation provided by \emph{Lean}, one could synthesize these partial and fallible contributions into a complete and validated description of the entire implication graph.  While this particular graph was a comparatively simple structure to analyze, we believe that this paradigm could also serve as a model for future projects devoted to exploring more sophisticated large-scale mathematical structures.

Several factors appeared to be helpful in ensuring the success of the project, including the following:
\begin{itemize}
\item \textbf{A clearly stated primary goal, with an end condition and precise numerical metrics to measure partial completion.}  From the outset, there was a specific goal to attain, namely to completely determine and then formalize the implication graph on the original set of $4694$ laws.  Progress towards that goal could be measured by a number of metrics, such as the number of implications that were conjectured but unformalized, or not conjectured at all.   Such metrics allowed participants to see how partial contributions, such as formalizing a certain subset of implications, advanced the project directly towards its primary goal.  This is not to say that all activity was devoted solely towards this primary goal, but it did provide a coherent focus to help guide and motivate other secondary activities.
\item \textbf{A highly modular project}.  It was possible for any given coauthor to work on a small subset of implications and focus on a single proof technique, without needing to understand or rely upon other contributions to the project.  This allowed the work to be both parallelized and decentralized; many contributors launched their own investigations broadly within the framework of the project, without needing centralized approval or coordination.
\item \textbf{Low levels of required mathematical and formal prerequisites}.  The problems considered in the project did not require advanced mathematical knowledge (beyond a general familiarity with abstract algebra), nor a sophisticated understanding of formal proof assistants.  This permitted contributions from a broad spectrum of participants, including those without a graduate mathematical training, as well as mathematicians with no experience in proof formalization.  At a technical level, it also meant that formalization of proofs into \emph{Lean} could be done immediately once certain base definitions (such as \texttt{Magma}) were constructed.  This can be compared for instance with the recent formalization of the Polynomial Freiman--Ruzsa conjecture~\cite{PFR_Tao_Dilles_2023}, in which significant effort was expended in the first few days to settle on a suitable framework to formalize the mathematics of Shannon entropy.  While some more sophisticated formal structures (such as the syntactic description of laws as pairs of words in a \texttt{FreeMagma}) were later introduced in the project, it was relatively straightforward to refactor previously written code to be compatible with these structures as they were incorporated into the project.
\item \textbf{Variable levels of difficulty, and the amenability to partial progress.}  Traditional mathematics projects generally involve a small number of extremely hard problems, with incomplete progress on these problems being difficult to convert into clean partial results.  In contrast, the ETP studied a large number of problems with a very broad range of difficulty, so that even if a given proof strategy did not work for a given implication, it could be the case that there was some class of easier implications for which the strategy was successful.  This allowed for a means to validate such ideas, and allowed the project to build up a useful and diverse toolbox of proof techniques which became increasingly necessary to handle the final and most difficult implications in the project.  It also created a dynamic in which the project initially focused on easy techniques to resolve a significant fraction of the implications, gradually transitioning into more sophisticated methods that focused on a much smaller number of outstanding implications that had proven resistant (or even ``immune'') to all easier approaches.
\item \textbf{Centralized and standardized platforms for discussion, project management, and validation.}  While the project was decentralized at the level of the participant, there was a centralized location (a channel\footnote{\url{https://leanprover.zulipchat.com/\#narrow/channel/458659-Equational}} on the Lean Zulip) to discuss all aspects of the project, as well as a centralized repository\footnote{\url{https://github.com/teorth/equational_theories}} to track all contributions and outstanding issues, a centralized blueprint\footnote{\url{https://teorth.github.io/equational_theories/blueprint/}} to describe technical details of proofs to be formalized, and a single formal language (\emph{Lean}) to validate all contributions. A significant portion of the activity in the early stages of the project was devoted to setting out the standards and workflows for handling both the discussion and the contributions, in particular setting up a contributions page\footnote{\url{https://github.com/teorth/equational_theories/blob/main/CONTRIBUTING.md}} and adopting a code of conduct\footnote{\url{https://github.com/teorth/equational_theories/blob/main/CODE_OF_CONDUCT.md}}.  This gave some structure and predictability to what might otherwise be a chaotic effort.
\item \textbf{Development of custom visualization tools.}  As discussed in \Cref{sec:gui-sec}, several tools were developed (in part with AI assistance) to help visualize and navigate the implication graph while it was in a partial stage of development, allowing for participants to independently identify problems to work on, and to validate and use the contributions of other participants even before they were fully formalized.  For instance, a participant could propose a finite counterexample to an implication by posting a link to the magma in \emph{Finite Magma Explorer}, allowing for immediate validation of the counterexample, or use \emph{Equation Explorer} or \emph{Graphiti} to observe some interesting phenomenon in the implication graph that other participants could reproduce and study.
\item \textbf{Applicability of existing software tools.}  As described in \Cref{automated-sec}, many of the implications in the ETP were amenable to application of ``off-the-shelf'' automated theorem provers (ATPs); while some trial and error was needed to determine good choices of parameters, these tools could largely be applied directly to the project without extensive customization.  (However, the later transcription of ATP output into Lean was sometimes non-trivial.)
\item \textbf{Receptiveness to new techniques and tools.}  Crucially, the methods used to make progress on the project were not specified in advance, and contributions from participants with new ideas, techniques, or software tools that were not initially anticipated were welcomed.  For instance, the theory of canonizers (\Cref{canon-sec}) was not initially known to the first project participants, but was brought to the attention of the project by a later contributor.  Conversely, while there were hopes expressed early in the project that modern large language models (LLMs) could automatically generate many of the proofs required, it turned out in practice that other forms of automation, particularly ATPs, were significantly more effective at this task (at least if one restricted to publicly available LLMs), and the project largely moved away from the use of such LLMs (other than to help create the code for the visualization tools).
\looseness=-1
\end{itemize}

There are several mathematical and computational questions that could potentially be addressed in future work building upon the outcomes of ETP\@.
Here is a list of some possible such future directions.
\begin{enumerate}
  \item Does the law $\x \formaleq \y \op (\x \op ((\y \op \x) \op \y))$ \eqref{eq677} imply $\x \formaleq ((\x \op \x) \op \x) \op \x$ \eqref{eq255} for finite magmas, i.e., $\Eq{677} \modelsfin \Eq{255}$? This is the last remaining implication (up to duality) for finite magmas to be resolved.  A number of partial results on this problem may be found at \url{https://teorth.github.io/equational_theories/blueprint/677-chapter.html}.

  \item The ETP focused on determining relations $\E \models \E'$ between one law and another.  Could the same methods also systematically determine more complex logical relations, such as $\E_1 \wedge \E_2 \models \E_3$, for all laws $\E_1,\E_2,\E_3$ in a specified set?  This includes the question of implications between equational laws in semigroups (associative magmas).  One could also consider implications involving magma properties that are not equational laws, such as cancellability or existence of a unit element.

  \item Call an implication $\E_1 \models \E_2$ ``irreducible'' if there is no equational law $\E$ with $\E_1 \models \E \models \E_2$, other than those laws equivalent to either $\E_1$ or $\E_2$.  For instance, $\Eq{2} \models \Eq{4}$ is irreducible, since $\Eq{4}$ implies any law of the form $w \formaleq w'$ where the left-most variable of $w$ matches the left-most variable of $w'$. On the other hand, $\Eq{4}$ in conjunction with any law not of that form yields $\Eq{2}$.  Similar \emph{ad hoc} arguments can produce other irreducible implications, e.g., $\Eq{2} \models \Eq{n}$ for $n = 5, 895, 26302$.  Could one replicate the ETP to classify all stable implications among the same \num{4694} equations studied in this project?

  \item For a given finite non-implication $\E_1 \nmodelsfin \E_2$, are there bounds on the proportion of variable assignments for which~$\E_2$ holds, similarly to how in a finite group either all elements square to the neutral element, or at most $3/4$ of them do?
\end{enumerate}
Some other directions do not concern implications between laws, but may benefit from data generated by the ETP\@.
\begin{enumerate}[resume*]
  \item Does the law $\x \formaleq \y \op (\y \op (\y \op (\x \op (\z \op \y))))$ \eqref{eq5093} have any infinite models? In \cite{Kisielewicz2} it was shown that it has no non-trivial finite models, but the infinite model case was left as an open question.  A partial classification of laws of order~$5$ with infinite models but no finite models is given at \url{https://teorth.github.io/equational_theories/blueprint/order-5-austin-laws.html}.

  \item A key feature of finite magmas $\Magma$ is that they are surjunctive, in the sense that any definable map from $M$ to itself that is injective, is also surjective (or vice versa), where ``definable'' is with respect to the language of magmas.  Are there equational theories that admit surjunctive models, but yet do not have any non-trivial finite models?

  \item Are all finite weak central groupoids, namely magmas obeying $\x \formaleq (\y \op \x) \op (\x \op (\z \op \y))$ \eqref{eq1485}, necessarily of size $n^2$ or $2n^2$?  More generally, what is the spectrum of each law or conjunction of laws, and what are the possible asymptotics for the fine spectrum in terms of model size?

  \item How ``stable'' is a given law~$\E$?  For instance, if a finite magma satisfies a law~$\E$ some proportion $1-\eps$ of the time, with $\eps$~small, can the magma be perturbed into one that satisfies~$\E$ exactly?  Related to this is the question of whether a law~$\E$ is ``rigid'' or ``mutable'': is it possible to add an element or to make a small number of modifications to a magma satisfying~$\E$, in a way that still preserves~$\E$?  Such properties helped suggest whether certain magma construction techniques, such as modifying a base magma, were likely to be successful.

  \item For each law, can its free magma with one or more generators be described explicitly?

  \item Which laws admit an interesting theory of smooth magmas, analogous to Lie groups?
\end{enumerate}

\subsection{Miscellaneous remarks}

It is possible that the timing in which certain proof methods were introduced into the project created some opportunity costs.  For instance, by deploying automated theorem provers at an early stage, we might have settled some implications that had more interesting human-readable proofs that we missed.  Similarly, we developed some sophisticated theory for the equation $\Eq{854}$, such as \Cref{unique-factorization}, that is now superseded by finite counterexamples; but had the finite counterexamples been discovered first, we would not have found the theoretical arguments.  It may be productive for future work to revisit some portions of the implication graph and locate alternate proofs and methods.

\section*{Acknowledgments}

We are grateful to the many additional participants in the Equational Theories Project for their
numerous comments and encouragement, with particular thanks to Stanley Burris, Edward van de Meent and David Michael Roberts. We warmly thank Michael Kinyon for generously sharing his expertise with \emph{Prover9-Mace4}, and we are likewise grateful to Laura Kovács, Márton Hajdu, Martin Suda, and Michael Rawson of the \emph{Vampire} development team for their helpful explanations regarding \emph{Vampire}'s options and functionality. We gratefully acknowledge the computational resources provided by the SINGACOM research group at IMUVa, the Mathematics Research Institute of the University of Valladolid. Additionally, we note that Shreyas Srinivas is a doctoral student at the Saarbr\"{u}cken Graduate School for Computer Science.

\appendix
\crefalias{section}{appendix}
\raggedbottom
\section{Numbering system}\label{numbering-app}

In this section we record the numbering conventions we use for equational laws.

For this formal definition we use the natural numbers $0,1,2,\dots$ to represent and order indeterminate variables; however, in the main text, we use the symbols $\x,\y,\z,\w,\uu,\vv,\mathrm{r},\mathrm{s},\mathrm{t}$ instead (and do not consider any laws with more than eight variables).

To define the ordering we use on equational laws, we first consider the case where there is a single indeterminate $\ast$.
We place a well-ordering on words $w,w'$ with a single indeterminate $\ast$ by declaring $w > w'$ if one of the following holds:
\begin{itemize}
    \item $w$ has a larger order than $w'$.
    \item $w = w_1 \op w_2$ and $w' = w'_1 \op w'_2$ have the same order $n \geq 1$ with $w_1 > w'_1$.
    \item $w = w_1 \op w_2$ and $w' = w'_1 \op w'_2$ have the same order $n \geq 1$ with $w_1 = w'_1$ and $w_2 > w'_2$.
\end{itemize}
Thus
\begin{align*}
  \ast < \ast \op \ast &< \ast \op (\ast \op \ast) < (\ast \op \ast) \op \ast \\
  & < \ast \op (\ast \op (\ast \op \ast)) < \ast \op ((\ast \op \ast) \op \ast) < \dots
\end{align*}

We similarly place a well-ordering on equational laws $w_1 \formaleq w_2$ with a single indeterminate $\ast$ by declaring $w_1 \formaleq w_2 > w'_1 \formaleq w'_2$ if one of the following holds:
\begin{itemize}
\item  $w_1 \formaleq w_2$ has a larger order than $w'_1 \formaleq w'_2$.
\item If $w_1 \formaleq w_2$ has the same order as $w'_1 \formaleq w'_2$, and $w_1 > w'_1$.
\item If $w_1 \formaleq w_2$ has the same order as $w'_1 \formaleq w'_2$, $w_1 = w'_1$, and $w_2 > w'_2$.
\end{itemize}
Thus for instance
$$ (\ast \op \ast \formaleq \ast \op (\ast \op \ast)) < (\ast \op \ast \formaleq (\ast \op \ast) \op \ast).$$

Finally for equational laws with alphabet $\x,\y,\z,\w,\uu,\vv,\mathrm{r},\mathrm{s},\mathrm{t}$, define the \emph{shape} of that law to be the law formed by replacing all indeterminates with $\ast$; for instance, the shape of $\x \op (\y \op \z) = (\x \op \y) \op \z$ \eqref{eq4512}, is $\ast \op (\ast \op \ast) \formaleq (\ast \op \ast) \op \ast$.  We then place a well-ordering $w_1 \formaleq w_2$ with indeterminates $\x,\y,\z,\w,\uu,\vv,\mathrm{r},\mathrm{s},\mathrm{t}$ by declaring $w_1 \formaleq w_2 > w'_1 \formaleq w'_2$ if one of the following holds:
\begin{itemize}
\item The shape of $w_1 \formaleq w_2$ is greater than the shape of $w'_1 \formaleq w'_2$.
\item $w_1 \formaleq w_2$ and $w'_1 \formaleq w'_2$ have the same shape, and the string of variables appearing in $w_1 \formaleq w_2$ is lower in the lexicographical ordering (using $\x < \y < \z < \w < \uu < \vv < \mathrm{r} < \mathrm{s} < \mathrm{t}$) than the corresponding string for $w'_1 \formaleq w'_2$.
\end{itemize}
Thus for instance any law of shape $\ast \op \ast \formaleq \ast \op (\ast \op \ast)$ is lower than any law of shape
$\ast \op \ast \formaleq (\ast \op \ast) \op \ast$.  Among the laws of shape $\ast \op \ast \formaleq \ast \op (\ast \op \ast)$, the lowest is $\x \op \x \formaleq \x \op (\x \op \x)$, which is less than (say) $\x \op \x \formaleq \y \op (\y \op \y)$, which is in turn less than $\x \op \y \formaleq \x \op (\x \op \x)$.

We say that two equational laws are \emph{definitionally equivalent}\footnote{This can be distinguished from the weaker notion of \emph{propositional equivalence} (mutual entailment) used in the rest of the paper.} if one can be obtained from another by some combination of relabeling the variables and applying the symmetric law $w_1 \formaleq w_2 \iff w_2 \formaleq w_1$.  For instance, $(0 \op 1) \op 2 \formaleq 1$ is definitionally equivalent to $0 \formaleq (1 \op 0) \op 2$.  We then replace every equational law with their minimal element in their definitional equivalence class, which can be viewed as the \emph{normal form} for that law; for instance, the normal form of $(0 \op 1) \op 2 \formaleq 1$ would be $0 \formaleq (1 \op 0) \op 2$.  Finally, we eliminate any law of the form $w \formaleq w$ other than $0 \formaleq 0$.  We then number the remaining equations $\Eq{1}, \Eq{2}, \dots$.  For instance, $\Eq{1}$ is the trivial law $0 \formaleq 0$, $\Eq{2}$ is the singleton law $0 \formaleq 1$, $\Eq{3}$ is the idempotent law $0 \formaleq 0 \op 0$, and so forth.  Lists and code for generating these equations, or the equation number attached to a given equation, can be found in the ETP repository.

The number of equations in this list of order $n=0,1,2,\dots$ is given by
$$ 2, 5, 39, 364, 4284, 57882, 888365, \dots$$
(\url{https://oeis.org/A376640}).  The number can be computed to be
$$ C_{n+1} B_{n+2}/2$$
if $n$ is odd, $2$ if $n=0$, and
$$ (C_{n+1} B_{n+2}+ C_{n/2}(2D_{n+2}-B_{n+2}))/2 - C_{n/2} B_{n/2+1}$$
if $n > 2$ is even, where $C_n, B_n$ are the Catalan and Bell numbers, and $D_n$ is the number of partitions of $[n]$ up to reflection, which for $n=0,1,2,\dots$ is
$$ 1, 1, 2, 4, 11, 32, 117, \dots$$
(\url{https://oeis.org/A103293}).  A proof of this claim can be found in the ETP blueprint.  In particular, there are $4694$ equations of order at most $4$.

Below we record some specific equations appearing in this paper, using the alphabet $\x$, $\y$, $\z$, $\w$ in place of $0$, $1$, $2$, $3$ for readability.
\begingroup\allowdisplaybreaks
\begin{align}
    \x &\formaleq \x & \hbox{(Trivial law)} \label{eq1}\tag{E1} \\
    \x &\formaleq \y & \hbox{(Singleton law)} \label{eq2}\tag{E2} \\
    \x &\formaleq \x \op \x & \hbox{(Idempotent law)} \label{eq3}\tag{E3} \\
    \x &\formaleq \x \op \y & \hbox{(Left-absorptive law)} \label{eq4}\tag{E4} \\
    \x &\formaleq \y \op \x & \hbox{(Right-absorptive law)} \label{eq5}\tag{E5} \\
    \x &\formaleq \x \op (\y \op \x) \label{eq10}\tag{E10} \\
    \x &\formaleq \x \op (\y \op \y) & \hbox{(Right-unit squares law)} \label{eq11}\tag{E11} \\
    \x &\formaleq \y \op (\x \op \y) & \hbox{(Semi-symmetric quasigroups)} \label{eq14}\tag{E14} \\
    \x &\formaleq (\x \op \x) \op \x \label{eq23}\tag{E23} \\
    \x \op \x &\formaleq \y \op \y & \hbox{(Unipotence law)} \label{eq40}\tag{E40} \\
    \x \op \x &\formaleq \y \op \z & \hbox{(Equivalent to~\eqref{eq46})} \label{eq41}\tag{E41} \\
    \x \op \y &\formaleq \y \op \x & \hbox{(Commutative law)} \label{eq43}\tag{E43} \\
    \x \op \y &\formaleq \z \op \w & \hbox{(Constant law)} \label{eq46}\tag{E46} \\
    \x &\formaleq \x \op (\x \op (\x \op \x)) \label{eq47}\tag{E47} \\
    \x &\formaleq \y \op (\y \op (\x \op \y))  \label{eq73}\tag{E73} \\
    \x &\formaleq (\x \op \x) \op (\x \op \x) \label{eq151}\tag{E151} \\
    \x &\formaleq (\y \op \x) \op (\x \op \z) & \hbox{(Central groupoid law)} \label{eq168}\tag{E168} \\
    \x &\formaleq (\x \op (\x \op \y)) \op \y \label{eq206}\tag{E206} \\
    \x &\formaleq ((\x \op \x) \op \x) \op \x \label{eq255}\tag{E255} \\
    \x \op \y &\formaleq \x \op (\y \op \z) & \hbox{(Right reduction law)} \label{eq327}\tag{E327} \\
    \x \op \y &\formaleq (\x \op \y) \op \y & \hbox{(Right idempotence law)} \label{eq378}\tag{E378} \\
    \x \op \y &\formaleq (\z \op \x) \op \y & \hbox{(Left reduction law)} \label{eq395}\tag{E395} \\
    \x &\formaleq \x \op (\x \op (\x \op (\y \op \x))) \label{eq413}\tag{E413} \\
    \x &\formaleq \x \op (\y \op (\z \op (\y \op \x))) \label{eq450}\tag{E450} \\
    \x &\formaleq \y \op (\x \op (\z \op (\z \op \y))) & \hbox{(Totally symmetric quasigroups)} \label{eq492}\tag{E492} \\
    \x &\formaleq \y \op (\z \op (\x \op (\y \op \z))) & \hbox{(Tarski's axiom)} \label{eq543}\tag{E543} \\
    \x &\formaleq \x \op (\y \op ((\z \op \x) \op \y)) & \hbox{(Non-trivially equivalent to~\eqref{eq4})} \label{eq650}\tag{E650} \\
    \x &\formaleq \y \op (\x \op ((\y \op \x) \op \y)) & \hbox{(Last open implication)} \label{eq677}\tag{E677} \\
    \x &\formaleq \x \op ((\x \op \x) \op (\x \op \x)) \label{eq817}\tag{E817} \\
    \x &\formaleq \x \op ((\y \op \z) \op (\x \op \z)) \label{eq854}\tag{E854}\\
    \x &\formaleq \x \op ((\y \op (\y \op \x)) \op \x) \label{eq1045}\tag{E1045} \\
    \x &\formaleq \x \op ((\y \op (\z \op \x)) \op \x) \label{eq1055}\tag{E1055} \\
    \x &\formaleq \y \op ((\y \op (\x \op \x)) \op \y) \label{eq1110}\tag{E1110} \\
    \x &\formaleq \y \op ((\y \op (\x \op \z)) \op \z) \label{eq1117}\tag{E1117} \\
    \x &\formaleq \y \op (((\x \op \y) \op \x) \op \y) \label{eq1286}\tag{E1286} \\
    \x &\formaleq \y \op (((\y \op \y) \op \x) \op \y) \label{eq1323}\tag{E1323} \\
    \x &\formaleq (\y \op \x) \op (\x \op (\z \op \y)) & \hbox{(Weak central groupoids)}\label{eq1485}\tag{E1485} \\
    \x &\formaleq (\y \op \y) \op (\x \op (\y \op \x)) \label{eq1518}\tag{E1518} \\
    \x &\formaleq (\y \op \z) \op (\y \op (\x \op \z)) & \hbox{(Boolean groups)} \label{eq1571}\tag{E1571} \\
    \x &\formaleq (\x \op \x) \op ((\x \op \x) \op \x) \label{eq1629}\tag{E1629} \\
    \x &\formaleq (\x \op \y) \op ((\x \op \y) \op \y) \label{eq1648}\tag{E1648} \\
    \x &\formaleq (\x \op \y) \op ((\y \op \y) \op \z) \label{eq1659}\tag{E1659} \\
    \x &\formaleq (\y \op \x) \op ((\x \op \z) \op \z) & \hbox{(Equivalent to~\eqref{eq2})} \label{eq1689}\tag{E1689} \\
    \x &\formaleq (\y \op \y) \op ((\y \op \x) \op \y) \label{eq1729}\tag{E1729} \\
    \x &\formaleq (\y \op (\x \op (\y \op \x))) \op \y \label{eq2301}\tag{E2301} \\
        \x &\formaleq (\x \op ((\x \op \x) \op \x)) \op \x \label{eq2441}\tag{E2441} \\
        \x &\formaleq ((\y \op \y) \op (\y \op \x)) \op \y \label{eq2744}\tag{E2744} \\
        \x &\formaleq ((\y \op (\x \op \y)) \op \x) \op \y & \hbox{(Dual of \eqref{eq677})} \label{eq2910}\tag{E2910} \\
        \x \op \y &\formaleq \x \op (\y \op (\x \op \y)) \label{eq3316}\tag{E3316} \\
        \x \op \y &\formaleq \x \op ((\y \op \y) \op \z) \label{eq3523}\tag{E3523} \\
        \x \op \y &\formaleq (\x \op \z) \op (\y \op \z) \label{eq3737}\tag{E3737} \\
        \x \op \y &\formaleq (\x \op (\y \op \x)) \op \y \label{eq3925}\tag{E3925} \\
        \x \op (\y \op \x) &\formaleq \x \op (\y \op \z) \label{eq4315}\tag{E4315} \\
        \x \op (\x \op \x) &\formaleq (\x \op \x) \op \x & \hbox{(Cube-associativity law)} \label{eq4380}\tag{E4380} \\
        \x \op (\y \op \y) &\formaleq (\y \op \y) \op \x & \hbox{(Central squares law)} \label{eq4482}\tag{E4482} \\
        \x \op (\y \op \z) &\formaleq (\x \op \y) \op \z & \hbox{(Associative law)} \label{eq4512}\tag{E4512} \\
        \x \op (\y \op \z) &\formaleq (\y \op \z) \op \x & \hbox{(Central products law)} \label{eq4531}\tag{E4531} \\
        \x &\formaleq \y \op (\y \op (\y \op (\x \op (\z \op \y)))) \label{eq5093}\tag{E5093} \\
        \x &\formaleq \y \op (\z \op ((\y \op \x) \op (\z \op (\y \op \z)))) & \hbox{(Eisenstein modules)} \label{eq85914} \tag{E85914} \\
        \x &\formaleq \y \op (\z \op ((\y \op \z) \op (\w \op (\x \op \w)))) & \hbox{(Gaussian modules)} \label{eq86082} \tag{E86082} \\
        \x &\formaleq (\y \op ((\x \op \y) \op \y)) \op (\x \op (\z \op \y)) & \hbox{(Sheffer stroke)} \label{eq345169}\tag{E345169}
\end{align}
We also list some order-$8$ characterizations of group division relevant for \Cref{higman-neumann}.
\begin{align}
        \x &\formaleq \y \op ((((\x \op \x) \op \x) \op \z) \op (((\x \op \x) \op \y) \op \z)) & \hbox{(McCune law)} \label{eq42302852}\tag{E42302852} \\
        \x &\formaleq \y \op ((((\x \op \x) \op \x) \op \z) \op (((\y \op \y) \op \y) \op \z)) & \label{eq42302946}\tag{E42302946} \\
        \x &\formaleq \y \op ((((\y \op \y) \op \x) \op \z) \op (((\y \op \y) \op \y) \op \z)) & \hbox{(Higman--Neumann law)} \label{eq42323216}\tag{E42323216} \\
        \x &\formaleq (\y \op \y) \op (\y \op ((\x \op \z) \op (((\x \op \x) \op \y) \op \z))) & \hbox{(in finite magmas)} \label{eq67953597}\tag{E67953597} \\
        \x &\formaleq ((\y \op \y) \op \y) \op ((((\x \op \x) \op \x) \op \z) \op (\y \op \z)) & \label{eq89176740}\tag{E89176740} \\
        \x &\formaleq ((\y \op \y) \op ((\x \op \z) \op \x)) \op ((\z \op \w) \op (\x \op \w)) & \label{eq102744082}\tag{E102744082} \\
        \x &\formaleq ((\y \op \y) \op (\y \op (\x \op (((\y \op \y) \op \y) \op \z)))) \op \z & \hbox{(McCune law)} \label{eq147976245}\tag{E147976245}
\end{align}
\endgroup

\section{Author contributions}

In a \href{https://github.com/teorth/equational_theories/blob/main/paper/contributions.md}{companion document} to this paper, the contributions of each author of this paper to the ETP are described, following the standard CRediT categories\footnote{\url{https://credit.niso.org/}}. Below are the affiliations and grant acknowledgments of individual participants.

\begin{itemize}\raggedright
    \item Matthew Bolan: University of Toronto, matthew.bolan@mail.utoronto.ca. Supported by an Ontario Graduate Scholarship.
    \item Joachim Breitner: Lean FRO, mail@joachim-breitner.de, ORCID 0000-0003-3753-6821
    \item Jose Brox: IMUVA-Mathematics Research Institute, Universidad de Valladolid, josebrox@uva.es, ORCID 0000-0001-9822-5838. Supported by a postdoctoral fellowship “Convocatoria 2021” funded by Universidad de Valladolid, and partially supported by grant PID2022-137283NB-C22 funded by MCIN/AEI/10.13039/501100011033 and ERDF “A way of making Europe”
    \item Nicholas Carlini: Unaffiliated, nicholas@carlini.com
    \item Mario Carneiro: Chalmers University of Technology \& Gothenburg University, Sweden, marioc@chalmers.se, ORCID 0000-0002-0470-5249
    \item Floris van Doorn: University of Bonn, vdoorn@math.uni-bonn.de, ORCID 0000-0003-2899-8565
    \item Martin Dvorak: Institute of Science and Technology Austria, martin.dvorak@matfyz.cz, ORCID 0000-0001-5293-214X
    \item Andr\'es Goens: TU Darmstadt, andres.goens@tu-darmstadt.de, ORCID 0000-0002-0409-1363
    \item Aaron Hill: Unaffiliated, aa1ronham@gmail.com, ORCID 0009-0007-0827-1277
    \item Harald Husum: Intelecy, harald.husum@intelecy.com, ORCID 0009-0003-0634-7435
    \item Hern\'an Ibarra Mejia: Unaffiliated, hernan@ibarramejia.com
    \item Zoltan A. Kocsis: University of New South Wales, z.kocsis@unsw.edu.au
    \item Bruno Le Floch: CNRS and Laboratoire de Physique Th\'eorique et Hautes \'Energies, Sorbonne Universit\'e, blefloch@lpthe.jussieu.fr, ORCID 0000-0002-3965-9705. Partially supported by the grants Projet-ANR-23-CE40-0010 and HORIZON-MSCA-2022-SE/101131233
    \item Amir Livne Bar-on: Unaffiliated, amir.livne.baron@gmail.com
    \item Lorenzo Luccioli: University of Bologna, lorenzo.luccioli2@unibo.it, ORCID 0009-0009-2256-2833
    \item Douglas McNeil: Unaffiliated, dsm054@gmail.com, ORCID 0009-0006-4662-0469
    \item Alex Meiburg: Perimeter Institute for Theoretical Physics / University of Waterloo Institute for Quantum Computing, teqtp@ohaithe.re, ORCID 0000-0002-4506-9146
    \item Pietro Monticone: University of Trento, pietro.monticone@studenti.unitn.it, ORCID 0000-0002-2731-9623
    \item Pace P. Nielsen: Department of Mathematics, Brigham Young University, pace@math.byu.edu
    \item Emmanuel Osalotioman Osazuwa: University of Benin, emmanuel.osazuwa@physci.uniben.edu, ORCID 0009-0003-1415-8263
    \item Giovanni Paolini: University of Bologna, g.paolini@unibo.it, ORCID 0000-0002-3964-9101
    \item Marco Petracci: University of Bologna, marco.petracci@studio.unibo.it
    \item Bernhard Reinke: Aix-Marseille Université, bernhard.reinke@univ-amu.fr, ORCID 0000-0001-9024-2449
    \item David Renshaw: Institute for Computer-Aided Reasoning in Mathematics, renshaw@icarm.io, ORCID 0000-0002-9987-9144
    \item Marcus Rossel: Barkhausen Institut, marcus.rossel@barkhauseninstitut.org, ORCID 0009-0001-3567-6890
    \item Cody Roux: Amazon Web Services, cody.roux@gmail.com, ORCID 0009-0004-5304-037X
    \item J\'er\'emy Scanvic: Laboratoire de Physique, École Normale Supérieure de Lyon, jeremy.scanvic@ens-lyon.fr, ORCID 0009-0003-6117-0492
    \item Shreyas Srinivas: CISPA Helmholtz Center for Information Security, Saarbr\"{u}cken, Germany, shreyas.srinivas@cispa.de, ORCID 0000-0002-3993-1596
    \item Anand Rao Tadipatri: University of Cambridge, art71@cam.ac.uk, ORCID 0009-0007-0057-4169
    \item Terence Tao: Department of Mathematics, UCLA, tao@math.ucla.edu. Supported by the James and Carol Collins Chair, the Mathematical Analysis \& Application Research Fund, and by NSF grants DMS-2347850, and is particularly grateful to recent donors to the Research Fund, ORCID 0000-0002-0140-7641
    \item Vlad Tsyrklevich: Unaffiliated, vlad@tsyrklevi.ch, ORCID 0009-0009-9511-5460
    \item Fernando Vaquerizo-Villar: Biomedical Engineering Group, University of Valladolid, and CIBER de Bioingeniería, Biomateriales y Nanomedicina, Instituto de Salud Carlos III, fernando.vaquerizo@uva.es, ORCID 0000-0002-5898-2006
    \item Daniel Weber: Ben-Gurion University of the Negev, weberdan@post.bgu.ac.il, ORCID 0009-0008-4615-6445
    \item Fan Zheng: Unaffiliated, fanzheng1729@outlook.com
\end{itemize}

\bibliographystyle{plainurl}
\bibliography{references}

\end{document}